\documentclass[12pt]{article}
\usepackage[left=0.8in,top=0.8in,right=0.8in,bottom=0.8in]{geometry}
\usepackage{graphicx}
\usepackage{amsmath}
\usepackage{amssymb}
\usepackage{amsthm}
\usepackage{bm}
\usepackage{stmaryrd}
\usepackage{mathrsfs}
\usepackage{array}
\usepackage{multirow}

\SetSymbolFont{stmry}{bold}{U}{stmry}{m}{n}

\newtheorem{proposition}{Proposition}
\newtheorem{remark}{Remark}
\newcolumntype{P}[1]{>{\centering\arraybackslash}p{#1}}

\begin{document}
\title{An energy-stable mixed formulation for isogeometric analysis of incompressible hyper-elastodynamics}
\author{Ju Liu$^{\textup{ a}}$, Alison L. Marsden$^{\textup{ a}}$, Zhen Tao$^{\textup{ b}}$ \\
$^a$ \textit{\small Department of Pediatrics (Cardiology), Bioengineering,}\\
\textit{\small and Institute for Computational and Mathematical Engineering, Stanford University,}\\
\textit{\small Clark Center E1.3, 318 Campus Drive, Stanford, CA 94305, USA}\\
$^b$ \textit{\small Institute for Computational Engineering and Sciences, The University of Texas at Austin,}\\
\textit{\small 201 East 24th Street, 1 University Station C0200, Austin, TX 78712, USA}\\
\small \textit{E-mail address:} liuju@stanford.edu, amarsden@stanford.edu, taozhen.cn@gmail.com
}
\date{}
\maketitle

\section*{Abstract}
We develop a mixed formulation for incompressible hyper-elastodynamics based on a continuum modeling framework recently developed in \cite{Liu2018} and smooth generalizations of the Taylor-Hood element based on non-uniform rational B-splines (NURBS). This continuum formulation draws a link between computational fluid dynamics and computational solid dynamics. This link inspires an energy stability estimate for the spatial discretization, which favorably distinguishes the formulation from the conventional mixed formulations for finite elasticity. The inf-sup condition is utilized to provide a bound for the pressure field. The generalized-$\alpha$ method is applied for temporal discretization, and a nested block preconditioner is invoked for the solution procedure. The inf-sup stability for different pairs of NURBS elements is elucidated through numerical assessment. The convergence rate of the proposed formulation with various combinations of mixed elements is examined by the manufactured solution method. The numerical scheme is also examined under compressive and tensile loads for isotropic and anisotropic hyperelastic materials. Finally, a suite of dynamic problems is numerically studied to corroborate the stability and conservation properties.

\vspace{5mm}

\noindent \textbf{Keywords:} Incompressible elasticity, Mixed formulation, Inf-sup condition, Energy stability, Generalized-$\alpha$ method, Anisotropic arterial wall model

\section{Introduction}
\label{sec:introduction}

\subsection{Motivation and literature survey}
Over the past few decades, significant progress has been achieved in the finite element modeling of solid mechanics problems. A central topic is to devise a numerical scheme that works well in the incompressible limit. Under the small-strain assumption, this issue is well-understood, and it boils down to interpolating the displacement and pressure with elements that satisfy the Ladyzhenskaya-Babu\v{s}ka-Brezzi (LBB) or the inf-sup condition \cite{Boffi2013}. Under large strains, most materials exhibit volume-preserving behavior, which makes it imperative to appropriately handle the incompressibility constraint. This issue is particularly relevant for modeling biological tissues, which are largely incompressible due to their high water content. In fact, the nonlinear nature of large strain analysis, together with the kinematic constraint, makes the numerical analysis of incompressible materials quite challenging. Classical treatments of this class of problems include the $\bar{\textup{F} }$-projection method \cite{SouzaNeto1996,Elguedj2008,Hughes1975}, the enhanced assumed strain method (EAS) \cite{Simo1990,Simo1992}, and the mixed $u/p$ formulation \cite{Sussman1987}.

The $\bar{\textup{F} }$-projection and EAS methods share some similarities. Both methods are developed based on heuristic splits of the deformation gradient; the geometrically linear versions of the two methods are linked with the mixed finite element method \cite{Boffi2013,Hughes1977}. Nevertheless, there are drawbacks of both. For the $\bar{\textup{F} }$-projection method, its implementation requires a nonlocal matrix inversion if the projection is onto a continuous finite element space. The EAS method relies crucially on a static condensation procedure to maintain the pure displacement code structure. The penalty nature of the pure displacement formulation inevitably induces an ill-conditioned stiffness matrix, which imposes a severe constraint on the choice of linear solvers. It has long been known that both methods suffer from mesh instability or the hourglass mode \cite{Wriggers1996} and hence necessitate further refinements to numerical technologies for the hourglass control.

The mixed $u/p$ formulation introduces a \textit{pressure-like} variable as the Lagrange multiplier for the incompressibility constraint in the strain energy \cite{Sussman1987}. The resulting scheme necessitates interpolating the displacement and pressure fields independently. Performing a linearization of this formulation provides a justification for the use of inf-sup stable elements \cite{Auricchio2005}. Yet, for nonlinear problems, linearized stability is often insufficient to guarantee nonlinear stability \cite{Holm1985}. It remains unclear whether there is any a priori nonlinear stability estimate for the mixed $u/p$ formulation.

In the meantime, the stabilized finite element method, as a technique initially developed for computational fluid dynamics, has been extended to solid mechanics based on various variational formulations \cite{Abboud2018,Bonet2015a,Klaas1999,Liu2018,Rossi2016,Scovazzi2016,Zeng2017}. Using the stabilized formulation allows one to interpolate physical quantities with equal-order interpolations. This feature gives practitioners maximum flexibility in mesh generation and numerical implementation, and allows for low-order elements which are more robust than their higher-order counterparts. Equal-order interpolations always give an optimal constraint ratio \cite[Chapter 4]{Hughes1987}, which may be regarded as another appealing feature for incompressible elasticity. The stabilization term can be interpreted as a subgrid scale model within the variational multiscale framework \cite{Hughes1995,Hughes1986a,Liu2018,Rossi2016}. It has been observed that for inelastic models, the subgrid scale model requires careful design \cite{Zeng2017}. This observation partly motivates this work, in which we aim to design a stable numerical formulation for incompressible hyperelasticity that does not rely on subgrid scale numerical models with tunable parameters.

\subsection{Overview of the proposed method}
It is well-known that a finite element scheme is based on the formulation (i.e., the variational principle) and the discrete function spaces (i.e., the elements). Both components need to be properly accounted for in the design of numerical schemes. In this work, we introduce a mixed variational formulation different from the existing mixed $u/p$ formulation \cite{Sussman1987}. In that formulation, the momentum balance equations are coupled with an algebraic equation of state, which relates the pressure with $J$, the determinant of the deformation gradient \cite[Chapter 8]{Holzapfel2000}. In the incompressible limit, this relation reduces to $J=1$. In the new mixed formulation, the momentum equations are coupled with the differential mass equation written in terms of the pressure primitive variable set. The volumetric behavior is reflected through the so-called isothermal compressibility factor \cite{Liu2018}. In the incompressible limit, this term approaches zero, and the mass equation degenerates to the divergence-free constraint for the velocity field. Although $J=1$ is equivalent to the divergence-free constraint for the velocity field at the continuum level, they lead to different schemes at the discrete level. Based on the new mixed formulation, an a priori energy stability estimate can be obtained, and the inf-sup condition leads to a bound for the pressure solution. We regard these estimates as critical numerical properties embedded in the formulation that guarantee reliable results.

It should be pointed out that there are some existing formulations \cite{Hoffman2011,Idelsohn2008,Liu1986} that bear some similarity to ours, the key difference being that the Cauchy stress was expressed in a rate form in prior formulations. It is known that the rate constitutive equations are not built from free energies and cannot account for reversible elastic behavior \cite{Simo1984}. Therefore, prior formulations cannot have an a priori energy stability estimate. Additionally, the rate constitutive equation requires special numerical considerations \cite{Hughes1980a}. We aim to address these issues through the proposed formulation.  

The choice of elements plays an equally critical role in numerical design for large-strain elasticity problems. Here, we attempt to provide a numerical technique that can be conveniently and robustly extended to the higher-order regime. The NURBS elements have been shown to enjoy superior robustness for large strain analysis \cite{Cottrell2006,Lipton2010}. We adopt the same set of NURBS basis functions for the description of the geometry and approximation of the displacement field, aligning the proposed numerical formulation with the paradigm of isogeometric analysis \cite{Hughes2005}. The unique concept of $k$-refinement in isogeometric analysis allows one to generate higher-continuity basis functions without proliferation of degrees of freedom. However, it should be pointed out that in the setting of mixed finite elements, although the $k$-refinement leads to a pair of velocity-pressure elements that enjoy nearly the optimal constraint ratio \cite[Chapter 4]{Hughes1987}, it has been observed that such element types are not always inf-sup stable \cite{Rueberg2012}. To remedy this issue, it has been proposed to use subdivision technology to generate a NURBS analogue for the Q1-iso-Q2 element \cite{Dortdivanlioglu2017,Kadapa2016,Rueberg2012}. In this work, we adopt an alternative approach, the inf-sup stable smooth generalizations of the Taylor-Hood element. In our opinion, the Taylor-Hood element is more convenient for implementation, especially in the parallel setting. We numerically assess the inf-sup stability for different combinations of the $p$- and $k$-refinements for generating the Taylor-Hood elements. It will be observed that the elements pass the numerical test if the polynomial degree is elevated at least once by the $p$-refinement to generate the discrete velocity space. Using the above new mixed formulation and the stable smooth generalizations of the Taylor-Hood element offer a new approach for incompressible large strain elastodynamics with several appealing features: it is well-behaved in the incompressible regime, the semi-discrete formulation respects energy stability, it does not involve tunable parameters or subgrid scale numerical models, it can achieve improved accuracy, especially for stress calculations, by employing higher-order smooth basis functions.

The remainder of the work is organized as follows. In Section \ref{sec:elasticity}, we state the governing equations and weak formulation for hyper-elastodynamics. In Section \ref{sec:numerical_formulation}, the numerical scheme is presented and its numerical properties are analyzed. Following that, we numerically assess the inf-sup stability of different pairs of mixed NURBS elements. The elements that pass the test are used in the simulations for benchmark problems in Section \ref{sec:numerical_results}. We draw conclusions in Section \ref{sec:conclusion}.

\section{Hyper-elastodynamics}
\label{sec:elasticity}
\subsection{The initial boundary-value problem}
Let $\Omega_{\bm X}$ and $\Omega_{\bm x}$ be bounded open sets in $\mathbb R^{d}$ with Lipschitz boundaries, wherein $d$ represents the number of spatial dimensions. The motion of the body is described by a family of diffeomorphisms parameterized by the time coordinate $t$,
\begin{align*}
\bm\varphi_t(\cdot) = \bm\varphi(\cdot, t) : \Omega_{\bm X} &\rightarrow \Omega_{\bm x}^t = \bm \varphi(\Omega_{\bm X}, t) = \bm \varphi_t(\Omega_{\bm X}), \quad \forall t \geq 0, \\
\bm X &\mapsto \bm x = \bm \varphi(\bm X, t) = \bm \varphi_t(\bm X), \quad \forall \bm X \in \Omega_{\bm X}.
\end{align*}
In the above, $\bm x$ represents the current position of a material particle originally located at $\bm X$, which implies $\bm \varphi(\bm X, 0) = \bm X$. The displacement and velocity of the material particle are defined as
\begin{align*}
\bm U := \bm \varphi(\bm X, t) - \bm \varphi(\bm X, 0) = \bm \varphi(\bm X, t) - \bm X, \qquad
\bm V := \left. \frac{\partial \bm \varphi}{\partial t}\right|_{\bm X}= \left. \frac{\partial \bm U}{\partial t}\right|_{\bm X} = \frac{d\bm U}{dt}.
\end{align*}
In this work, we use $d\left( \cdot \right)/dt$ to denote a total time derivative. The spatial velocity is defined as $\bm v := \bm V \circ \bm \varphi_t^{-1}$. Analogously, we define $\bm u := \bm U \circ \varphi_t^{-1}$. The deformation gradient, the Jacobian determinant, and the right Cauchy-Green tensor are defined as
\begin{align*}
\bm F := \frac{\partial \bm \varphi}{\partial \bm X}, \qquad
J := \textup{det}\left(\bm F \right), \qquad \bm C := \bm F^T \bm F.
\end{align*}
We define $\tilde{\bm F}$ and $\tilde{\bm C}$ as
\begin{align*}
\tilde{\bm F} := J^{-\frac13}\bm F, \qquad \tilde{\bm C} := J^{-\frac23}\bm C,
\end{align*}
which represent the distortional parts of $\bm F$ and $\bm C$. We denote the thermodynamic pressure of the continuum body as $p$ and the density as $\rho$. The mechanical behavior of an elastic material can be described by a Gibbs free energy $G(\tilde{\bm C}, p)$. It is shown in \cite{Liu2018} that the Gibbs free energy can be additively split into an isochoric part and a volumetric part,
\begin{align*}
G(\tilde{\bm C}, p) = G_{ich}(\tilde{\bm C}) + G_{vol}(p).
\end{align*}
The constitutive relations for the density $\rho$, the isothermal compressibility factor $\beta_{\theta}$, and the deviatoric part of the Cauchy stress can be described in terms of the Gibbs free energy as follows,
\begin{align*}
\rho(p) := \left( \frac{d G_{vol}}{d p} \right)^{-1}, \quad \beta_{\theta}(p) := \frac{1}{\rho} \frac{d\rho}{d p} = -\frac{\partial^2 G_{vol}}{\partial p^2} / \frac{\partial G_{vol}}{\partial p}, \quad \bm \sigma^{dev} := J^{-1} \tilde{\bm F} \left( \mathbb P : \tilde{\bm S} \right) \tilde{\bm F}^T ,
\end{align*}
wherein the projector $\mathbb P$ and the fictitious second Piola-Kirchhoff stress $\tilde{\bm S}$ are defined as
\begin{align*}
\mathbb P := \mathbb I - \frac13 \bm C^{-1} \otimes \bm C, \quad \tilde{\bm S} := 2 \frac{\partial \left( \rho_0 G\right) }{\partial \tilde{\bm C}} = 2 \frac{\partial \left( \rho_0 G_{ich} \right)}{\partial \tilde{\bm C}}, 
\end{align*}
$\mathbb I$ is the fourth-order identity tensor, and $\rho_0$ is the density in the referential configuration. Interested readers are referred to \cite{Liu2018} for a detailed discussion of the governing equations and the constitutive relations. It is known that $\rho J = \rho_0$ due to mass conservation in the Lagrangian description. We can therefore introduce $\rho(J) = \rho_0 / J$ as an alternative way of defining the density in the Lagrangian framework. In fact, we will adopt this choice in the following discussion. Under the isothermal condition, the energy equation is decoupled, and it suffices to consider the following equations for the motion of the continuum body,
\begin{align}
\label{eq:strong_form_kinematic}
& \bm 0 = \frac{d\bm u}{dt} - \bm v, && \mbox{ in } \Omega_{\bm x}^t, \displaybreak[2] \\
\label{eq:strong_form_pressure}
& 0 = \beta_{\theta}(p) \frac{dp}{dt} + \nabla_{\bm x} \cdot \bm v && \mbox{ in } \Omega_{\bm x}^t,  \displaybreak[2] \\
\label{eq:strong_form_momentum}
& \bm 0 = \rho(J) \frac{d\bm v}{dt} - \nabla_{\bm x} \cdot \bm \sigma^{dev} + \nabla_{\bm x} p - \rho(J) \bm b, && \mbox{ in } \Omega_{\bm x}^t.
\end{align}
In the above system, the equations \eqref{eq:strong_form_kinematic} describe the kinematic relation, and the equations \eqref{eq:strong_form_pressure} and \eqref{eq:strong_form_momentum} describe the conservation of mass and the balance of linear momentum. The boundary $\Gamma_{\bm x}^t = \partial \Omega^t_{\bm x}$ can be partitioned into two non-overlapping subdivisions:
$
\Gamma_{\bm x}^t = \Gamma_{\bm x}^{g,t} \cup \Gamma_{\bm x}^{h,t},
$
wherein $\Gamma^{g,t}_{\bm x}$ is the Dirichlet part of the boundary, and $\Gamma^{h,t}_{\bm x}$ is the Neumann part of the boundary. Boundary conditions can be stated as
\begin{align}
\label{eq:strong_form_dirichlet_u}
& \bm u = \bm g, \mbox{ on } \Gamma_{\bm x}^{g,t}, \qquad
\bm v = \frac{d\bm g}{dt}, \mbox{ on } \Gamma_{\bm x}^{g,t}, \qquad (\bm \sigma^{dev} - p\bm I) \bm n = \bm h, \mbox{ on } \Gamma_{\bm x}^{h,t}.
\end{align}
Given the initial data $\bm u_0$, $p_0$, and $\bm v_0$, the initial conditions can be stated as
\begin{align}
\label{eq:initial_condition_v}
\bm u(\bm x, 0) = \bm u_0(\bm x), \qquad p(\bm x, 0) = p_0(\bm x), \qquad \bm v(\bm x, 0) = \bm v_0(\bm x).
\end{align}
The equations \eqref{eq:strong_form_kinematic}-\eqref{eq:initial_condition_v} constitute an initial-boundary value problem for elastodynamics. 

\begin{remark}
It is known that $J=1$ is equivalent to $\nabla_{\bm x} \cdot \bm v = 0$ due to the identity $dJ/dt = J \nabla_{\bm x}\cdot \bm v$. However, the usage of $\nabla_{\bm x} \cdot \bm v = 0$, or more generally \eqref{eq:strong_form_pressure}, is uncommon in the literature. A reason is that the constraint $J=1$ is fitted into the elastostatic model, and the usage of $\bm v$ inevitably necessitates an elastodynamic model, which needs additional considerations in the numerical formulation. Another reason could be the missing link between $\beta_{\theta}$ and the strain energy. The constitutive relation for $\beta_{\theta}$ allows compressible materials and is recently derived in \cite{Liu2018}.
\end{remark}

Since the above system looks different from the existing theory for hyperelasticity, we give an example of the constitutive model here. Let $I_1$ and $I_2$ designate the first and second invariants of the right Cauchy-Green tensor, that is,
\begin{align*}
I_1 := \textup{tr} \bm C, \qquad I_2 := \frac12 \left[ \left( \textup{tr}\bm C \right)^2 - \textup{tr}\left( \bm C^2 \right) \right].
\end{align*}
For isotropic materials, the isochoric part of the free energy can be conveniently expressed in terms of $\tilde{I}_1 := J^{-2/3} I_1$ and $\tilde{I}_2 := J^{-4/3} I_2$. The Mooney-Rivlin model can be expressed as
\begin{align*}
G_{ich}(\tilde{\bm C}) = \frac{c_1}{2\rho_0} \left( \tilde{I}_1 - 3 \right) + \frac{c_2}{2\rho_0} \left( \tilde{I}_2 - 3 \right),
\end{align*}
where $c_1$ and $c_2$ are parameters that have the same dimension as pressure. The volumetric part of the Gibbs free energy can be built as a Legendre transformation of the Helmholtz volumetric free energy \cite{Liu2018}. Here, we give an example
\begin{align}
\label{eq:compressible_vol_energy_example}
G_{vol}(p) = \frac{\kappa}{\rho_0} \left( 1 - e^{-\frac{p}{\kappa}} \right),
\end{align}
which is transformed from the energy proposed in \cite{Liu1994}. In \eqref{eq:compressible_vol_energy_example}, $\kappa$ designates the bulk modulus. This free energy leads to the relation 
\begin{align*}
\rho(p) = \rho_0 e^{\frac{p}{\kappa}}, \qquad \beta_{\theta}(p) = 1 / \kappa.
\end{align*}
As the bulk modulus $\kappa$ approaches infinity, the material becomes incompressible, and we have
$G_{vol}(p) = p / \rho_0$ in the limit. This volumetric energy leads to $\rho(p) = \rho_0$ and $\beta_{\theta}(p) = 0$.

\subsection{Reduction to the small-strain theory}
Assuming the strain is infinitesimally small, we have $\nabla_{\bm x} = \nabla_{\bm X}$ and $\rho(J) = \rho_0$. We also assume that $G_{vol}$ adopts the form given in \eqref{eq:compressible_vol_energy_example}. Then the mass equation \eqref{eq:strong_form_pressure} can be written as 
\begin{align}
\label{eq:small_strain_diff_mass}
0 = \frac{1}{\kappa}\frac{dp}{dt} + \frac{d}{dt} \nabla_{\bm x}\cdot \bm u = \frac{d}{dt}\left( \frac{p}{\kappa} + \nabla_{\bm x} \cdot \bm u \right).
\end{align}
Integrating the above relation in time results in
\begin{align}
\label{eq:small_strain_mass}
0 = \frac{p}{\kappa} + \nabla_{\bm x}\cdot \bm u,
\end{align}
with a proper choice of the reference value for the pressure. Assuming further that the we are seeking a static equilibrium solution, the momentum equation \eqref{eq:strong_form_momentum} becomes
\begin{align}
\label{eq:small_strain_momentum}
\nabla_{\bm x} \cdot \bm \sigma^{dev} - \nabla_{\bm x} p = \rho_0 \bm b.
\end{align}
The equations \eqref{eq:small_strain_mass}-\eqref{eq:small_strain_momentum} constitute the classical mixed formulation for the small strain elastostatics \cite[Chapter 4]{Hughes1987}.

\begin{remark}
For elastodynamics, one may instinctively add an inertial term to \eqref{eq:small_strain_momentum} and couple it with \eqref{eq:small_strain_mass}. However, numerical simulations indicate that this system is probably ill-posed. It is suggested to couple \eqref{eq:small_strain_momentum} with \eqref{eq:small_strain_diff_mass} rather than \eqref{eq:small_strain_mass} for dynamic calculations \cite{Scovazzi2016}. A potential mathematical explanation is that \eqref{eq:small_strain_mass} does not provide the proper coercive structure in the dynamic setting. This point will be further clarified in Proposition \ref{prop:energy_stability}.
\end{remark}

\subsection{Weak formulation}
\label{subsec:weak_formulation}
Henceforth, we restrict our discussion to \textit{fully incompressible} materials. Let us denote the trial solution spaces for the displacement, velocity, and pressure in the current domain as $\mathcal S_{\bm u}$, $\mathcal S_{\bm v}$, and $\mathcal S_p$, respectively. The Dirichlet boundary condition defined on $\Gamma^g_{\bm x}$ is properly built into the definitions of the $\mathcal S_{\bm u}$ and $\mathcal S_{\bm v}$. Let $\mathcal V_{\bm v}$ and $\mathcal V_p$ denote the corresponding test function spaces. The mixed formulation on the current configuration can be stated as follows. Find $\bm y(t) := \left\lbrace  \bm u(t), p(t), \bm v(t)\right\rbrace^T \in \mathcal S_{\bm u} \times \mathcal S_p \times \mathcal S_{\bm v}$ such that for $t\in [0, T]$,
\begin{align}
\label{eq:mix_solids_kinematics_current}
& \bm 0 = \mathbf B^k\left( \dot{\bm y}, \bm y \right) :=  \frac{d\bm u}{dt} - \bm v, \displaybreak[2]\\
\label{eq:mix_solids_mass_current}
& 0 = \mathbf B^p\left( w_{p}; \dot{\bm y}, \bm y \right) := \int_{\Omega_{\bm x}^t} w_{p} \nabla_{\bm x} \cdot \bm v d\Omega_{\bm x}, \displaybreak[2] \\
\label{eq:mix_solids_momentum_current}
& 0 = \mathbf B^m\left( \bm w_{\bm v}; \dot{\bm y}, \bm y \right) := \int_{\Omega_{\bm x}^t} \bm w_{\bm v} \cdot \rho(J) \frac{d\bm v}{dt} + \nabla_{\bm x} \bm w_{\bm v} : \bm \sigma_{dev} - \nabla_{\bm x} \cdot \bm w_{\bm v} p - \bm w_{\bm v} \cdot \rho(J)  \bm b d\Omega_{\bm x}  \displaybreak[2] \nonumber \\
& \qquad \qquad \qquad \qquad \qquad - \int_{\Gamma_{\bm x}^{h,t}} \bm w_{\bm v} \cdot \bm h d\Gamma_{\bm x}, 
\end{align}
for $\forall \left\lbrace  w_{p}, \bm w_{\bm v}\right\rbrace \in \mathcal V_p \times \mathcal V_{\bm v}$, with $\bm y(0) = \left\lbrace \bm u_{0}, p_{0}, \bm v_{0} \right\rbrace^T$. Here $\bm u_{0}$, $p_{0}$, and $\bm v_{0}$ are the $\mathcal L^2$ projections of the initial data onto the trial solution spaces. It is worth pointing out that although the material is fully incompressible, we still use $\rho(J)=\rho_0 / J$ in \eqref{eq:mix_solids_momentum_current}, since the resulting discrete scheme cannot guarantee pointwise satisfaction of $J=1$. In the above and henceforth, the formulations for the kinematic equations, the mass equation, and the linear momentum equations are indicated by the superscripts $k$, $p$ and $m$, respectively. The equations \eqref{eq:mix_solids_kinematics_current}-\eqref{eq:mix_solids_momentum_current} constitute the weak form of the problem. Performing integration by parts and using the localization argument, one can show the equivalence between the weak-form problem and the initial-boundary value problem. Let us define the following quantities on the material frame of reference via a pull-back operator:
\begin{align*}
& W_{P}(\bm X,t) := w_p(\varphi_t(\bm X),t), && \bm W_{\bm V}(\bm X ,t) := \bm w_{\bm v}(\varphi_t(\bm X),t), && P(\bm X,t) := p(\varphi_t(\bm X),t), \\
& \bm B(\bm X,t) := \bm b(\varphi_t(\bm X),t), && \bm H(\bm X,t):=\bm h(\varphi_t(\bm X),t), && \bm G(\bm X,t):=\bm g(\varphi_t(\bm X),t).
\end{align*}
Correspondingly, the trial solution spaces are denoted as $\mathcal S_{\bm U}$, $\mathcal S_{P}$, and $\mathcal S_{\bm V}$; the test function spaces are denoted as $\mathcal V_{P}$ and $\mathcal V_{\bm V}$. The weak formulation can be alternatively stated as follows. Find $\bm Y(t) := \left\lbrace  \bm U(t), P(t), \bm V(t)\right\rbrace^T \in \mathcal S_{\bm U} \times \mathcal S_P \times \mathcal S_{\bm V}$ such that for $t\in [0, T]$,
\begin{align}
& \bm 0 = \mathbf B^k\left( \dot{\bm Y}, \bm Y \right) :=  \frac{d\bm U}{dt} - \bm V,\displaybreak[2]\\
& 0 = \mathbf B^p\left( W_P; \dot{\bm Y}, \bm Y \right) := \int_{\Omega_{\bm X}} W_P \nabla_{\bm X} \bm V : \left( J \bm F^{-T} \right) d\Omega_{\bm X}, \displaybreak[2]\\
& 0 = \mathbf B^m\left( \bm W_{\bm V}; \dot{\bm Y}, \bm Y \right) := \int_{\Omega_{\bm X}} \bm W_{\bm V} \cdot \rho_0 \frac{d\bm V}{dt} + \nabla_{\bm X} \bm W_{\bm V} : \left( J \bm \sigma_{dev} \bm F^{-T} \right) - \nabla_{\bm X} \bm W_{\bm V} : \left( J \bm F^{-T} \right) P\displaybreak[2] \nonumber \\
& \hspace{5cm} - \bm W_{\bm V}\cdot \rho_0 \bm B d\Omega_{\bm X} - \int_{\Gamma_{\bm X}^H} \bm W_{\bm V} \cdot \bm H d\Gamma_{\bm X}.
\end{align}
for $\forall \left\lbrace  W_{P}, \bm W_{\bm V}\right\rbrace \in \mathcal V_P \times \mathcal V_{\bm V}$, with $\bm Y(0) = \left\lbrace \bm U_{0}, P_{0}, \bm V_{0} \right\rbrace^T$. Here $\bm U_0$, $P_0$, and $\bm V_0$ are the $\mathcal L^2$ projections of the initial data onto the spaces $\mathcal S_{\bm U}$, $\mathcal S_{P}$, and $\mathcal S_{\bm V}$ respectively.

\section{Numerical formulation}
\label{sec:numerical_formulation}
In this section, we discuss the numerical procedures for the solution of the incompressible hyper-elastodynamics based on the weak formulation given in Section \ref{subsec:weak_formulation}.

\subsection{Spline spaces on the parametric domain}
We start by reviewing the construction of B-splines and NURBS basis functions. Given the polynomial degree $\mathsf p$ and the dimensionality of the B-spline space $\mathsf n$, the knot vector can be represented by $\Xi := \left\{\xi_1, \cdots, \xi_{\mathsf n + \mathsf p + 1} \right\}$, wherein $0=\xi_1 \leq \xi_2 \leq \cdots \leq \xi_{\mathsf n + \mathsf p + 1}=1$.
With the knot vector, the B-spline basis functions of degree $\mathsf p$, denoted as $\mathsf N_{i}^{\mathsf p}$ for $i=1,\cdots, \mathsf n$, can be defined recursively. The definition starts with the case of $\mathsf p=0$, in which the basis functions are defined as piecewise constants,
\begin{align*}
\mathsf N_{i}^0(\xi) = 
\begin{cases}
1 & \mbox{ if } \xi_{i} \leq \xi < \xi_{i+1}, \\
0 & \mbox{ otherwise}.
\end{cases}
\end{align*}
For $\mathsf p\geq 1$, the basis functions are defined through the Cox-de Boor recursion formula,
\begin{align*}
\mathsf N_{i}^{\mathsf p}(\xi)= \frac{\xi - \xi_{i}}{\xi_{i+ \mathsf p} - \xi_{i}} \mathsf N_{i}^{\mathsf p-1}(\xi) + \frac{\xi_{i+\mathsf p+1} - \xi}{\xi_{i+ \mathsf p+1} - \xi_{i+1}} \mathsf N_{i+1}^{\mathsf p-1}(\xi).
\end{align*}
The NURBS basis functions of degree $\mathsf p$ are defined by the B-spine basis functions and a weight vector $\{\mathsf w_1, \cdots , \mathsf w_{\mathsf n} \}$ as
\begin{align*}
\mathsf R_i^{\mathsf p}(\xi) := \frac{\mathsf w_i \mathsf N^{\mathsf p}_i(\xi)}{\sum_{j=1}^{n}\mathsf w_j \mathsf N^{\mathsf p}_j(\xi)}.
\end{align*}
If we ignore the repetitive knots, the knot vector can be defined by a vector $\{\zeta_1, \cdots, \zeta_{\mathsf m} \}$ representing the distinctive knots and a vector $\{r_1, \cdots, r_{\mathsf m}\}$ recording the corresponding knot multiplicities. In this work, we consider open knot vectors, meaning $r_1 = r_{\mathsf m} = \mathsf p+1$. We further assume that $r_i \leq \mathsf p$ for $i=2,\cdots, \mathsf m-1$. At the point $\zeta_i$, the B-spline basis functions have $\alpha_i := \mathsf p - r_i$ continuous derivatives. The vector
\begin{align*}
\bm \alpha := \{ \alpha_1, \alpha_2, \cdots, \alpha_{\mathsf m-1}, \alpha_m \} = \{ -1, \alpha_2, \cdots, \alpha_{\mathsf m-1}, -1 \}
\end{align*}
is referred to as the regularity vector. We adopt the notation
\begin{align*}
\bm \alpha - 1 := \{\alpha_1, \alpha_2 - 1, \cdots, \alpha_{\mathsf m-1}-1, \alpha_{\mathsf m}  \} = \{-1, \alpha_2 - 1, \cdots, \alpha_{\mathsf m-1}-1, -1 \}.
\end{align*}
When $\alpha_i$ takes the value $-1$, the basis functions are discontinuous at $\zeta_i$. The spaces $\mathcal N^{\mathsf p}_{\bm \alpha}$ and $\mathcal R^{\mathsf p}_{\bm \alpha}$ are defined as
\begin{align*}
\mathcal N^{\mathsf p}_{\bm \alpha} := \textup{span}\{\mathsf N_{i}^{\mathsf p}\}_{i=1}^n, \quad \mathcal R^{\mathsf p}_{\bm \alpha} := \textup{span}\{\mathsf R_{i}^{\mathsf p}\}_{i=1}^n .
\end{align*}
The notations $\mathcal N^{\mathsf p}_{\alpha}$ and $\mathcal R^{\mathsf p}_{\alpha}$ are used to indicate that $\alpha_i = \alpha$ for $i=2,\cdots, \mathsf m-1$, meaning the spline function spaces have continuity $C^{\alpha}$. The construction of multivariate B-spline and NURBS basis functions follows a tensor-product manner. Consider a unit cube $\hat{\Omega} := (0,1)^{d}$, which is referred to as the parametric domain. Given $\mathsf p_{l}$, $\mathsf n_{l}$ for $l=1,\cdots , d$, we denote the knot vectors as $\Xi_{\mathsf l} = \{\xi_{1,l}, \cdots , \xi_{\mathsf n_{l}+\mathsf p_{l}+1, l}\}$. Associated with each knot vector, the univariate B-spline basis functions $\mathsf N^{\mathsf p_{l}}_{i_{l},l}$ for $i_{l} = 1, \cdots, \mathsf n_{l}$ are defined. Consequently, the tensor-product B-spline basis functions can be defined as
\begin{align*}
\mathsf N^{\mathsf p_1, \cdots, \mathsf p_{d}}_{i_1, \cdots, i_{d}}(\xi_1, \cdots, \xi_d) := \mathsf N^{\mathsf p_1}_{i_1, 1}(\xi_1) \otimes \cdots \otimes \mathsf N^{\mathsf p_{d}}_{i_{d}, d}(\xi_d),  \mbox{ for } i_1 = 1,\cdots, \mathsf n_1, \quad \cdots, \quad i_{d} = 1, \cdots, \mathsf n_d.
\end{align*}
Given the weight vectors $\{\mathsf w_{1,l}, \cdots , \mathsf w_{\mathsf n,l}\}$ for $l=1,\cdots, d$, the univariate NURBS basis functions $\mathsf R^{\mathsf p_l}_{i_l,l}$ are defined. Correspondingly, the multivariate NURBS basis functions are defined as
\begin{align*}
\mathsf R^{\mathsf p_1,\cdots , \mathsf p_d}_{i_1,\cdots ,i_d}(\xi_1,\cdots ,\xi_d) := \mathsf R^{\mathsf p_1}_{i_1, 1}(\xi_1) \otimes \cdots \otimes \mathsf R^{\mathsf p_{d}}_{i_{d}, d}(\xi_d), \mbox{ for } i_1 = 1,\cdots, \mathsf n_1, \quad \cdots, \quad i_{d} = 1, \cdots, \mathsf n_d.
\end{align*}
The tensor product NURBS space is denoted as
\begin{align*}
\mathcal R^{\mathsf p_1,\cdots \mathsf p_d}_{\bm \alpha_1, \cdots ,\bm \alpha_d} := \mathcal R^{\mathsf p_1}_{\bm \alpha_1} \otimes \cdots \otimes \mathcal R^{\mathsf p_d}_{\bm \alpha_d} = \textup{span}\{ \mathsf R^{\mathsf p_1,\cdots ,\mathsf p_d}_{i_1, \cdots , i_d} \}_{i_1=1, \cdots , i_d =1}^{\mathsf n_1, \cdots , \mathsf n_d}.
\end{align*}

\subsection{Semi-discrete formulation and a priori estimates}
\label{subsec:semi_discrete_formulation}
In this work, we always consider three-dimensional problems (i.e. $d=3$). Two discrete function spaces $\hat{\mathcal S}_h$ and $\hat{\mathcal P}_h$ can be defined on $\hat{\Omega} = (0,1)^3$ as
\begin{align*}
\hat{\mathcal S}_h :=&
\mathcal R^{\mathsf p+\mathsf a, \mathsf p+\mathsf a, \mathsf p+\mathsf a}_{\bm \alpha_1+\mathsf b, \bm \alpha_2+\mathsf b, \bm \alpha_3+\mathsf b} \times \mathcal R^{\mathsf p+\mathsf a, \mathsf p+\mathsf a, \mathsf p+\mathsf a}_{\bm \alpha_1+\mathsf b, \bm \alpha_2+\mathsf b, \bm \alpha_3+\mathsf b} \times \mathcal R^{\mathsf p+\mathsf a, \mathsf p+\mathsf a, \mathsf p+\mathsf a}_{\bm \alpha_1+\mathsf b, \bm \alpha_2+\mathsf b, \bm \alpha_3+\mathsf b}, \\
\hat{\mathcal P}_h :=& \mathcal R^{\mathsf p, \mathsf p, \mathsf p}_{\bm \alpha_1, \bm \alpha_2, \bm \alpha_3},
\end{align*}
where $1 \leq \mathsf a $ and $0 \leq \mathsf b \leq \mathsf a$ are integers. We assume that the referential configuration of the body can be exactly parametrized by a geometrical mapping $\bm \psi : \hat{\Omega} \rightarrow \Omega_{\bm X}$. The discrete functions on the referential domain are defined through the pull-back operators,
\begin{align*}
& \mathcal S_h :=  \{ \bm w : \bm w \circ \bm \psi \in \hat{\mathcal S}_h \}, \quad \mathcal P_h := \{ q : q \circ \bm \psi \in \hat{\mathcal P}_h \}.
\end{align*}
This pair of elements can be viewed as a generalization of the Taylor-Hood element \cite{Hood1973}, where the polynomial degree and the continuity can achieve arbitrarily high order. With the discrete function spaces $\mathcal S_h$ and $\mathcal P_h$ defined, we define the trial solution spaces for the displacement, pressure, and velocity on the referential configuration as
\begin{align*}
\mathcal S_{\bm U_h} &= \Big\lbrace \bm U_h : \bm U_h(\cdot, t) \in \mathcal S_h, t \in [0,T], \quad \bm U_h(\cdot,t) = \bm G \mbox{ on } \Gamma_{\bm X}^{G}  \Big\rbrace , \displaybreak[2] \\
\mathcal S_{P_h} &= \Big\lbrace P_h : P_h(\cdot, t) \in \mathcal P_h, t \in [0,T] \Big\rbrace , \displaybreak[2] \\
 \mathcal S_{\bm V_h} &= \left\lbrace \bm V_h : \bm V_h(\cdot, t) \in \mathcal S_h, t \in [0,T], \quad \bm V_h(\cdot,t) = \frac{d\bm G}{dt} \mbox{ on } \Gamma_{\bm X}^{G}  \right\rbrace .
\end{align*}
Given the displacement $\bm U_h \in \mathcal S_{\bm U_h}$, one may obtain $\bm \varphi_h = \bm U_h(\bm X, t) + \bm X$. Consequently, the trial solution spaces for the displacement, pressure, and velocity on the current configuration can be defined as
\begin{align*}
\mathcal S_{\bm u_h} &= \Big\lbrace \bm u_h : \bm u_h \circ \bm \varphi_h \in \mathcal S_h, t \in [0,T], \bm u_h(\cdot,t) = \bm g \mbox{ on } \Gamma_{\bm x}^{g}  \Big\rbrace , \displaybreak[2] \\
\mathcal S_{p_h} &= \Big\lbrace p_h : p_h \circ \bm \varphi_h \in \mathcal P_h, t \in [0,T] \Big\rbrace , \displaybreak[2] \\
 \mathcal S_{\bm v_h} &= \left\lbrace \bm v_h : \bm v_h \circ \bm \varphi_h \in \mathcal S_h, t \in [0,T], \bm v_h(\cdot,t) = \frac{d\bm g}{dt} \mbox{ on } \Gamma_{\bm x}^{g}  \right\rbrace ,
\end{align*}
and the test function spaces are defined as
\begin{align*}
\mathcal V_{p_h} &= \Big\lbrace w_{p_h} : w_{p_h} \circ \bm \varphi_h \in \mathcal P_h, t \in [0,T]  \Big\rbrace , \displaybreak[2] \\
 \mathcal V_{\bm v_h} &= \Big\lbrace \bm w_{\bm v_h} : \bm w_{\bm v_h} \circ \bm \varphi_h \in \mathcal S_h, t \in [0,T], \bm w_{\bm v_h}(\cdot,t) = \bm 0 \mbox{ on } \Gamma_{\bm x}^{g} \Big\rbrace .
\end{align*}
The semi-discrete formulation can be stated as follows. Find $\bm y_h(t) := \left\lbrace  \bm u_h(t), p_h(t), \bm v_h(t)\right\rbrace^T \in \mathcal S_{\bm u_h} \times \mathcal S_{p_h} \times \mathcal S_{\bm v_h}$ such that for $t\in [0, T]$,
\begin{align}
\label{eq:kinematics_current}
& \bm 0 = \mathbf B^k\left( \dot{\bm y}_h, \bm y_h  \right) := \frac{d\bm u_h}{dt} - \bm v_h, \displaybreak[2]\\
\label{eq:mass_current}
& 0 = \mathbf B^p\left( w_{p_h}; \dot{\bm y}_h, \bm y_h  \right) := \int_{\Omega_{\bm x}^t} w_{p_h} \nabla_{\bm x} \cdot \bm v_h d\Omega_{\bm x}, \displaybreak[2] \\
\label{eq:momentum_current}
& 0 = \mathbf B^m\left( \bm w_{\bm v_h}; \dot{\bm y}_h, \bm y_h  \right) := \int_{\Omega_{\bm x}^t} \bm w_{\bm v_h} \cdot \rho(J_h) \frac{d\bm v_h}{dt} + \nabla_{\bm x} \bm w_{\bm v_h} : \bm \sigma^{dev} - \nabla_{\bm x} \cdot \bm w_{\bm v_h} p_h - \bm w_{\bm v_h} \cdot \rho(J_h)  \bm b d\Omega_{\bm x}  \nonumber \\
& \hspace{4.5cm}  - \int_{\Gamma_{\bm x}^{h,t}} \bm w_{\bm v_h} \cdot \bm h d\Gamma_{\bm x}, 
\end{align}
for $\forall \left\lbrace w_{p_h}, \bm w_{\bm v_h}\right\rbrace \in \mathcal V_{p_h} \times \mathcal V_{\bm v_h}$, with $\bm y_h(0) := \left\lbrace \bm u_{h0}, p_{h0}, \bm v_{h0} \right\rbrace^T$. Here $\bm u_{h0}$, $p_{h0}$, and $\bm v_{h0}$ are the $\mathcal L^2$ projections of the initial data onto the finite dimensional trial solution spaces. In the following, we demonstrate that the above semi-discrete formulation is embedded with energy stability and momentum conservation properties. The properties guarantee that the numerical solutions preserve critical structures of the original system. In contrast, to the best of the authors' knowledge, there is no such stability estimate for the conventional mixed $u/p$ formulation \cite{Sussman1987} or the formulations based on rate constitutive equations \cite{Hoffman2011,Idelsohn2008,Liu1986}. 

\begin{proposition}[A priori energy stability estimate]
\label{prop:energy_stability}
For fully incompressible materials, assuming the boundary data $\bm g$ is time independent, we have 
\begin{align}
\label{eq:semi_discrete_energy_stability}
\frac{d}{dt} \int_{\Omega_{\bm X}} \frac12 \rho_0 \|\bm V_h\|^2 + \rho_0 G_{ich}(\tilde{\bm C}_h) d\Omega_{\bm X} = \int_{\Omega_{\bm X}} \rho_0 \bm V_h \cdot \bm B d\Omega_{\bm X} + \int_{\Gamma_{\bm X}} \bm V_h \cdot \bm H d\Gamma_{\bm X}.
\end{align}
\end{proposition}

\begin{proof}
Since the Dirichlet boundary data $\bm g$ is independent of time, one is allowed to choose $w_{p_h} = p_h$ in \eqref{eq:mass_current} and $\bm w_{\bm u_h} = \bm v_h$ in \eqref{eq:momentum_current}, and this leads to the following,
\begin{align*}
0 =& \mathbf B^p\left( p_h; \dot{\bm y}_h, \bm y_h  \right) + \mathbf B^m\left( \bm v_h; \dot{\bm y}_h, \bm y_h  \right) \nonumber \\
=& \int_{\Omega_{\bm x}^t} p_h \nabla_{\bm x} \cdot \bm v_h d\Omega_{\bm x} + \int_{\Omega_{\bm x}^t} \bm v_h \cdot \rho(J_h) \frac{d\bm v_h}{dt} + \nabla_{\bm x} \bm v_h : \bm \sigma^{dev} - \nabla_{\bm x} \cdot \bm v_h p_h - \bm v_h \cdot \rho(J_h)  \bm b d\Omega_{\bm x} \nonumber \\
&  - \int_{\Gamma_{\bm x}^{h,t}} \bm v_h \cdot \bm h d\Gamma_{\bm x} \\
=& \frac{d}{dt}\int_{\Omega_{\bm X}} \frac12 \rho_0 \|\bm V_h\|^2 d\Omega_{\bm X} + \int_{\Omega_{\bm X}} \frac{d}{dt}\bm F_h : \frac{d\left( \rho_0 G_{ich} ( \tilde{\bm C}_h ) \right)}{d \bm F} - \bm V_h \cdot \rho_0  \bm B d\Omega_{\bm X} - \int_{\Gamma_{\bm X}} \bm V_h \cdot \bm H d\Gamma_{\bm X}.
\end{align*}
Rearranging terms in the above equality leads to
\begin{align*}
\frac{d}{dt} \int_{\Omega_{\bm X}} \frac12 \rho_0 \|\bm V_h\|^2 + \rho_0 G_{ich}(\tilde{\bm C}_h) d\Omega_{\bm X} = \int_{\Omega_{\bm X}} \rho_0 \bm V_h \cdot \bm B d\Omega_{\bm X} + \int_{\Gamma_{\bm X}} \bm V_h \cdot \bm H d\Gamma_{\bm X}.
\end{align*}
\end{proof}
\begin{remark}
For compressible materials, one may analogously obtain a stability bound where a pressure-squared term enters into the integral on the left-hand side of \eqref{eq:semi_discrete_energy_stability}. This gives a mathematical reason for the success of equal-order interpolations when the material is compressible. However, we do not favor this type of `energy' estimates because the pressure-squared term does not carry physical meanings. To remedy this issue, an entropy variable can be introduced by leveraging the convexity of the volumetric energy, and a physically relevant entropy stability is expected \cite{Liu2013,Shakib1991}. This is beyond the scope of this work and remains an area of future research.
\end{remark}

\begin{proposition}[Semi-discrete momentum conservation]
Considering the pure Neumann boundary condition, we have the following conservation properties of the semi-discrete formulation,
\begin{align*}
\frac{d}{dt}\int_{\Omega_{\bm X}} \rho_0 \bm V_h d\Omega_{\bm X} =& \int_{\Omega_{\bm X}} \rho_0 \bm B d\Omega_{\bm X} + \int_{\Gamma_{\bm X}} \bm H d\Gamma_{\bm X}, \\
\frac{d}{dt}\int_{\Omega_{\bm X}} \rho_0 \bm \varphi_h \times \bm V_h d\Omega_{\bm X} =& \int_{\Omega_{\bm X}} \rho_0 \bm \varphi_h \times \bm B d\Omega_{\bm X} + \int_{\Gamma_{\bm X}} \bm \varphi_h \times \bm H d\Gamma_{\bm X}.
\end{align*}
\end{proposition}
\begin{proof}
The above conservation properties are direct consequences of choosing $\bm w_{\bm v_h} = \bm e_i$ and $\bm w_{\bm v_h} = \bm e_i \times \bm \varphi_h$ respectively in \eqref{eq:momentum_current}, where $\bm e_i$ is a unit vector in the $i$-th direction.
\end{proof}

Due to the incompressibility, the pressure force does not contribute to the energy. Therefore, the energy stability estimate \eqref{eq:semi_discrete_energy_stability} does not involve the pressure field. The inf-sup condition needs to be utilized to provide a bound for the pressure field. We assume that there exists a positive constant $\beta$ such that
\begin{align}
\label{eq:inf_sup_cond}
\inf_{p_h \in \mathcal S_{p_h}} \sup_{\bm v_h \in \mathcal S_{\bm v_{h}}} \frac{ \int_{\Omega_{\bm x}} p_h \nabla_{\bm x} \cdot \bm v_h d\Omega_{\bm x}  }{\|\bm v_h\|_{1}\|p_h\|_{0}} \geq \beta,
\end{align}
wherein $\|\cdot\|_0$ and $\|\cdot\|_1$ denote the $\mathcal L^2$ and $\mathcal H^1$ norm over $\Omega_{\bm x}$. Using the semi-discrete equation \eqref{eq:momentum_current}, the above inequality implies
\begin{align*}
\beta \|p_h\|_{0} &\leq \sup_{\hat{\bm v}_h \in \mathcal S_{\bm v_{h}}} \frac{\int_{\Omega_{\bm x}} p_h \nabla_{\bm x} \cdot \hat{\bm v}_h d\Omega_{\bm x}}{ \|\hat{\bm v}_h\|_{1} } \\
&= \sup_{\hat{\bm v}_h \in \mathcal S_{\bm v_{h}}} \frac{\int_{\Omega_{\bm x}} \rho(J_h) \hat{\bm v}_h \cdot \frac{d\bm v_h}{dt} + \nabla_{\bm x} \hat{\bm v}_h : \bm \sigma^{dev} - \rho(J_h) \hat{\bm v}_h \cdot \bm b  d\Omega_{\bm x} + \int_{\Gamma_{\bm x}^{h,t}} \hat{\bm v}_h \cdot \bm h d\Gamma_{\bm x}  }{ \|\hat{\bm v}_h\|_{1} }.
\end{align*}
If we further assume that $\rho(J_h)$ is uniformly bounded, using the Cauchy-Schwarz  inequality, we may get 
\begin{align*}
\|p_h\|_{0} \leq \tilde{C} \left( \| \frac{d\bm v_h}{dt} \|_{L_2(\Omega_{\bm x})} + \|\bm \sigma^{dev}\|_{L_2(\Omega_{\bm x})} + \|\bm b\|_{L_2(\Omega_{\bm x})} + \|\bm h\|_{L_2(\Gamma_{\bm x}^{h})} \right),
\end{align*}
with $\tilde{C}$ being a constant. Therefore, given the velocity, the deformation state, and the external forces, the pressure field is bounded. We note that the assumption on the boundedness of the density cannot be rigorously justified based on the current numerical formulation. It is anticipated that this issue can be resolved by invoking the structure-preserving discretization technique \cite{Evans2013}, which results in discrete solutions with pointwise divergence-free velocity field. With the exact satisfaction of the incompressibility constraint, the density remains as a constant.
\begin{remark}
The linearization of $J-1=0$ results in a divergence operator acting on the virtual displacement field. This fact has been frequently used to justify the usage of inf-sup stable elements in the two-field variational principle \cite{Auricchio2005}. However, we feel this may not be a good interpretation. First, the linearization argument cannot recover the compressible case \eqref{eq:small_strain_mass}. Second, the solvability of the Newton-Raphson procedure does not provide a bound for the solution.
\end{remark}

\subsection{Temporal discretization}
\label{subsec:temporal_discretization}
We invoke the generalized-$\alpha$ method \cite{Jansen2000} for the temporal discretization of the weak form problem \eqref{eq:mix_solids_kinematics_current}-\eqref{eq:mix_solids_momentum_current}. The time interval $[0,T]$ is divided into a set of $n_{ts}$ subintervals of size $\Delta t_n := t_{n+1} - t_n$ delimited by a discrete time vector $\left\lbrace t_n \right\rbrace_{n=0}^{n_{ts}}$. The solution vector and its first-order time derivative evaluated at the time step $t_n$ are denoted as $\bm y_n$ and $\dot{\bm y}_n$. The fully discrete scheme can be stated as follows. At time step $t_n$, given $\dot{\bm y}_n$, $\bm y_n$, the time step size $\Delta t_n$, and the parameters $\alpha_m$, $\alpha_f$, and $\gamma$, find $\dot{\bm y}_{n+1}$ and $\bm y_{n+1}$ such that for $\forall \left\lbrace  w_{p}, \bm w_{\bm v}\right\rbrace \in \mathcal V_p \times \mathcal V_{\bm v}$,
\begin{align}
\label{eq:gen_alpha_fully_discrete_kinematic_eqn}
& \mathbf B^k_{t_{n+\alpha_f}}\left( \dot{\bm y}_{n+\alpha_m}, \bm y_{n+\alpha_f} \right) = \bm 0, \displaybreak[2] \\
\label{eq:gen_alpha_fully_discrete_pressure_rate}
& \mathbf B^p_{t_{n+\alpha_f}}\left( w_{p} ;  \dot{\bm y}_{n+\alpha_m}, \bm y_{n+\alpha_f} \right) = 0, \displaybreak[2] \\
\label{eq:gen_alpha_fully_discrete_linear_momentum}
& \mathbf B^m_{t_{n+\alpha_f}}\left( \bm w_{\bm v} ;  \dot{\bm y}_{n+\alpha_m}, \bm y_{n+\alpha_f} \right) = 0, \displaybreak[2] \\
\label{eq:gen_alpha_def_y_n_plus_1}
& \bm y_{n+1} = \bm y_{n} + \Delta t_n \dot{\bm y}_n, + \gamma \Delta t_n \left( \dot{\bm y}_{n+1} - \dot{\bm y}_{n}\right), \displaybreak[2] \\
\label{eq:gen_alpha_def_y_n_alpha_m}
& \dot{\bm y}_{n+\alpha_m} = \dot{\bm y}_{n} + \alpha_m \left(\dot{\bm y}_{n+1} - \dot{\bm y}_{n} \right), \displaybreak[2] \\
\label{eq:gen_alpha_def_y_n_alpha_f}
& \bm y_{n+\alpha_f} = \bm y_{n} + \alpha_f \left( \bm y_{n+1} - \bm y_{n} \right).
\end{align}
The choice of parameters $\alpha_m$, $\alpha_f$ and $\gamma$ determines the accuracy and stability of the temporal scheme. Importantly, the high-frequency dissipation can be controlled via a proper parametrization of these parameters, while maintaining second-order accuracy and unconditional stability (for linear problems). For the above first-order dynamic problems, the parametrization is
\begin{align*}
\alpha_m = \frac12 \left( \frac{3-\varrho_{\infty}}{1+\varrho_{\infty}} \right), \quad \alpha_f = \frac{1}{1+\varrho_{\infty}}, \quad \gamma = \frac{1}{1+\varrho_{\infty}},
\end{align*}
wherein $\varrho_{\infty} \in [0,1]$ denotes the spectral radius of the amplification matrix at the highest mode \cite{Jansen2000}. Setting $\varrho_{\infty} = 1$ recovers the mid-point rule. For nonlinear structural dynamics, the mid-point rule is observed to have a pile-up effect for the energy error and often leads to diverged results for long-time simulations. In this study, the value of $\varrho_{\infty}$ is fixed to be $0.5$.
\begin{remark}
Interested readers are referred to \cite{Chung1993} for the parametrization of $\alpha_m$, $\alpha_f$, and $\gamma$ for second-order structural dynamics. A recent study shows that using the generalized-$\alpha$ method for first-order structural dynamics enjoys improved dissipation and dispersion properties and does not suffer from overshoot \cite{Kadapa2017}. Moreover, using a first-order structural dynamic model is quite propitious for the design of an FSI scheme \cite{Liu2018}.
\end{remark}

\begin{remark}
It is tempting to apply the discrete energy-momentum methods \cite{Simo1992a} to the semi-discrete system. Those algorithms yield fully discrete systems that inherit the energy stability and momentum conservation properties and are thence particularly well-suited for transient analysis. For problems we are interested in, the solution may be driven to a static equilibrium by external forces, and the stress formula in the energy-momentum methods will become ill-defined. Because of this, we retain the generalized-$\alpha$ method in this work.
\end{remark}

\subsection{A Segregated predictor multi-corrector algorithm}
\label{subsec:seg_algorithm}
The equations \eqref{eq:gen_alpha_fully_discrete_kinematic_eqn}-\eqref{eq:gen_alpha_def_y_n_alpha_f} constitute a system of nonlinear algebraic equations to be solved in each time step, and we invoke the Newton-Raphson method with consistent linearization. At time step $t_{n+1}$, the solution vector $\bm y_{n+1}$ is solved by means of a predictor multi-corrector algorithm. We denote $\bm y_{n+1,(l)} := \left\lbrace \bm u_{n+1,(l)}, p_{n+1,(l)}, \bm v_{n+1,(l)} \right\rbrace^T$ as the solution vector at the Newton-Raphson iteration step $l=0, \cdots, l_{max}$. The residual vectors evaluated at the iteration stage $l$ are denoted as
\begin{align*}
\boldsymbol{\mathrm R}_{(l)} &:= \left\lbrace \boldsymbol{\mathrm R}^{k}_{(l)}, \boldsymbol{\mathrm R}^p_{(l)}, \boldsymbol{\mathrm R}^m_{(l)} \right\rbrace^T, \\
\boldsymbol{\mathrm R}^k_{(l)} &:= \boldsymbol{\mathrm R}^k\left( \dot{\bm y}_{n+\alpha_m, (l)}, \bm y_{n+\alpha_f, (l)} \right), \nonumber \\
\boldsymbol{\mathrm R}^p_{(l)} &:= \boldsymbol{\mathrm R}^p\left( \dot{\bm y}_{n+\alpha_m, (l)}, \bm y_{n+\alpha_f, (l)} \right), \nonumber \\
\boldsymbol{\mathrm R}^m_{(l)} &:= \boldsymbol{\mathrm R}^m\left( \dot{\bm y}_{n+\alpha_m, (l)}, \bm y_{n+\alpha_f, (l)} \right). \nonumber
\end{align*}
The consistent tangent matrix associated with the above residual vectors is
\begin{align*}
\boldsymbol{\mathrm K}_{(l)} =
\begin{bmatrix}
\boldsymbol{\mathrm K}^k_{(l), \dot{\bm u}} & \boldsymbol{\mathrm O} & \boldsymbol{\mathrm K}^k_{(l), \dot{\bm v}} \\[0.3mm]
\boldsymbol{\mathrm K}^p_{(l), \dot{\bm u}} & \boldsymbol{\mathrm O} & \boldsymbol{\mathrm K}^p_{(l), \dot{\bm v}} \\[0.3mm]
\boldsymbol{\mathrm K}^m_{(l), \dot{\bm u}} & \boldsymbol{\mathrm K}^m_{(l), \dot{p}} & \boldsymbol{\mathrm K}^m_{(l), \dot{\bm v}}
\end{bmatrix},
\end{align*}
wherein
\begin{align*}
& \boldsymbol{\mathrm K}^k_{(l), \dot{\bm u}} := \alpha_m \frac{\partial \boldsymbol{\mathrm R}^k_{(l)}\left( \dot{\bm y}_{n+\alpha_m, (l)}, \bm y_{n+\alpha_f, (l)} \right)}{\partial \dot{\bm u}_{n+\alpha_m}} = \alpha_m \boldsymbol{\mathrm I}, \displaybreak[2] \\
& \boldsymbol{\mathrm K}^k_{(l), \dot{\bm v}} := \alpha_f \gamma \Delta t_n \frac{\partial \boldsymbol{\mathrm R}^k_{(l)}\left( \dot{\bm y}_{n+\alpha_m, (l)}, \bm y_{n+\alpha_f, (l)} \right)}{\partial \bm v_{n+\alpha_f}} = -\alpha_f \gamma \Delta t_n \boldsymbol{\mathrm I},
\end{align*}
$\boldsymbol{\mathrm I}$ is the identity matrix, and $\boldsymbol{\mathrm O}$ is the zero matrix. The above diagonal structure of the two blocks can be utilized to construct a block factorization of $\boldsymbol{\mathrm K}_{(l)}$, with which the solution procedure of the linear system of equations in the Newton-Raphson method can be consistently reduced to a two-stage algorithm \cite{Liu2018,Scovazzi2016}. In the first stage, one obtains the increments of the pressure and velocity at the iteration step $l$ by solving the following linear system,
\begin{align}
\label{eq:NR_smaller_eqn_1}
&
\begin{bmatrix}
\boldsymbol{\mathrm K}^m_{(l), \dot{\bm v}} +  \frac{\alpha_f \gamma \Delta t_n}{\alpha_m} \boldsymbol{\mathrm K}^m_{(l), \dot{\bm u}} & \boldsymbol{\mathrm K}^m_{(l), \dot{p}} \\[0.3mm]
\boldsymbol{\mathrm K}^p_{(l), \dot{\bm v}} + \frac{\alpha_f \gamma \Delta t_n}{\alpha_m} \boldsymbol{\mathrm K}^p_{(l), \dot{\bm u}} & \boldsymbol{\mathrm O}
\end{bmatrix}
\begin{bmatrix}
\Delta \dot{\bm v}_{n+1,(l)} \\[0.3mm]
\Delta \dot{p}_{n+1,(l)}
\end{bmatrix}
= -
\begin{bmatrix}
\boldsymbol{\mathrm R}^m_{(l)}- \frac{1}{\alpha_m}  \boldsymbol{\mathrm K}^m_{(l), \dot{\bm u}} \boldsymbol{\mathrm R}^k_{(l)} \\[0.3mm]
\boldsymbol{\mathrm R}^p_{(l)} - \frac{1}{\alpha_m}  \boldsymbol{\mathrm K}^p_{(l), \dot{\bm u}}\boldsymbol{\mathrm R}^k_{(l)}
\end{bmatrix}.
\end{align}
In the second stage, one obtains the increments for the displacement by
\begin{align}
\label{eq:NR_smaller_eqn_2}
& \Delta \dot{\bm u}_{n+1,(l)} = \frac{\alpha_f \gamma \Delta t_n}{\alpha_m} \Delta \dot{\bm v}_{n+1,(l)} - \frac{1}{\alpha_m}\boldsymbol{\mathrm R}^k_{(l)}.
\end{align}
To simplify notations in the following discussion, we denote 
\begin{align}
\label{eq:def_matrix_ABC}
\boldsymbol{\mathrm A}_{(l)} := \boldsymbol{\mathrm K}^m_{(l), \dot{\bm v}} +  \frac{\alpha_f \gamma \Delta t_n}{\alpha_m} \boldsymbol{\mathrm K}^m_{(l), \dot{\bm u}}, \quad
\boldsymbol{\mathrm B}_{(l)} := \boldsymbol{\mathrm K}^m_{(l), \dot{p}}, \quad
\boldsymbol{\mathrm C}_{(l)} := \boldsymbol{\mathrm K}^p_{(l), \dot{\bm v}} + \frac{\alpha_f \gamma \Delta t_n}{\alpha_m} \boldsymbol{\mathrm K}^p_{(l), \dot{\bm u}}.
\end{align}
Readers are referred to the Appendix of \cite{Liu2019} for the explicit formulas of the block matrices in \eqref{eq:def_matrix_ABC}.
\begin{remark}
In \cite{Liu2018}, it was shown that $\boldsymbol{\mathrm R}^k_{(l)} = \bm 0$ for $l \geq 2$ for general predictor multi-corrector algorithms; in \cite{Rossi2016}, a special predictor is chosen so that $\boldsymbol{\mathrm R}^k_{(l)} = \bm 0$ for $l \geq 1$. In our experience, setting $\boldsymbol{\mathrm R}^k_{(l)} = \bm 0$ for  $l \geq 1$, regardless of the predictor chosen, simplifies the implementation and does not deteriorate the convergence rate of the Newton-Raphson solution procedure.
\end{remark}
Based on the above discussion, a predictor multi-corrector algorithm for solving the nonlinear algebraic equations in each time step can be summarized as follows.

\noindent \textbf{Predictor stage}: Set:
\begin{align*}
\bm y_{n+1, (0)} = \bm y_{n}, \quad
\dot{\bm y}_{n+1, (0)} = \frac{\gamma - 1}{\gamma} \dot{\bm y}_{n}.
\end{align*}

\noindent \textbf{Multi-corrector stage}:
Repeat the following steps for \(l = 1, \dots, l_{max}\):
\begin{enumerate}
\item Evaluate the solution vectors at the intermediate stages:
\begin{align*}
\bm y_{n+\alpha_f, (l)} = \bm y_n + \alpha_f \left( \bm y_{n+1,(l-1)} - \bm y_n \right), \quad
\dot{\bm y}_{n+\alpha_m, (l)} = \dot{\bm y}_n + \alpha_m \left( \dot{\bm y}_{n+1,(l-1)} - \dot{\bm y}_n \right).
\end{align*}

\item Assemble the residual vectors $\boldsymbol{\mathrm R}^m_{(l)}$ and $\boldsymbol{\mathrm R}^p_{(l)}$ using $\bm y_{n+\alpha_f, (l)}$ and $\dot{\bm y}_{n+\alpha_m, (l)}$.

\item Let $\|\boldsymbol{\mathrm R}_{(l)}\|_{\mathfrak{l}^2}$ denote the $\mathfrak l^2$-norm of the residual vector. If either one of the following stopping criteria 
\begin{align*}
& \frac{\|\boldsymbol{\mathrm R}_{(l)}\|_{\mathfrak l^2}}{\|\boldsymbol{\mathrm R}_{(0)}\|_{\mathfrak l^2}} \leq \textup{tol}_{\textup{R}}, \qquad \|\boldsymbol{\mathrm R}_{(l)}\|_{\mathfrak l^2} \leq \textup{tol}_{\textup{A}},
\end{align*}
is satisfied for two prescribed tolerances $\textup{tol}_{\textup{R}}$, $\textup{tol}_{\textup{A}}$, set the solution vector at the time step $t_{n+1}$ as $\bm y_{n+1} = \bm y_{n+1, (l-1)}$ and $\dot{\bm y}_{n+1} = \dot{\bm y}_{n+1, (l-1)}$, and exit the multi-corrector stage; otherwise, continue to step 4.

\item Assemble the tangent matrices \eqref{eq:def_matrix_ABC}.

\item Solve the following linear system of equations for $\Delta \dot{p}_{n+1,(l)}$ and $\Delta \dot{\bm v}_{n+1,(l)}$,
\begin{align}
\label{eq:pre-multi-corrector-stage-5-matrix-problems}
\begin{bmatrix}
\boldsymbol{\mathrm A}_{(l)} & \boldsymbol{\mathrm B}_{(l)} \\[0.3mm]
\boldsymbol{\mathrm C}_{(l)} & \boldsymbol{\mathrm O}
\end{bmatrix}
\begin{bmatrix}
\Delta \dot{\bm v}_{n+1,(l)} \\[0.3mm]
\Delta \dot{p}_{n+1,(l)}
\end{bmatrix}
= -
\begin{bmatrix}
\boldsymbol{\mathrm R}^m_{(l)} \\[0.3mm]
\boldsymbol{\mathrm R}^p_{(l)}
\end{bmatrix}.
\end{align}

\item Obtain $\Delta \dot{\bm u}_{n+1,(l)}$ from the relation \eqref{eq:NR_smaller_eqn_2}.

\item Update the solution vector as
\begin{align*}
\bm y_{n+1,(l)} = \bm y_{n+1,(l)} + \gamma \Delta t_n \Delta \dot{\bm y}_{n+1,(l)}, \quad \dot{\bm y}_{n+1,(l)} = \dot{\bm y}_{n+1,(l)} + \Delta \dot{\bm y}_{n+1,(l)}.
\end{align*}
and return to step 1.
\end{enumerate}
In our experience, the choice of the linear solver for \eqref{eq:pre-multi-corrector-stage-5-matrix-problems} critically impacts the overall numerical efficiency and robustness, especially for three-dimensional problems. Linear solvers based on algebraic factorizations (such as incomplete LU) are prone to fail due to the appearance of a zero sub-matrix $\boldsymbol{\mathrm O}$ in \eqref{eq:pre-multi-corrector-stage-5-matrix-problems}, which may lead to zero-pivoting. Hence, an iterative solution procedure for \eqref{eq:pre-multi-corrector-stage-5-matrix-problems} is specifically designed based on a nested block preconditioning technique. Readers are referred to \cite{Liu2019} for more details.

\section{Numerical results}
\label{sec:numerical_results}
In this section, we perform numerical investigations using the proposed scheme. Unless otherwise specified, we use $\mathsf p+\mathsf a+1$ Gauss quadrature points in each direction; the pressure function space is generated by the $k$-refinement to achieve the highest possible continuity.

\begin{figure}
\begin{center}
\begin{tabular}{cc}
\includegraphics[angle=0, trim=200 80 200 80, clip=true, scale = 0.15]{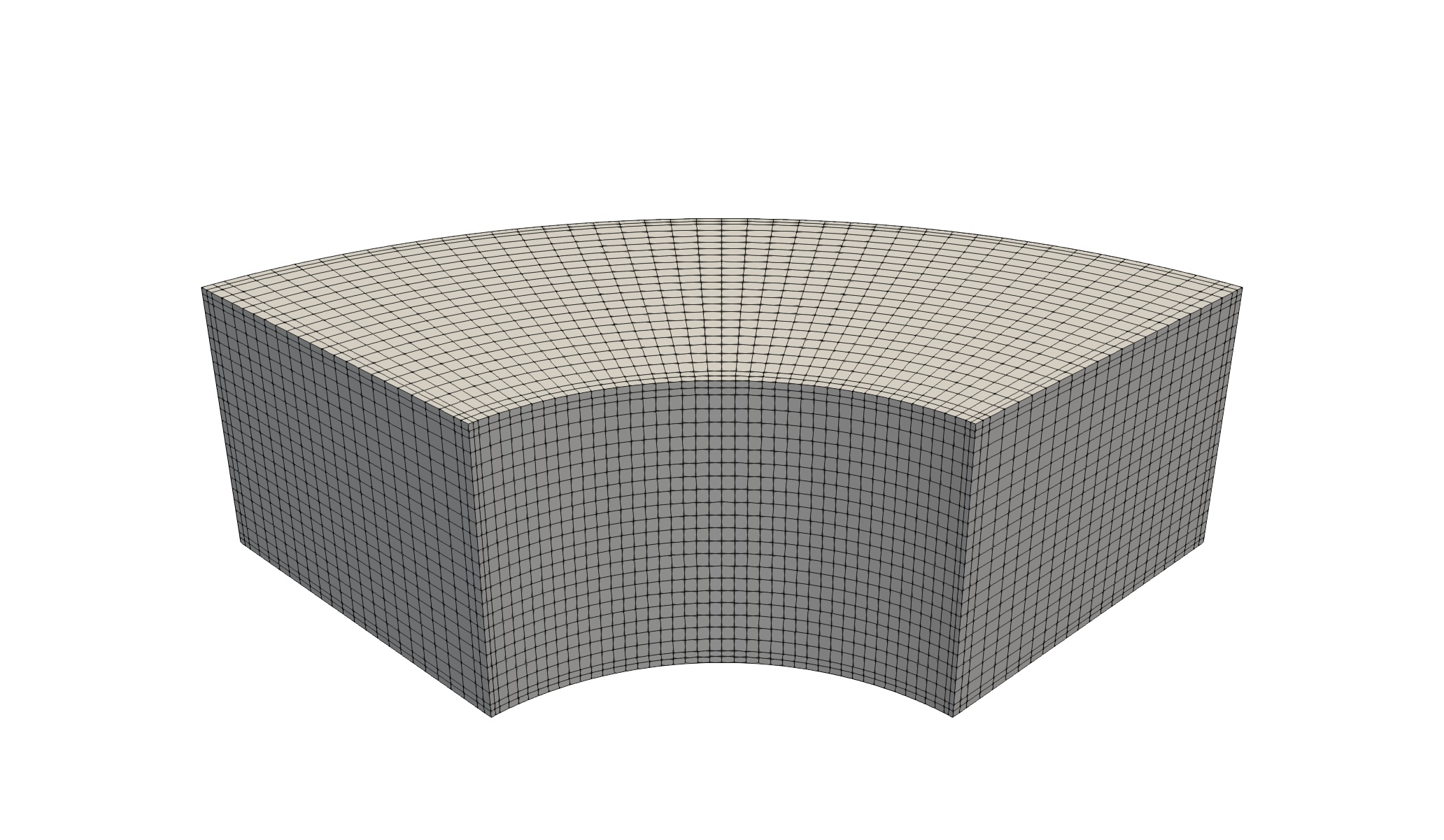} &
\includegraphics[angle=0, trim=360 680 360 380, clip=true, scale = 0.03]{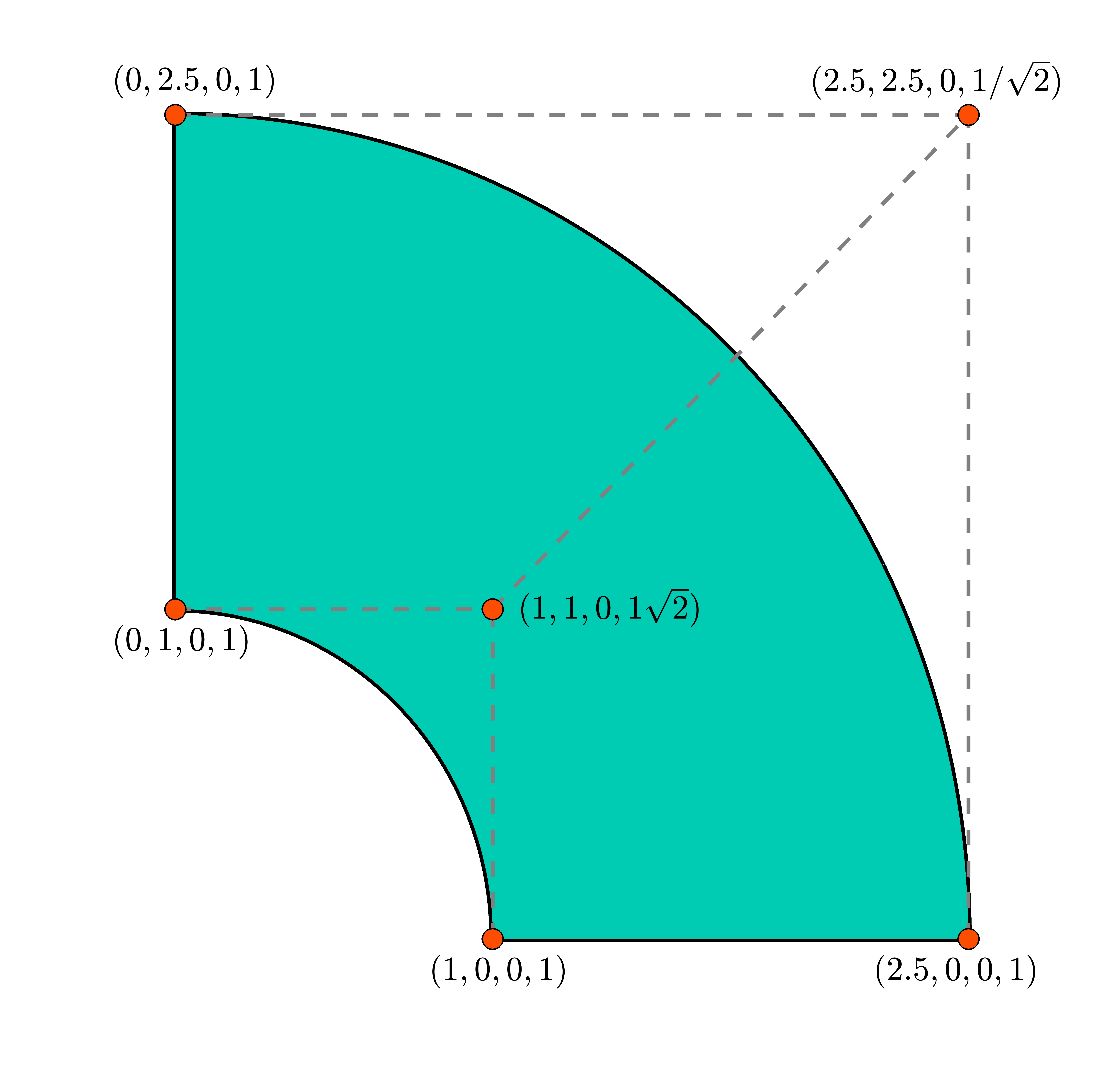}
\end{tabular}
\caption{The geometry of the thick-walled cylinder (left) and the control net with the control points' coordinates as well as weights on the bottom plane surface (right). The NURBS basis functions in the circumferential direction are built from the knot vector $\left \{ 0,0,0,1,1,1\right \}$. The NURBS basis functions in the radial and axis direction are built from the knot vector $\left \{ 0,0,1,1\right \}$.} 
\label{fig:inf_sup_geometry_setting}
\end{center}
\end{figure}

\begin{figure}
\begin{center}
\begin{tabular}{ccc}
\includegraphics[angle=0, trim=85 100 110 100, clip=true, scale = 0.35]{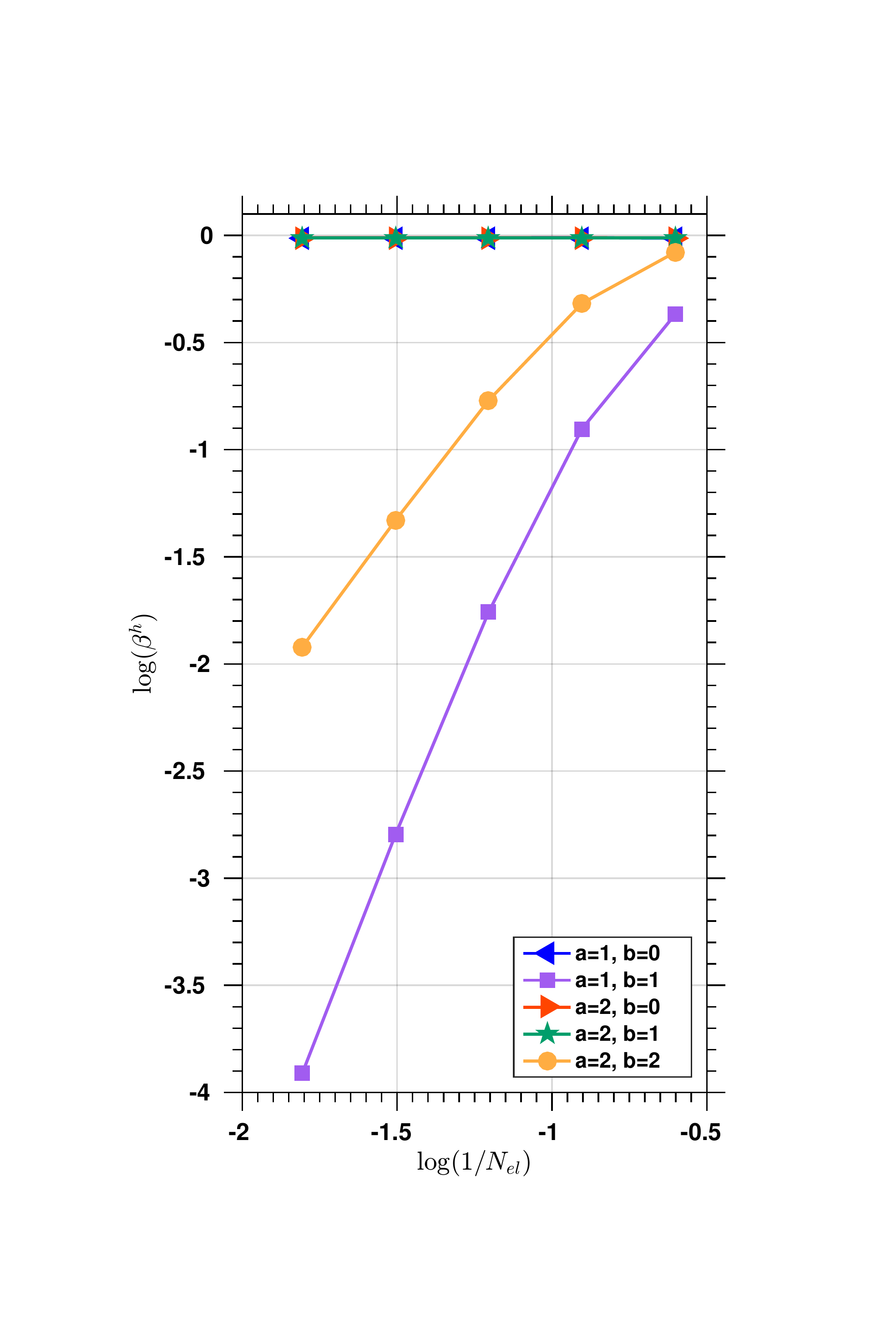} &
\includegraphics[angle=0, trim=85 100 110 100, clip=true, scale = 0.35]{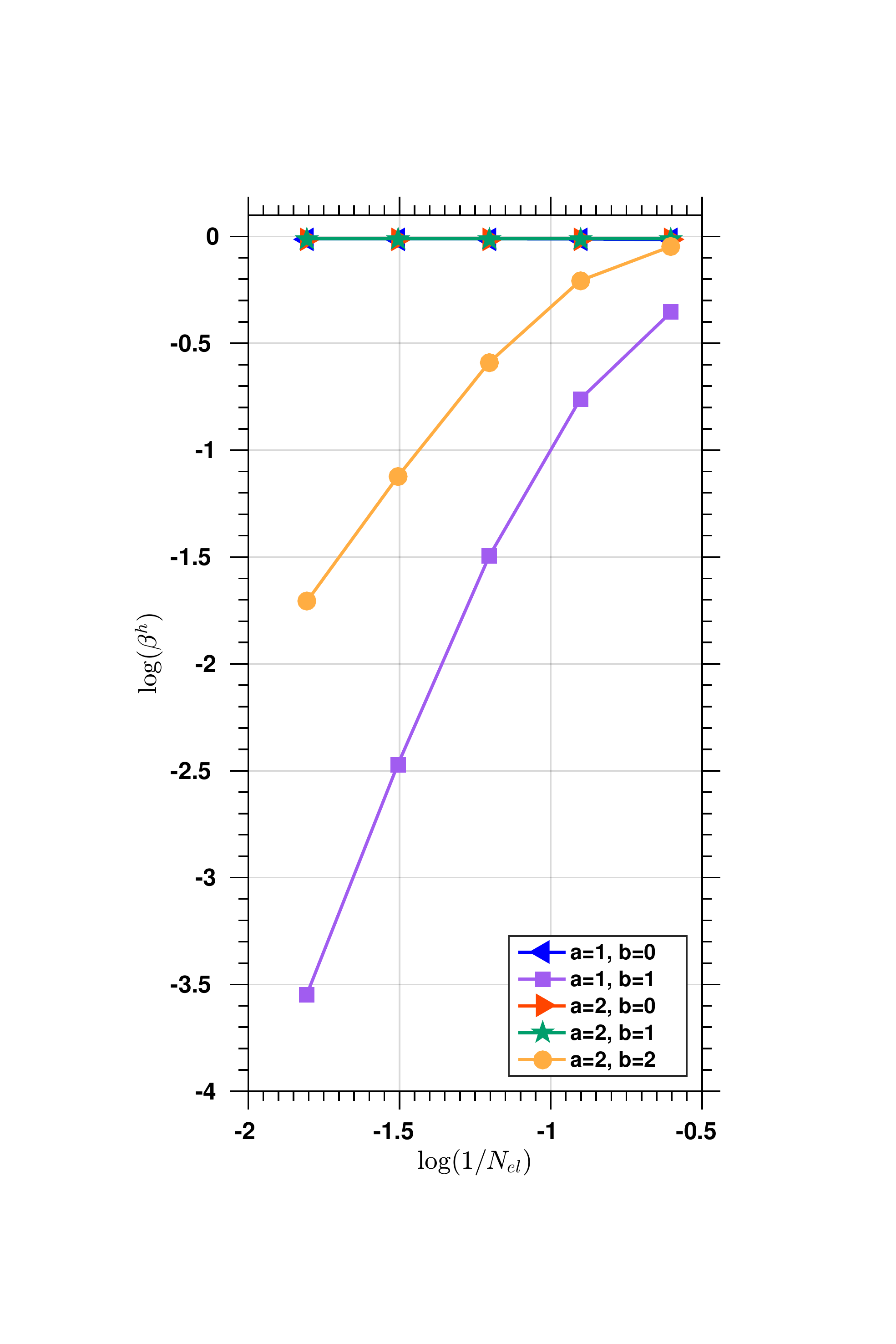} &
\includegraphics[angle=0, trim=85 100 110 100, clip=true, scale = 0.35]{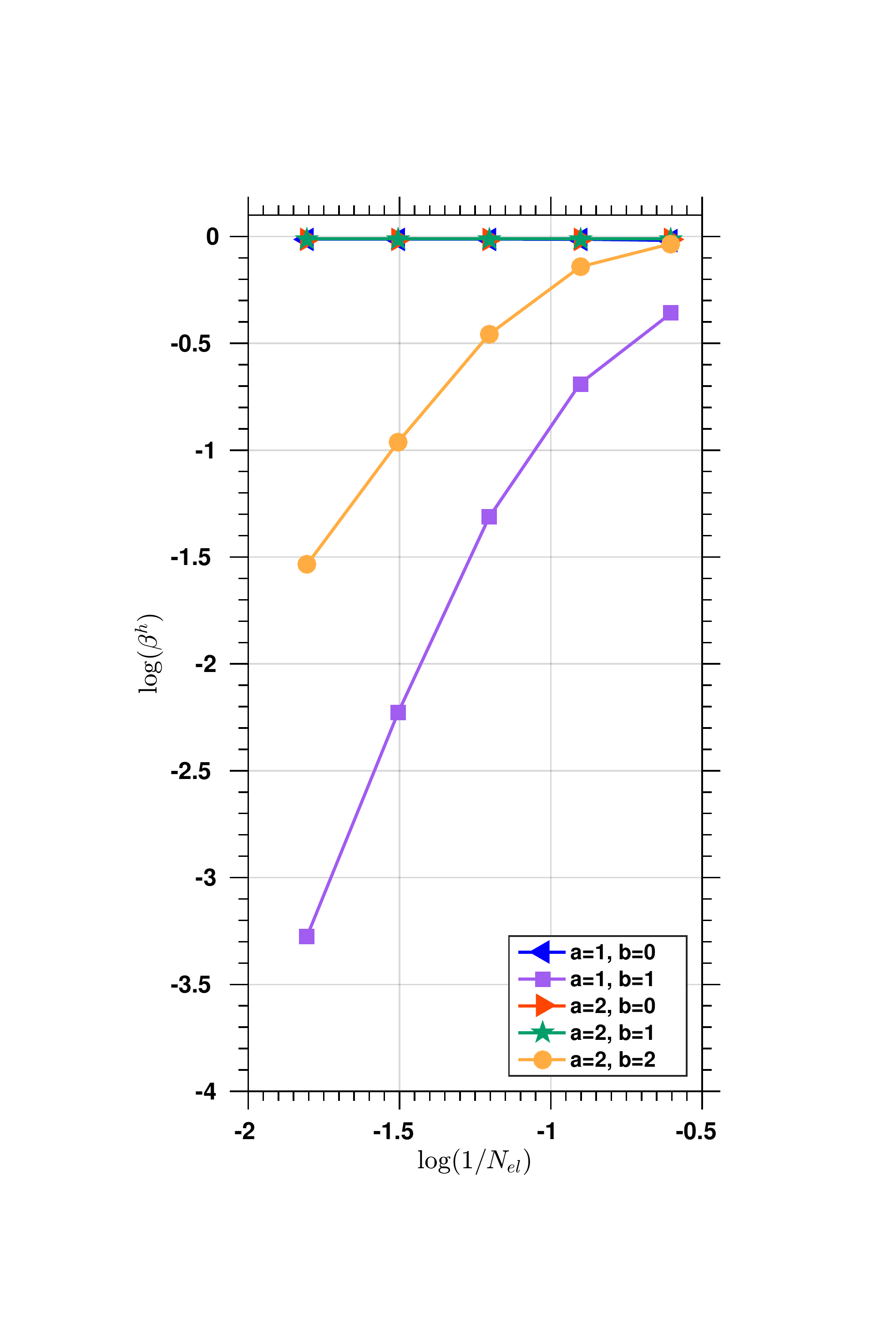} \\
(a) & (b) & (c)
\end{tabular}
\caption{The numerical inf-sup test for the thick-walled cylinder using (a) $\mathsf p=2$, (b) $\mathsf p=3$, and (c) $\mathsf p=4$ with $0\leq \mathsf b \leq \mathsf a=2$ and $N_{el}$ elements in each direction.} 
\label{fig:inf_sup_p2p3p4}
\end{center}
\end{figure}

\begin{figure}
\begin{center}
\begin{tabular}{cc}
\includegraphics[angle=0, trim=80 80 100 100, clip=true, scale = 0.5]{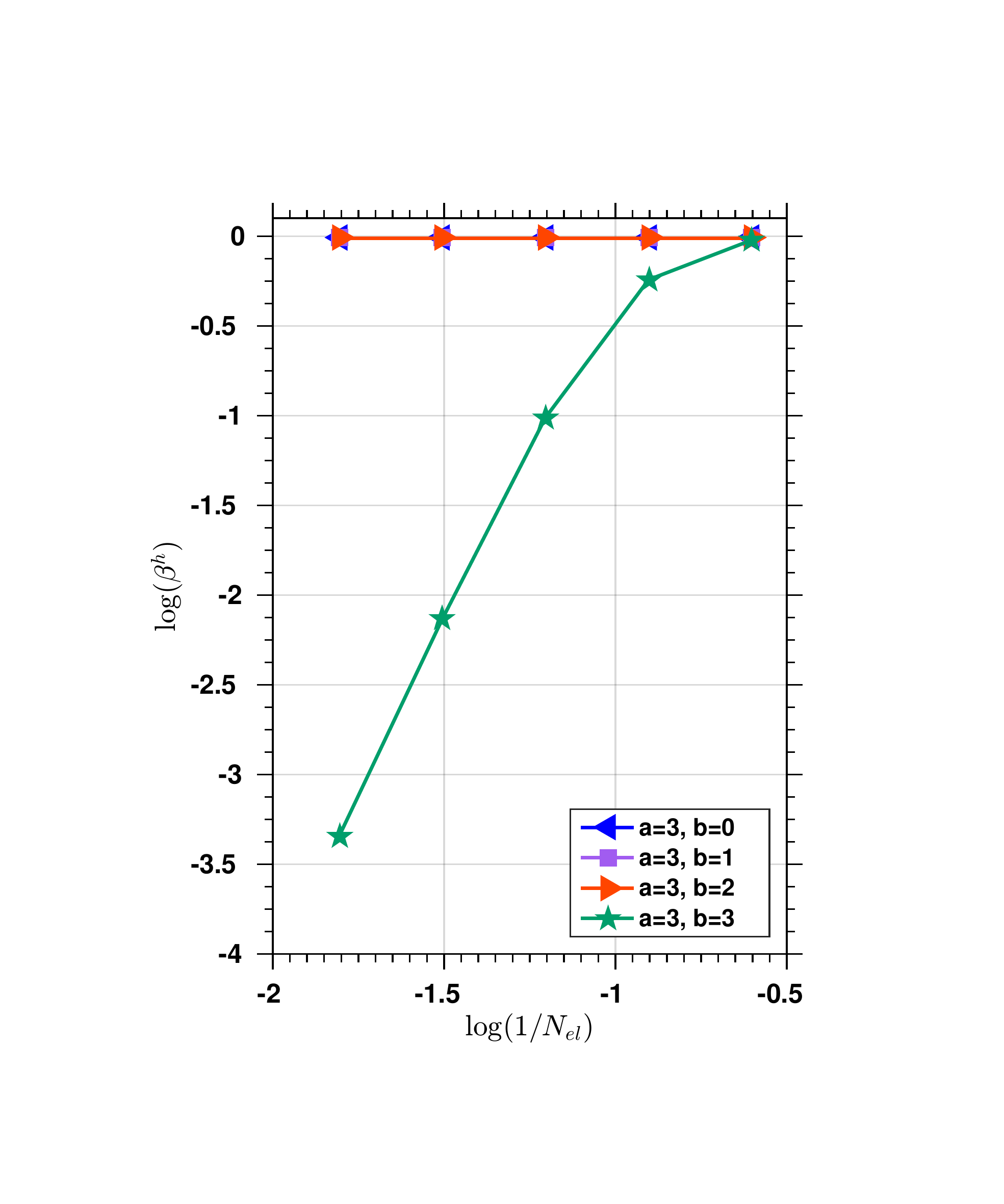} &
\includegraphics[angle=0, trim=80 80 100 100, clip=true, scale = 0.5]{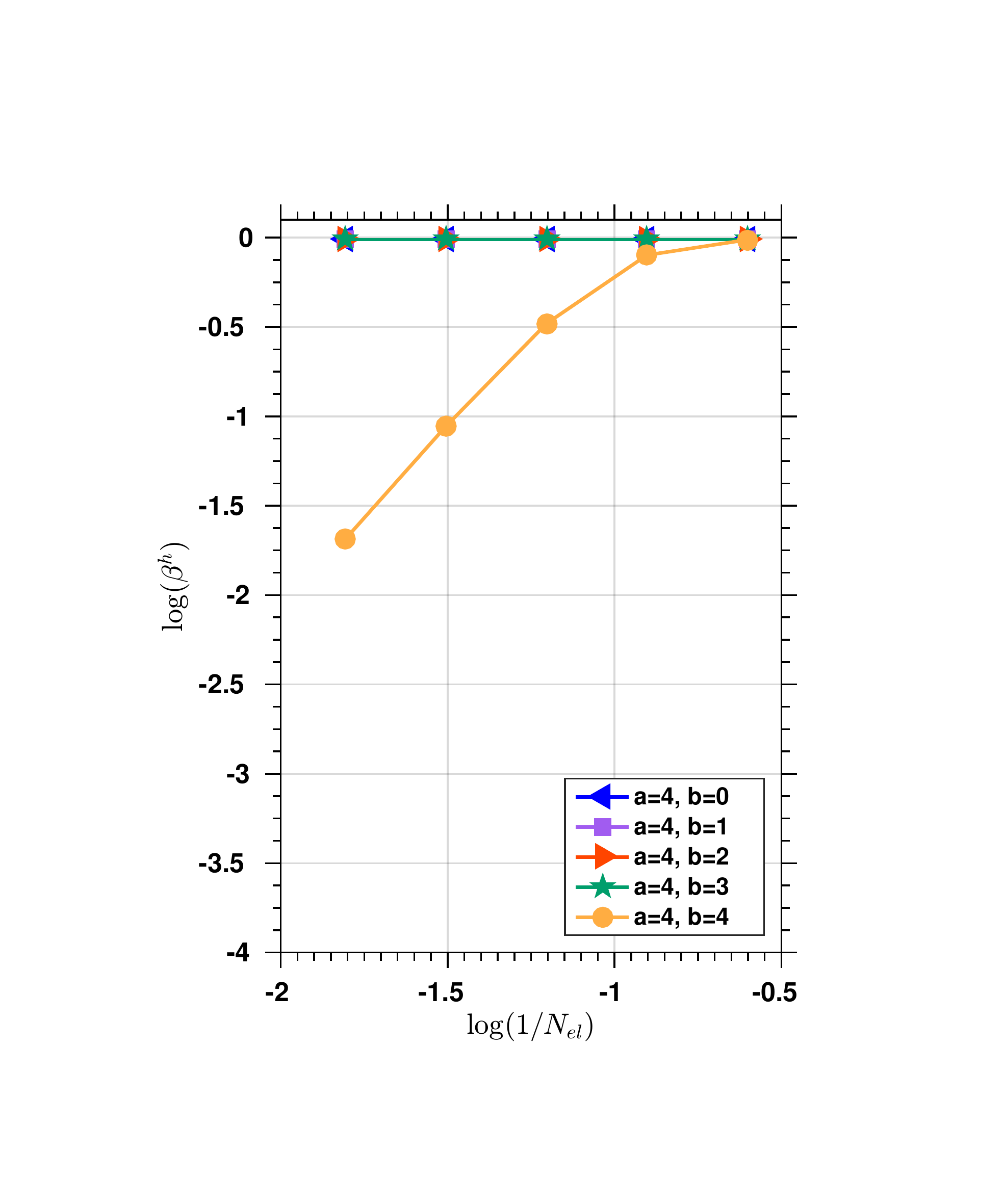} \\
(a) & (b)
\end{tabular}
\caption{The numerical inf-sup test for the thick-walled cylinder domain using $\mathsf p=2$ with (a) $0 \leq \mathsf b \leq \mathsf a=3$ and (b) $0 \leq \mathsf b \leq \mathsf a=4$ and $N_{el}$ elements in each direction.} 
\label{fig:inf_sup_p2a3a4}
\end{center}
\end{figure}

\subsection{Numerical Inf-Sup test}
The inf-sup condition for the discrete problem states that there exists a constant $\beta$ independent of the mesh size such that
\begin{align*}
\underset{q_h \in \mathcal S_{p_h}}{\textup{inf}} \underset{\bm v_h \in \mathcal S_{\bm v_h}}{ \textup{sup}}  \frac{ \int_{\Omega_{\bm x}} \nabla_{\bm x} \cdot \bm v_{h} q_h d\Omega_{\bm x} }{\|\bm v_h \|_1 \|q_h\|_0} = \beta^h \geq \beta > 0.
\end{align*}
We examine the inf-sup condition for the proposed discrete spaces $\mathcal S_{\bm v_h}$ and $\mathcal S_{p_h}$ using the numerical inf-sup test \cite{Boffi2013}. Let $N_A$ and $M_{\tilde{A}}$ denote the velocity and pressure basis functions on the current configuration where $A$ and $\tilde{A}$ are the node number. The following matrices are defined.
\begin{align*}
& \boldsymbol{\mathrm D} := \left[ \mathrm D^i_{A\tilde{B}} \right], && \mathrm D^i_{A\tilde{B}} := \int_{\Omega_{\bm x}} \nabla_{\bm x} N_A \cdot \bm e_i M_{\tilde{B}} d\Omega_{\bm x}, \\
& \boldsymbol{\mathrm W} := \left[ Q_{\tilde{A}\tilde{B}}\right], && \mathrm W_{\tilde{A}\tilde{B}} := \int_{\Omega_{\bm x}} M_{\tilde{A}} M_{\tilde{B}} d\Omega_{\bm x}, \\
& \boldsymbol{\mathrm V} := \left[ V^{ij}_{AB} \right], && \mathrm V^{ij}_{AB} := \int_{\Omega_{\bm x}} N_A N_B  + \nabla_{\bm x} N_A \cdot \nabla_{\bm x} N_B d\Omega_{\bm x} \delta_{ij}. 
\end{align*}
We consider the following eigenvalue problem: Find $\gamma^h_i$ and $\bm \psi_i$ such that
\begin{align*}
\boldsymbol{\mathrm D} \boldsymbol{\mathrm V}^{-1} \boldsymbol{\mathrm D}^T \bm \psi_i = \gamma_i^h \boldsymbol{\mathrm W} \bm \psi_i.
\end{align*}
The value of $\beta^h$ is determined as the square root of the smallest non-zero eigenvalue. The regularity vector $\bm \alpha = \{-1, \alpha, \cdots, \alpha, -1 \}$ is the same in all three directions. The numerical integration is performed by the Gauss quadrature rule with $\mathsf p+\mathsf a+2$ quadrature points in each direction to ensure accuracy. The eigenvalues are calculated by the SLEPc package \cite{Hernandez2005}. The trend of $\beta^h$ is examined as we progressively refine the mesh for $0 \leq \mathsf b \leq \mathsf a$. We consider a curved geometry for the domain, which is exactly represented by NURBS and illustrated in Figure \ref{fig:inf_sup_geometry_setting}. The computed values of $\beta^h$ for $\mathsf p=2$, $3$, and $4$ with $0 \leq \mathsf b \leq \mathsf a \leq 2$ are presented in Figure \ref{fig:inf_sup_p2p3p4}. It can be observed that $\beta^h$ approaches zero with mesh refinement when $\mathsf a=\mathsf b$. To confirm this observation, we investigate the cases of $\mathsf a=3$ and $\mathsf a=4$ with $\mathsf p$ fixed to be $2$, with results reported in Figure \ref{fig:inf_sup_p2a3a4}. Again, we observe that $\beta^h$ shows a clear trend of approaching zero with mesh refinement only when $\mathsf a=\mathsf b$. To further validate this finding, we also study a unit cube for the domain, which allows us to start the test with $\mathsf p=1$. Again, the same trend of $\beta^h$ is observed. Based on the collected results, we make the following salient observations. For the smooth generalizations of the Taylor-Hood element, if the velocity space is generated by pure $k$-refinement (i.e., $\mathsf a=\mathsf b$), the resulting element pair is \textit{not} inf-sup stable. If the velocity space is generated by pure $p$-refinement from the pressure space (i.e., $\mathsf b=0$), the smallest eigenvalues are bounded below from zero. Also, if $\mathsf a \geq 2$, the velocity spaces generated with $1 \leq \mathsf a - \mathsf b$ also pass the numerical inf-sup test. This suggests that one may still perform $k$-refinement to increase the regularity of the velocity space if it is followed by a $p$-refinement of order at least one.

\begin{figure}
\begin{center}
\begin{tabular}{cc}
\includegraphics[angle=0, trim=85 85 120 100, clip=true, scale = 0.46]{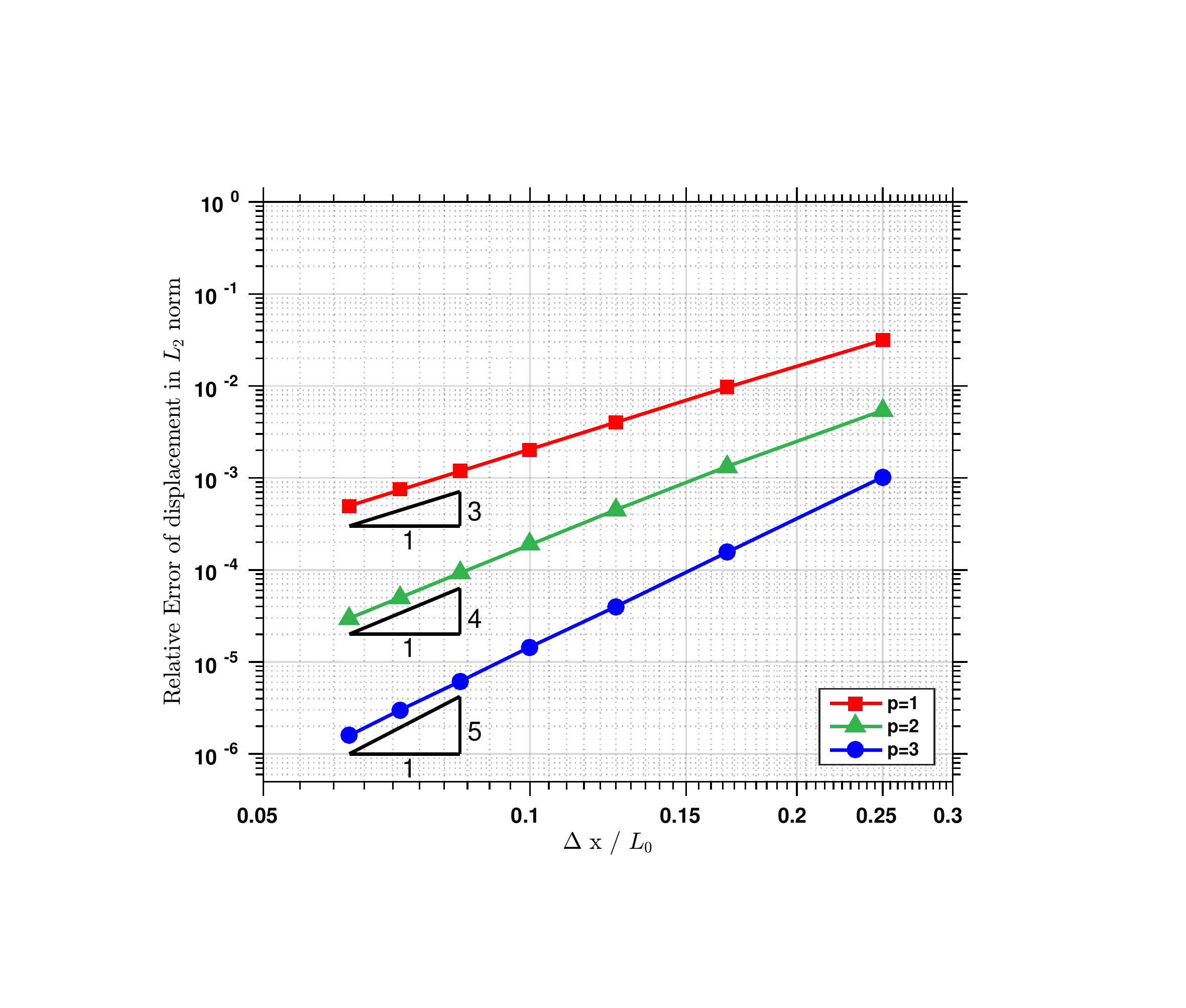} & 
\includegraphics[angle=0, trim=85 85 120 100, clip=true, scale = 0.46]{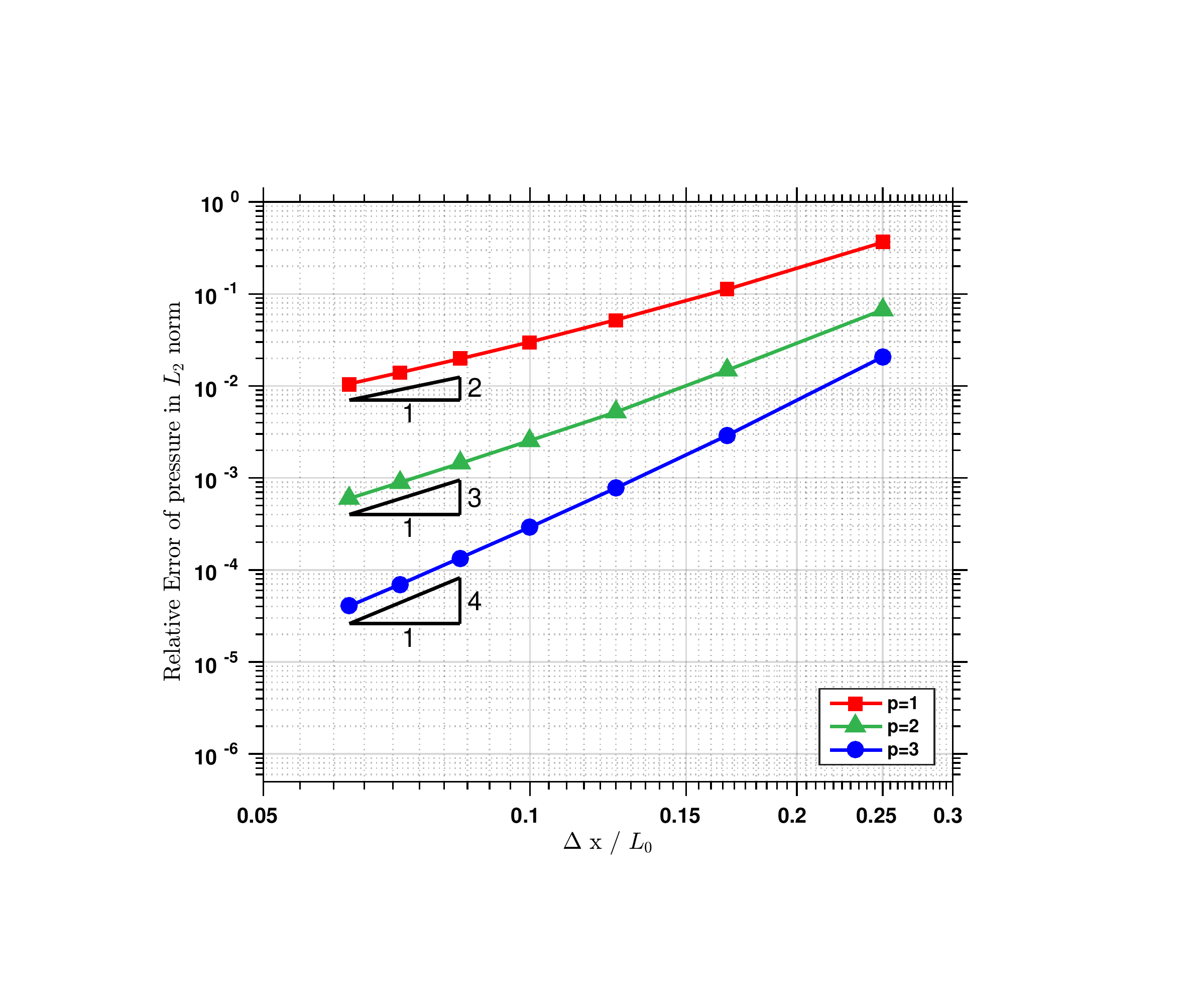} \\
(a) & (b) \\
\includegraphics[angle=0, trim=85 85 120 100, clip=true, scale = 0.46]{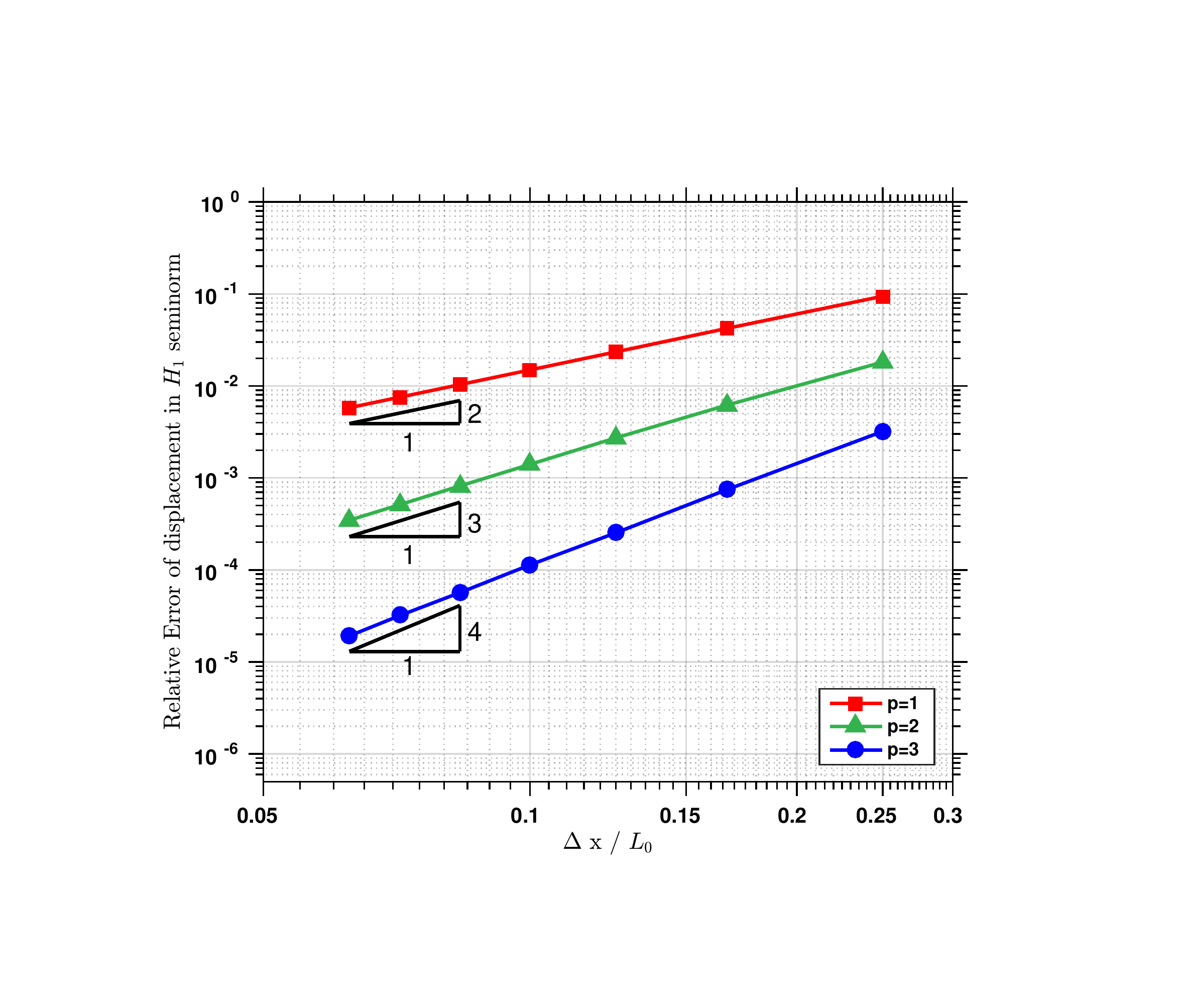} & 
\includegraphics[angle=0, trim=85 85 120 100, clip=true, scale = 0.46]{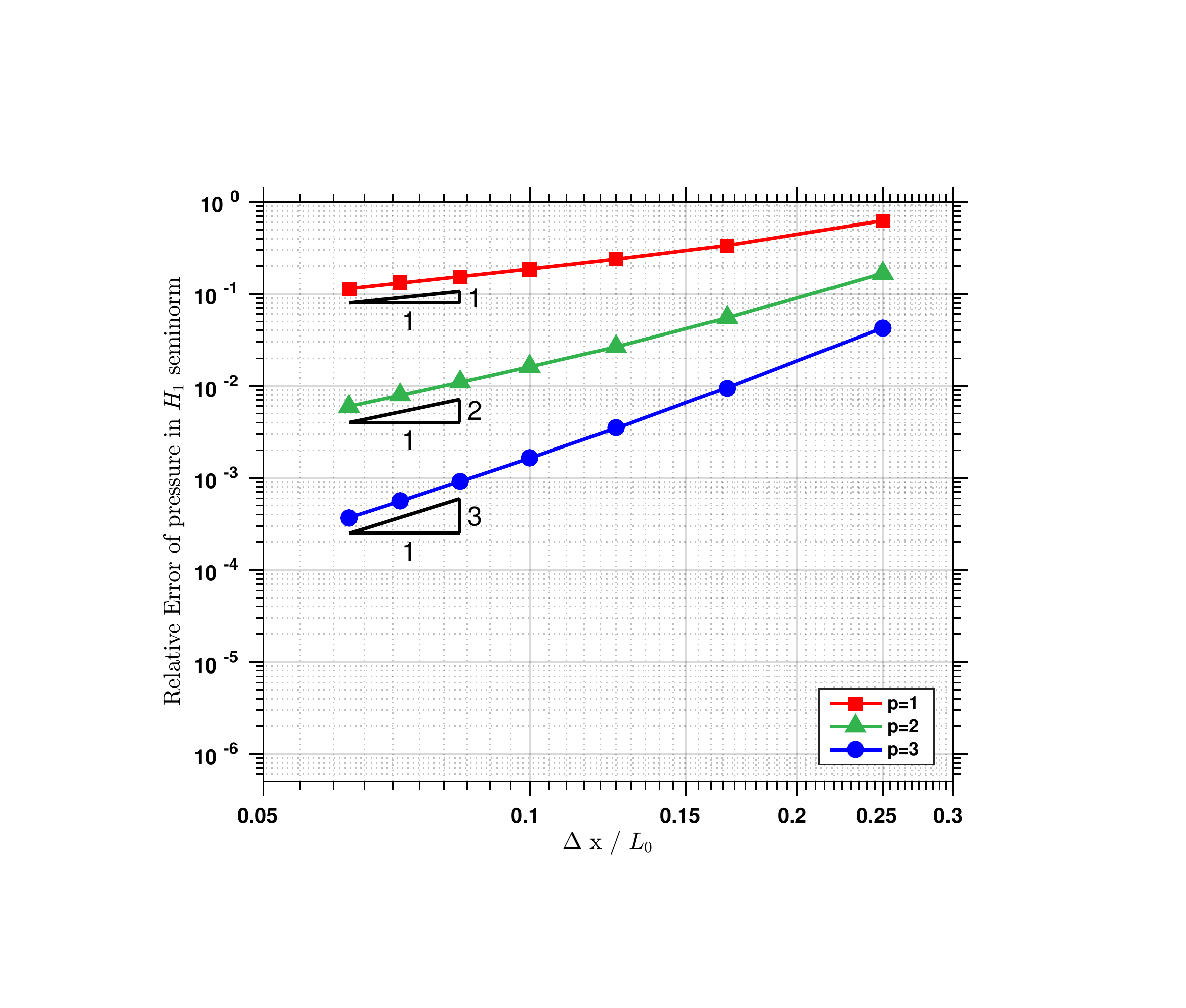} \\
(c) & (d) 
\end{tabular}
\caption{The relative errors of (a) the displacement in $L_2$ norm, (b) the pressure in $L_2$ norm, (c) the displacement in $H_1$ seminorm, and (d) the pressure in $H_1$ seminorm, under h-refinement with $\mathsf a=1$ and $\mathsf b=0$.} 
\label{fig:conv_rate_h_refine}
\end{center}
\end{figure}

\subsection{Convergence studies}
In this example, we investigate the convergence behavior of the proposed numerical scheme. We consider an incompressible Neo-Hookean material model
\begin{align*}
G(\tilde{\bm C}, p) = \frac{c_1}{2\rho_0} \left( \tilde{I}_1 - 3 \right) + \frac{p}{\rho_0}.
\end{align*}
The geometrical domain is a unit cube with dimension 1m $\times$ 1m $\times$ 1m. The modulus $c_1$ is chosen as $1$ Pa, and the density $\rho_0$ is 1 kg/m$^3$. The analytic forms of the displacement and pressure fields on the referential configuration adopt the following forms,
\begin{align*}
\bm U(\bm X,t) = c \frac{L_0}{T_0^2} t^2 \begin{bmatrix}
\sin(\gamma \frac{Y}{L_0}) \sin(\gamma \frac{Z}{L_0}) \\
0 \\
0
\end{bmatrix}, \quad P(\bm X,t) &= d\frac{M_0}{L_0T_0^4}t^2 \sin(\beta \frac{X}{L_0}) \sin(\beta \frac{Y}{L_0}) \sin(\beta \frac{Z}{L_0}).
\end{align*}
In this example, the reference values are chosen as $L_0 = 1$ m, $M_0 = 1$ kg, $T_0 =1$ s; both $\beta$ and $\gamma$ are chosen to be $2\pi$ rad; $c$ and $d$ are non-dimensional parameters that take the value 0.2. On the faces $Y=Z=0$ m and $Y=Z=1$ m, the body is fully clamped, and traction boundary conditions are applied on the rest faces. For the simulations, we use $\textup{tol}_{\textup{R}} = 10^{-10}$ and $\textup{tol}_{\textup{A}} = 10^{-12}$ as the stopping criteria in the predictor multi-corrector algorithm. Two different time step sizes are used to ensure that the temporal error does not pollute the spatial convergence rate. The relative errors of the displacement and pressure fields are reported in Figure \ref{fig:conv_rate_h_refine} for varying values of $\mathsf p$ with $\mathsf a=1$, $\mathsf b=0$. We notice immediately that all the errors decrease with the optimal rates. In Figure \ref{fig:conv_rate_h_refine_a2}, we report the convergence rates for $\mathsf a = 2$, which resembles a smooth generalization of the spectral element \cite{Yu2012}. In Figure \ref{fig:conv_rate_h_refine_a2} (a), we note that the increase of the value of $\mathsf a$ does not improve the convergence rate, regardless of the value of $\mathsf b$. Yet, the velocity error is smaller than that of the $\mathsf a = 1$ case. From Figure \ref{fig:conv_rate_h_refine_a2} (b), we can see that the pressure errors are almost indistinguishable for $\mathsf a = 1$ and $\mathsf a = 2$.

\begin{figure}
\begin{center}
\begin{tabular}{cc}
\includegraphics[angle=0, trim=85 85 120 100, clip=true, scale = 0.46]{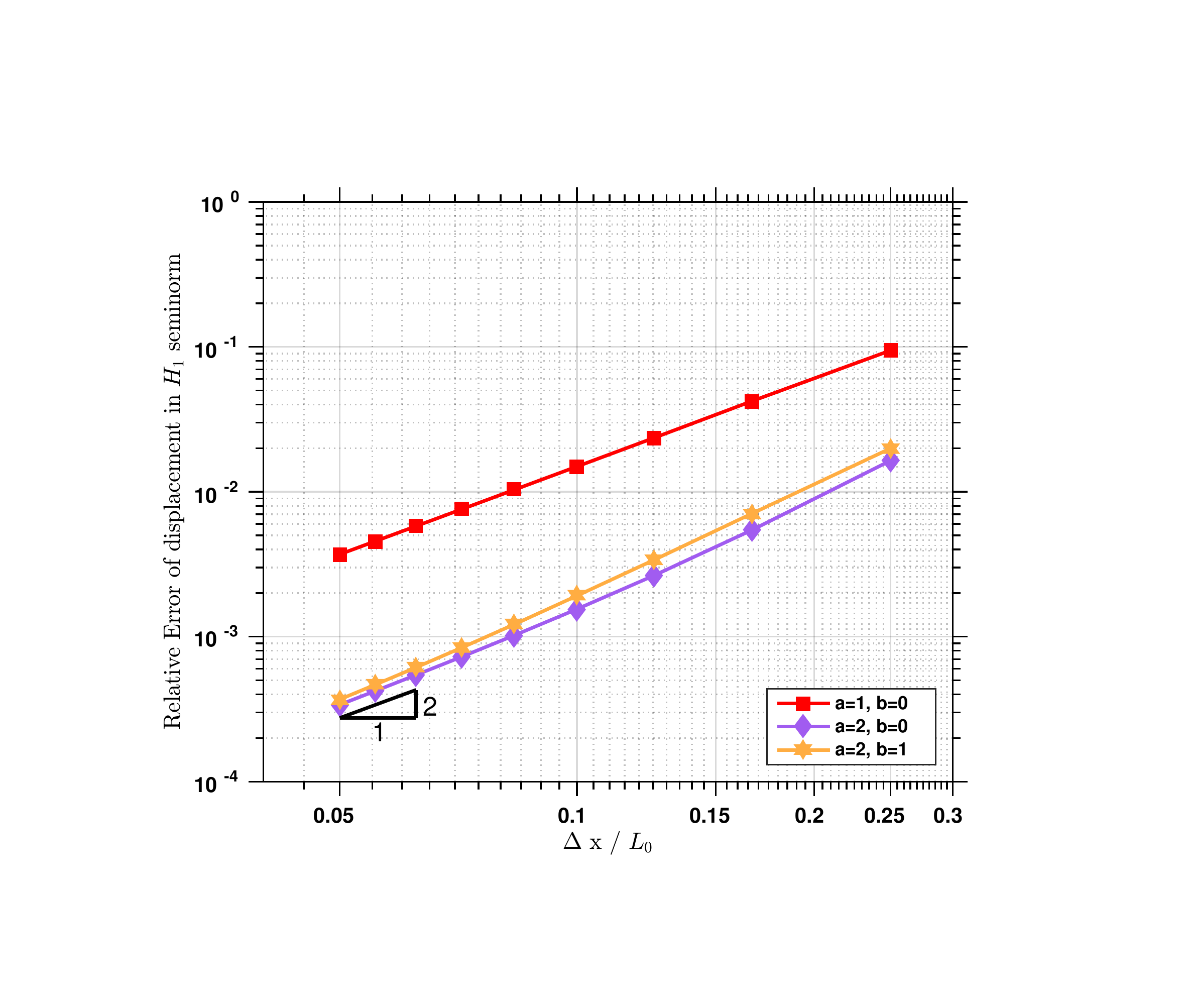} & 
\includegraphics[angle=0, trim=85 85 120 100, clip=true, scale = 0.46]{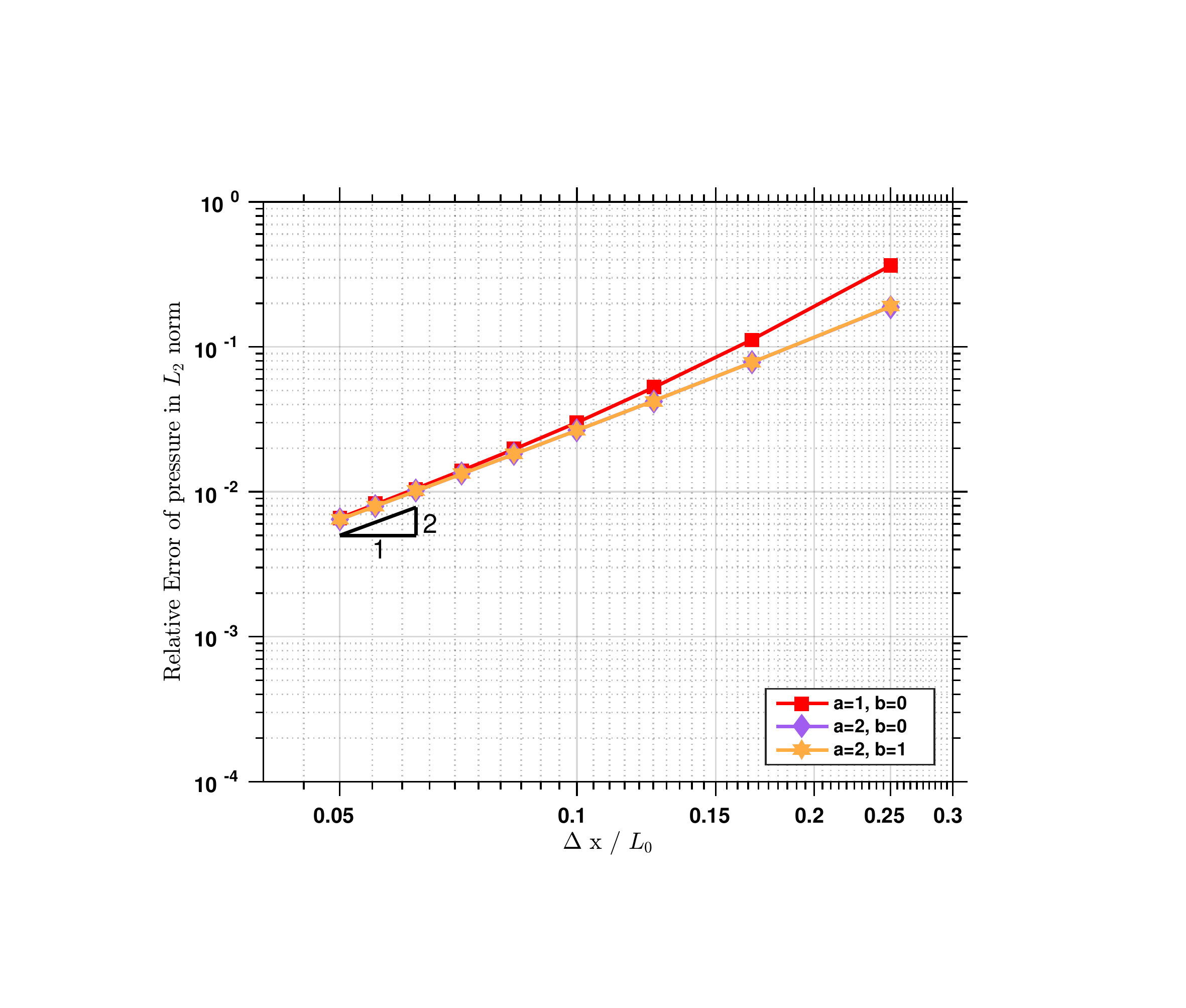} \\
(a) & (b) 
\end{tabular}
\caption{The relative errors of (a) the displacement in $H_1$ seminorm and (b) the pressure in $L_2$ norm, under h-refinement with $p=1$ and varying values of $\mathsf a$ and $\mathsf b$.} 
\label{fig:conv_rate_h_refine_a2}
\end{center}
\end{figure}

\begin{figure}
\begin{center}
\begin{tabular}{cc} 
\includegraphics[angle=0, trim=90 200 180 280, clip=true, scale = 0.6]{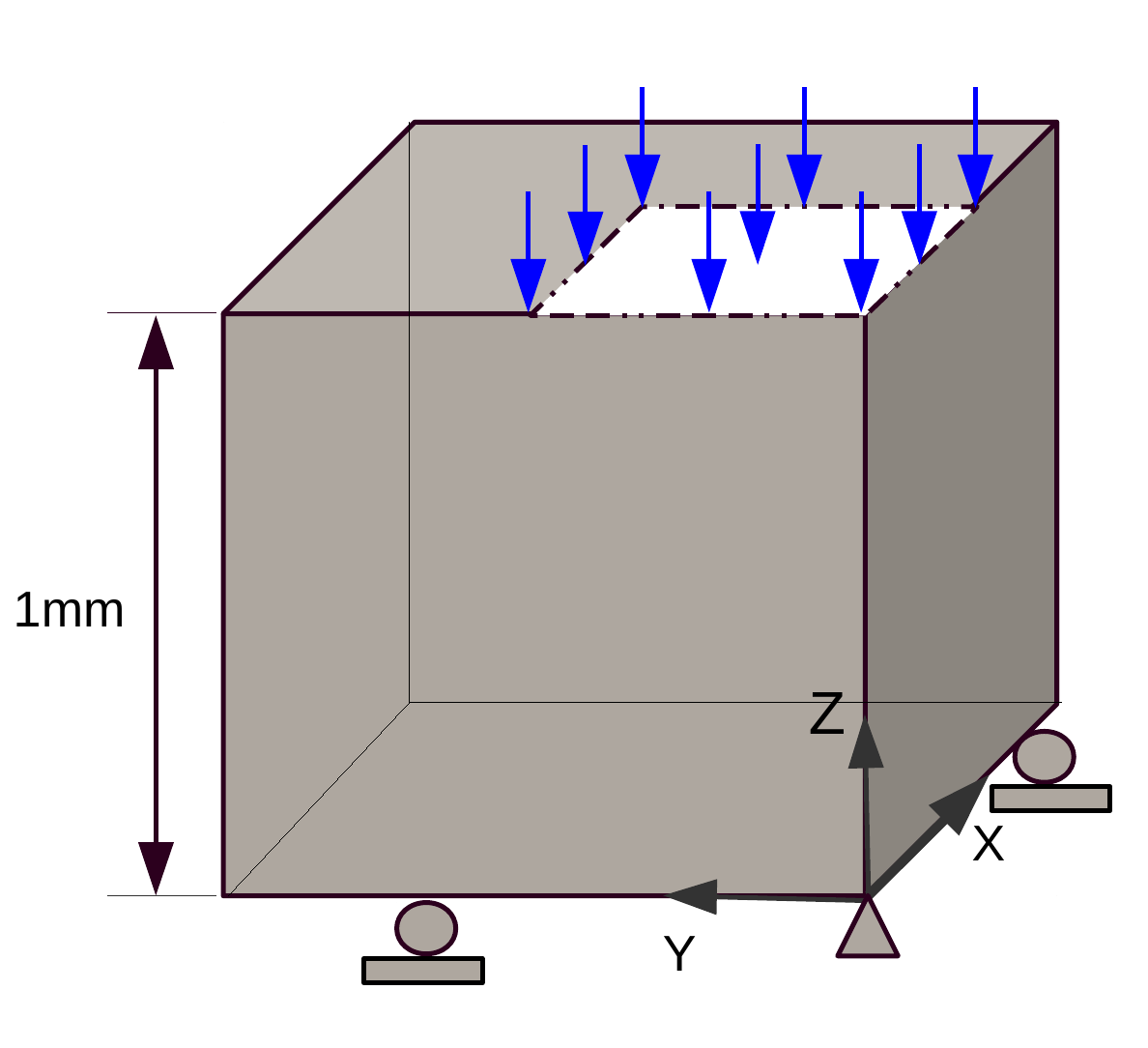} &
\includegraphics[angle=0, trim=100 80 100 80, clip=true, scale = 0.4]{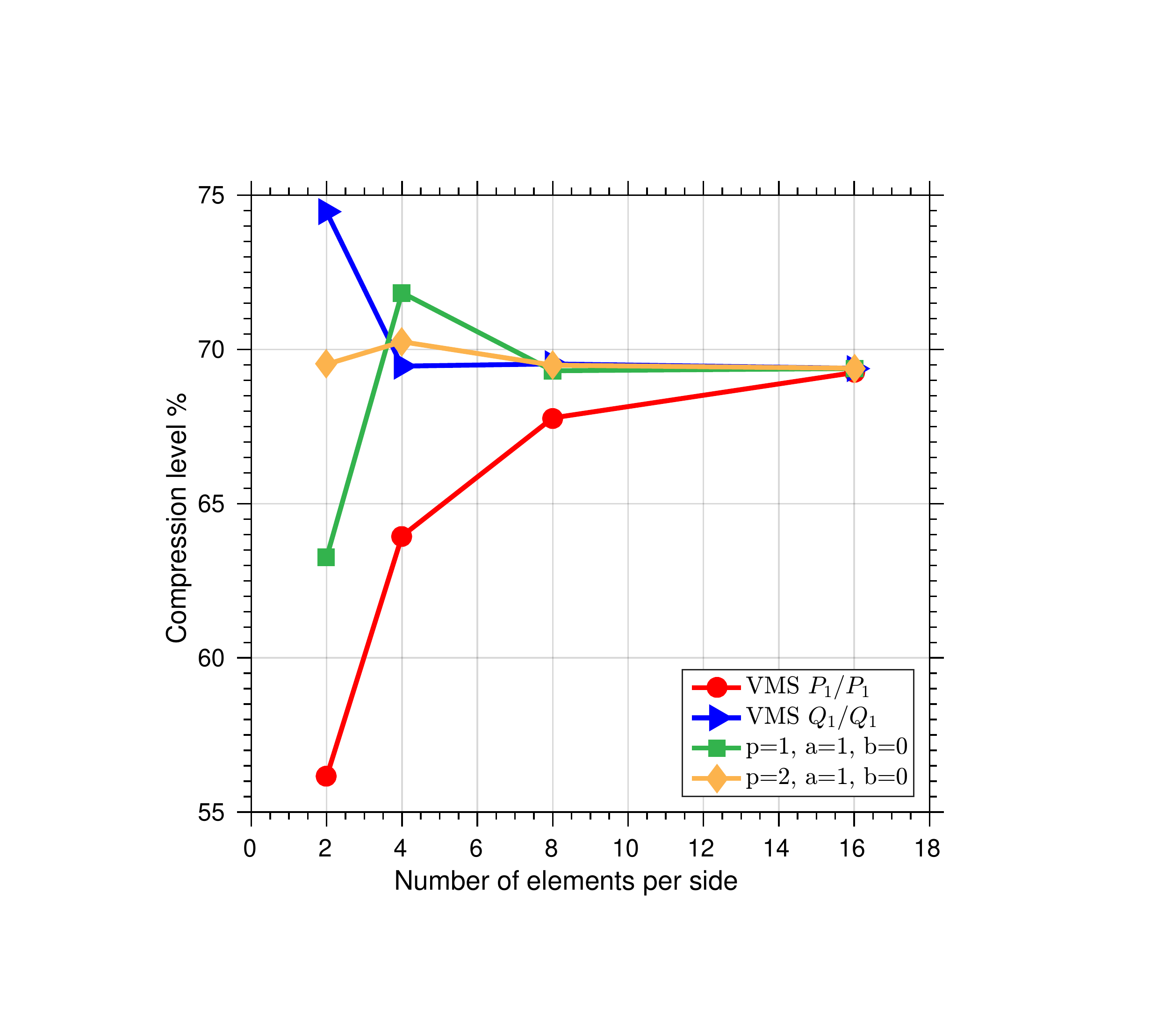} \\
(a) & (b)
\end{tabular}
\caption{Three-dimensional block compression: (a) geometry setting; (b) compression level in \% versus the number of elements per side.} 
\label{fig:block_compression_tip_disp}
\end{center}
\end{figure}

\subsection{Three-dimensional compression of a block}
In this example, we examine the performance of the new formulation using the benchmark problem initially designed in \cite{Reese2000}. On the boundary faces $X=Y=Z=0$, we apply symmetry boundary conditions, and we disallow horizontal displacement on the top surface. A `dead' load with magnitude $3.2 \times 10^8$ Pa is applied on a quarter portion of the top surface, which assumes the negative Z-direction in the referential configuration. The block is initially stress free with zero displacement. The surface traction load is applied as a linear function of time and reaches the prescribed magnitude at time $T=1$ s. We adopt an incompressible Neo-Hookean model given by the following energy function,
\begin{align*}
G(\tilde{\bm C}, p) = \frac{c_1}{2\rho_0} \left( \tilde{I}_1 - 3 \right) + \frac{p}{\rho_0}.
\end{align*}
The material properties are chosen as $\rho_0 = 1.0 \times 10^3$ kg/m$^3$ and $c_1 = 8.0194 \times 10^7$ Pa. We simulate the problem with a fixed time step size $\Delta t = 5.0\times 10^{-3}$ s. We fix the values of $\mathsf a$ and $\mathsf b$ to $1$ and $0$ in this example and choose $\textup{tol}_{\textup{R}} = 10^{-3}$ and $\textup{tol}_{\textup{A}} = 10^{-6}$ as the stopping criteria. For comparison purposes, we also simulate the problem with the variational multiscale (VMS) formulation \cite{Liu2018} using equal-order interpolations. As a classical benchmark problem, the primary quantity of interest is the displacement at the upper center point (i.e. the point at $X=Y=0$, $Z=1$ in the reference configuration). In Figure \ref{fig:block_compression_tip_disp} (b), the compression levels at this point calculated by different methods are illustrated. For the coarsest mesh (two elements per side), the stable element with $\mathsf p = 2$ gives a very good prediction of the compression level. It is interesting to note that the equal-order interpolation using the $Q_1/Q_1$ element with the VMS formulation gives a fairly good result for a finer mesh with four elements per side. Using the same mesh, the stable elements with $\mathsf p=1$ and $\mathsf p=2$ gives slightly softer predictions, which is due to oscillations of the higher-order methods at the tip. Using a mesh with eight elements per side, both stable elements give indistinguishable results in comparison with the reference value. In Figure \ref{fig:block_compression_pressure}, we further compare the pressure profiles at the current configuration calculated by a coarse mesh (two elements per side) with the value of $\mathsf p$ varying from $2$ to $6$. The pressure profile calculated by a fine mesh (16 elements per side and $\mathsf p = 2$) is depicted to serve as a reference solution profile. It can be observed that the increase of the polynomial degree $\mathsf p$ improves of the solution quality. For the case of $\mathsf p = 6$, the calculated result essentially captures the major feature of the pressure field.

\begin{figure}
\begin{center}
\begin{tabular}{cc} 
\includegraphics[angle=0, trim=250 0 300 200, clip=true, scale = 0.16]{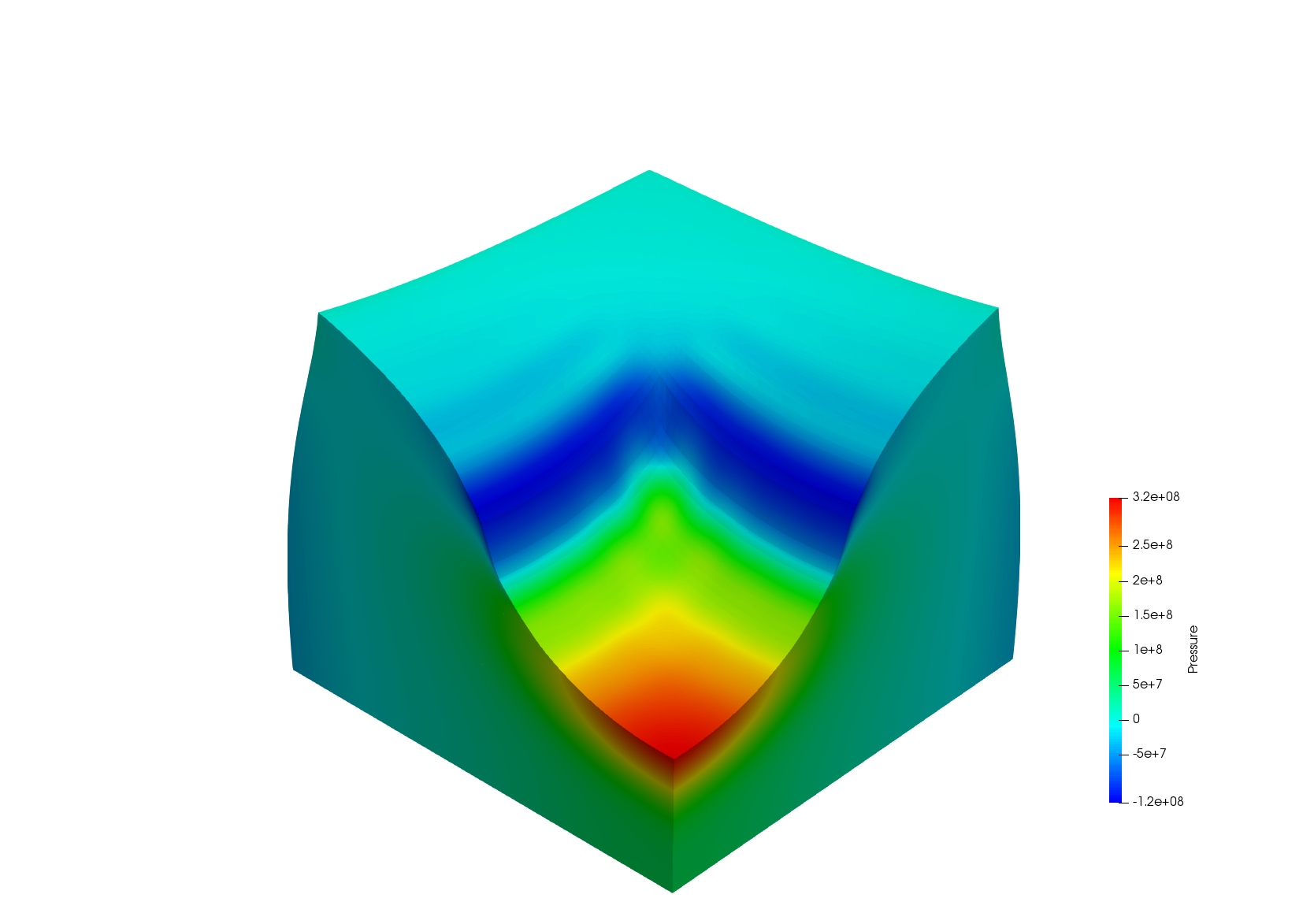} &
\includegraphics[angle=0, trim=250 0 300 200, clip=true, scale = 0.16]{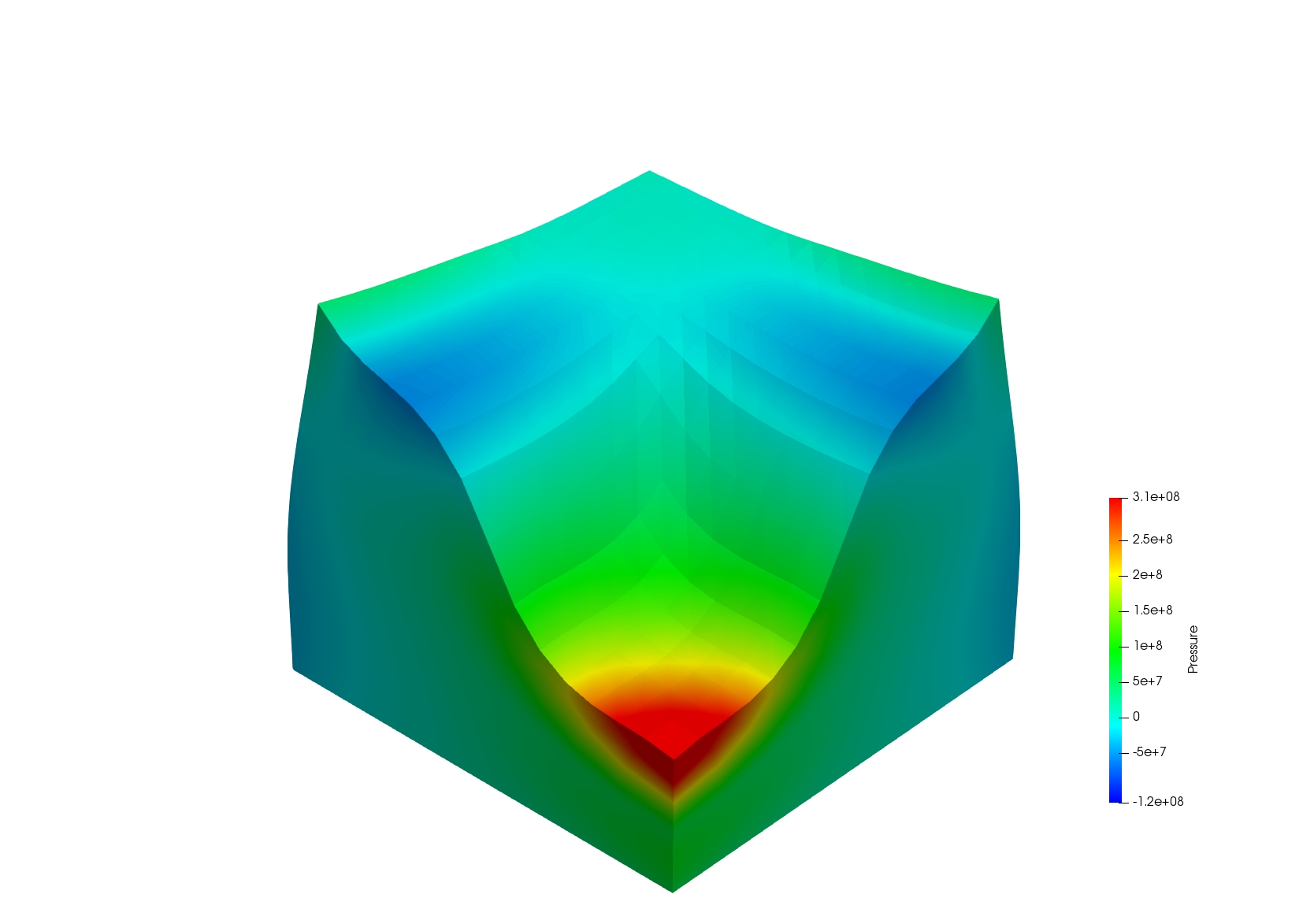} \\
(a) & (b) \\
\includegraphics[angle=0, trim=250 0 300 150, clip=true, scale = 0.16]{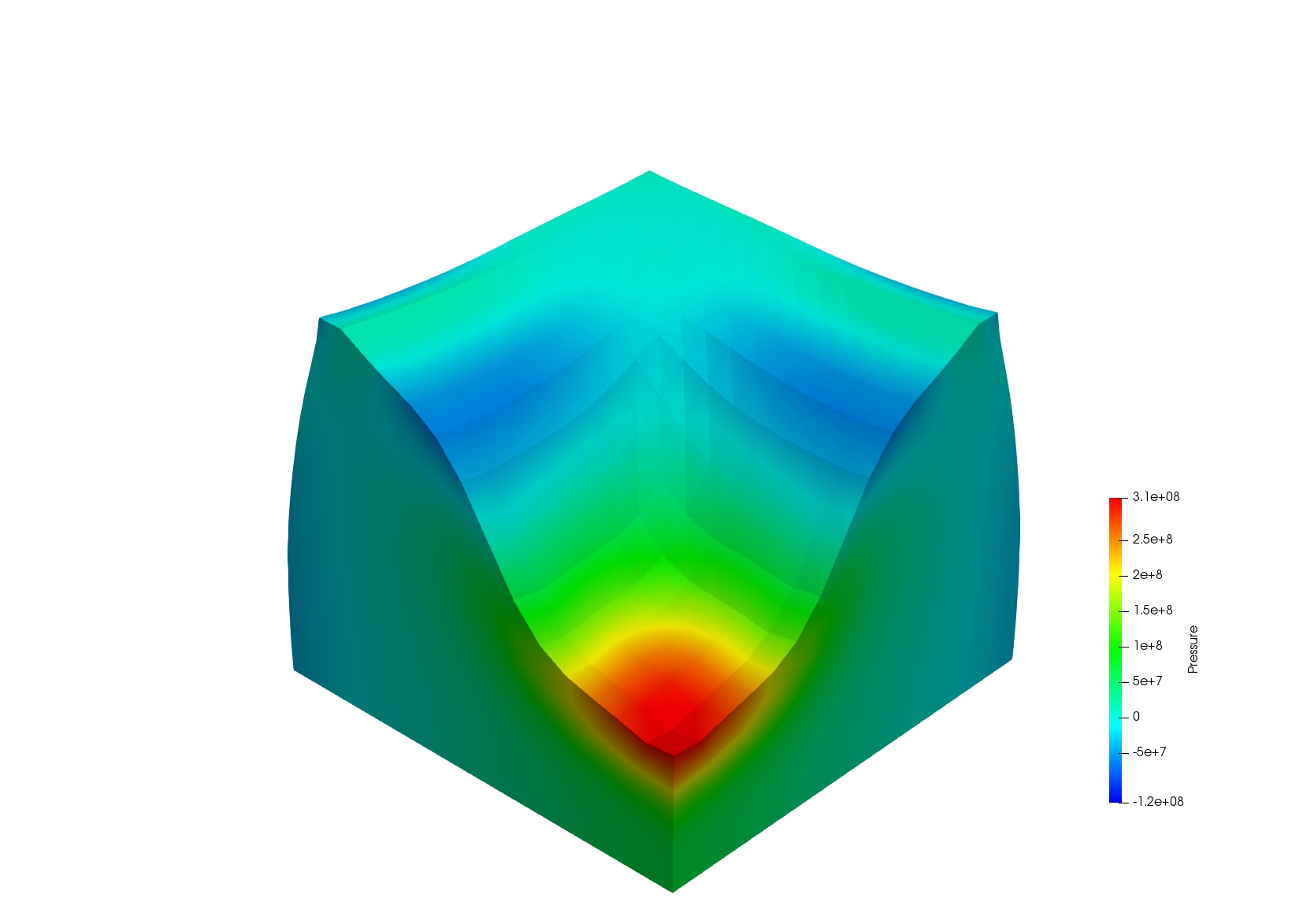} & 
\includegraphics[angle=0, trim=250 0 300 150, clip=true, scale = 0.16]{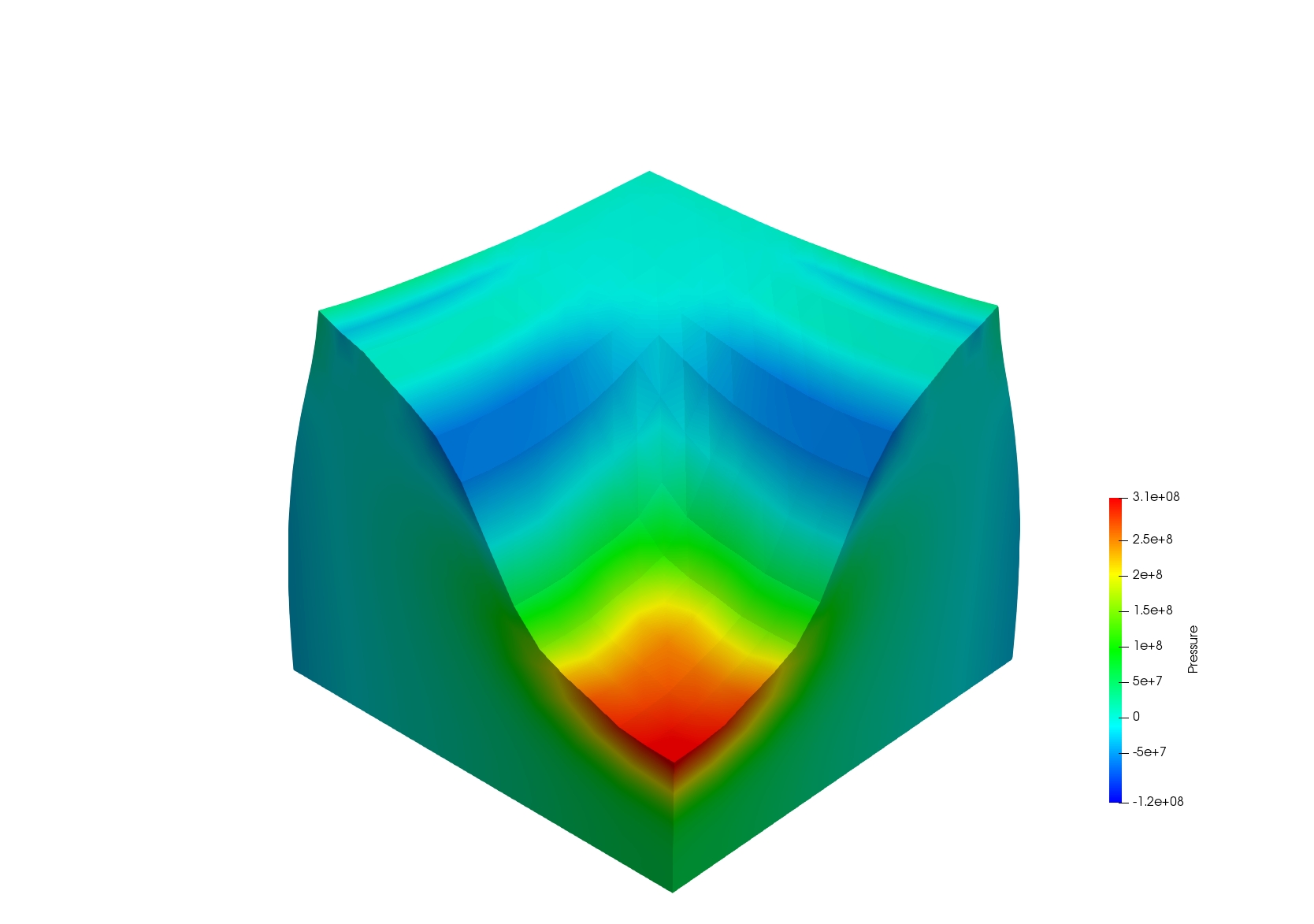} \\
(c) & (d) \\
\multicolumn{2}{c}{
\includegraphics[angle=0, trim=0 1050 700 10, clip=true, scale = 0.3]{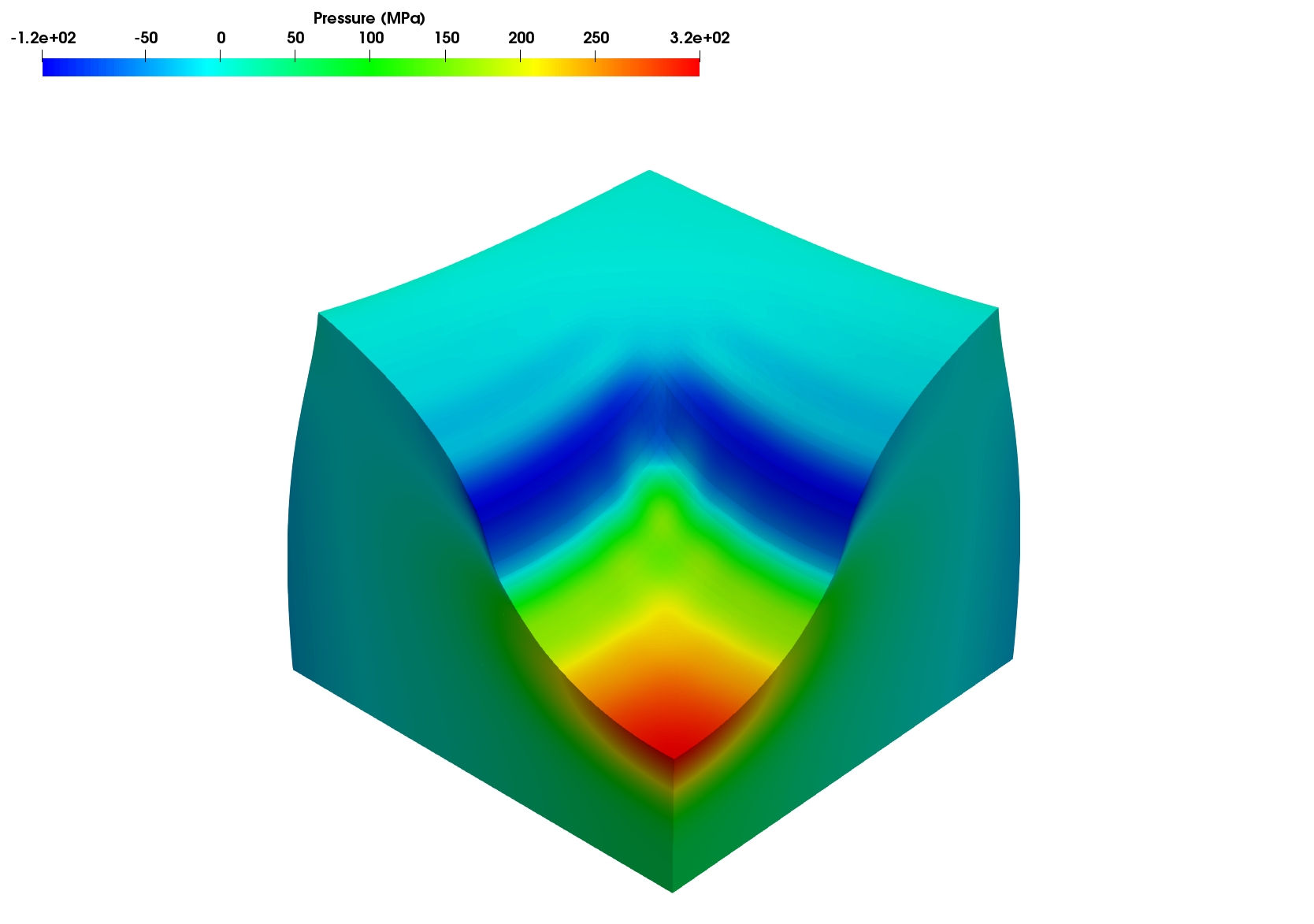}
}
\end{tabular}
\caption{Three-dimensional block compression: pressure profile in the current configuration with (a) $p=2$, $\Delta x = 1/16$, (b) $p=2$, $\Delta x = 1/2$, (c) $p=4$, $\Delta x = 1/2$, and (d) $p=6$, $\Delta x = 1/2$.} 
\label{fig:block_compression_pressure}
\end{center}
\end{figure}

\subsection{Tensile test of an anisotropic fiber-reinforced hyperelastic soft tissue specimen}
\label{subsec:tensile_test}
In this example, we examine the performance of the proposed formulation for an anisotropic hyperelastic material, which has been designed to describe arterial tissue layers with distributed collagen fibers \cite{Gasser2006}. The geometry set-up and the material model are summarized in Table \ref{table:3d_tensile_benchmark_geometry}. The groundmatrix is modeled as an isotropic Neo-Hookean material, with $c_1$ being the shear modulus. The $i$th family of collagen fibers is modeled by an exponential function $G^{f_i}_{ich}$. The unit vector $\bm a_i$ characterizes the mean orientation of the fiber, and $\kappa_d$ is a dispersion parameter that characterizes the distribution of the collagen fibers. In this study, we assume the mean orientation of the two families of fibers has no component in the radial direction and is completely determined by $\varphi$, the angle between the fiber orientation and the loading direction. For a circumferential specimen, the tensile load is along the circumferential direction and $\varphi = 49.98^\circ$; correspondingly, for an axial specimen, the value of $\varphi$ is $40.02^\circ$. We consider only one-eighth of the specimen by applying symmetry boundary conditions. On the loading surface, a master-slave relation is enforced for the nodes to ensure that the surface moves only in the loading direction. The loading traction is applied gradually and reaches $2$ N in 200 seconds. We simulate the problem with a fixed time step size $\Delta t = 2.0\times 10^{-2}$ s. Again, we fix the value of $\mathsf a$ and $\mathsf b$ to be $1$ and $0$ and use $\textup{tol}_{\textup{R}} = 10^{-3}$ and $\textup{tol}_{\textup{A}} = 10^{-6}$ as the stopping criteria. Three different meshes are used for the proposed formulation: mesh 1 consists of $61440$ elements with $\mathsf p=2$, mesh 2 consists of $120$ elements with $\mathsf p=1$, and mesh 3 consists of $120$ elements with $\mathsf p=2$. In Figure \ref{fig:tensile_stress_strain}, the load-displacement curves calculated by the three different meshes for the circumferential and axial specimen are plotted. It is hard to distinguish the results in Figure \ref{fig:tensile_stress_strain} (a). In Figure \ref{fig:tensile_stress_strain} (b), we provide a detailed comparison near the tensile load $0.35$ N. The curve obtained from mesh 3 is still very close to the reference solid line, indicating improved accuracy with increasing polynomial degree. For comparison purposes, we present the stress results calculated by the VMS formulation \cite{Liu2018} with linear tetrahedral elements using two different spatial resolutions (see Table 2). From Figures \ref{fig:tensile_test_stress_axial} and \ref{fig:tensile_test_stress_circum}, we observe that the essential feature of the Cauchy stress is captured in mesh 2, although there are slight oscillations near the corners. The results calculated from the mesh 1 and mesh 3 are almost indistinguishable, indicating that increasing the polynomial degree improves the accuracy of the stress results. In contrast, the stress is poorly resolved in mesh 4 due to the low-order elements. The results of mesh 5 illustrate that mesh refinement helps improve the quality of the stress results. Yet, one can still observe a discontinuous pattern and oscillations of the stress profile.

\begin{table}[htbp]
  \centering
  \begin{tabular}{ m{.38\textwidth}   m{.56\textwidth} }
    \hline
    \begin{minipage}{.38\textwidth}
      \includegraphics[angle=0, trim=40 60 40 20, clip=true, scale = 0.33]{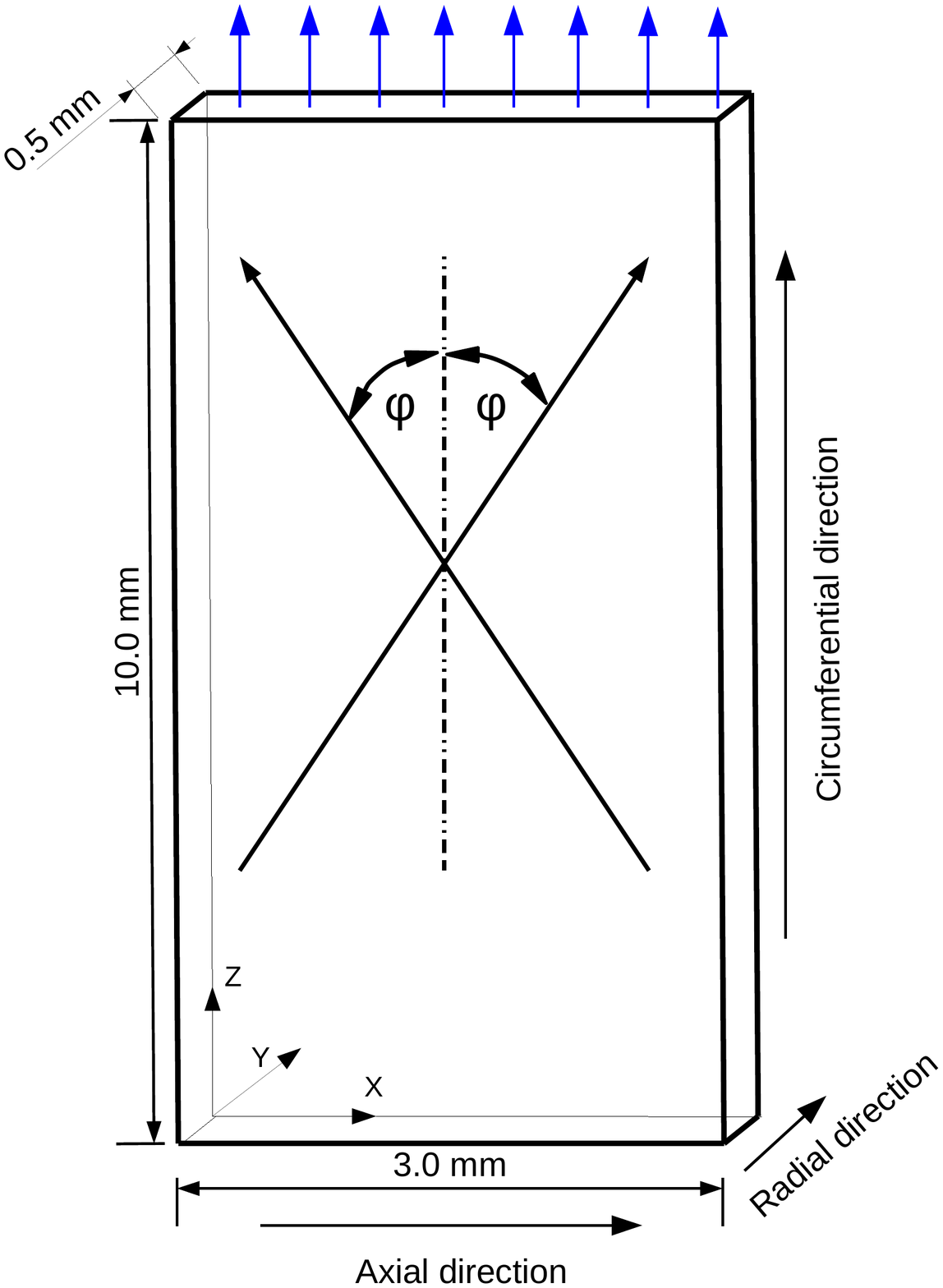}
    \end{minipage}
    &
    \begin{minipage}[t]{.56\textwidth}
    \begin{itemize}
      \item[] Material properties:
      \item[] $G(\tilde{\bm C},p) = G_{ich}^{g}( \tilde{\bm C} ) + \sum_{i=1,2} G_{ich}^{f_i}( \tilde{\bm C} ) + \frac{p}{\rho_0}$,
      \item[] $G_{ich}^{g}( \tilde{\bm C} ) = \frac{c_1}{2\rho_0} \left( \tilde{I}_1 -3 \right)$,
      \item[] $G_{ich}^{f_i}( \tilde{\bm C} ) = \frac{k_1}{2k_2\rho_0} \left( e^{k_2 \tilde{E}_i^2} - 1 \right)$,
      \item[] $\tilde{E}_i := \bm H_i : \tilde{\bm C} - 1$,
      \item[] $\bm H_i := k_d \bm I + (1-3k_d)(\bm a_i \otimes \bm a_i)$, 
	  \item[] $\rho_0 = 1.0 \times 10^3$ kg/m$^3$, $c_1 = 7.64 \times 10^3$ Pa,
	  \item[] $k_1 = 9.966 \times 10^{5}$ Pa, $k_2 = 524.6$, $k_d =0.226$.
      \end{itemize}
    \end{minipage}   
    \\ 
    \hline
  \end{tabular}
  \caption{Three-dimensional tensile test: geometry setting and material properties.} 
\label{table:3d_tensile_benchmark_geometry}
\end{table}

\begin{figure}
\begin{center}
\begin{tabular}{cc} 
\includegraphics[angle=0, trim=30 0 50 10, clip=true, scale = 0.28]{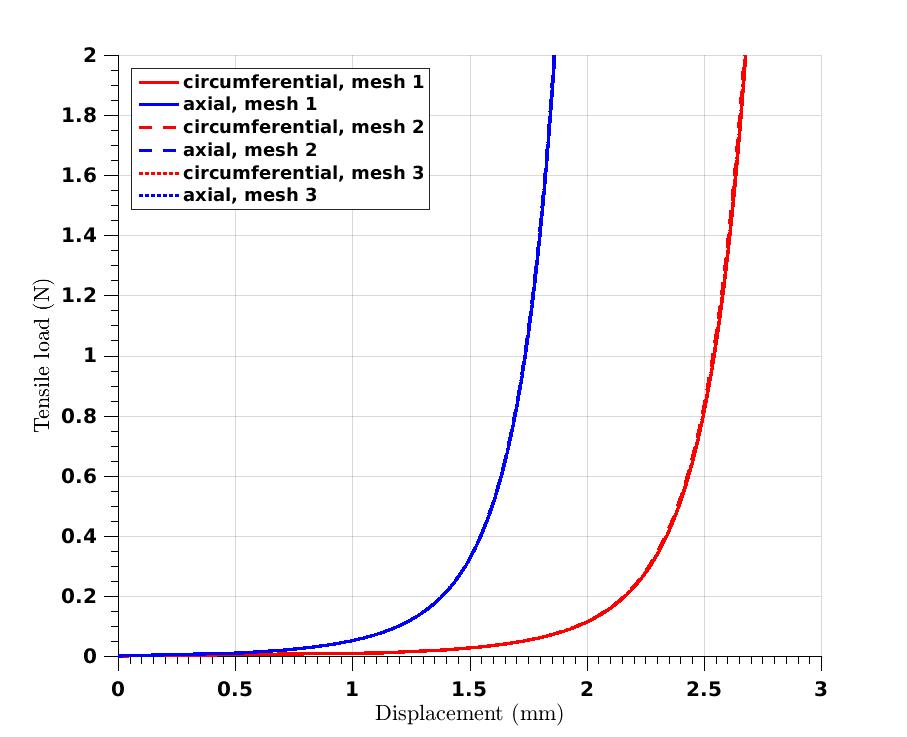} &
\includegraphics[angle=0, trim=20 0 30 10, clip=true, scale = 0.28]{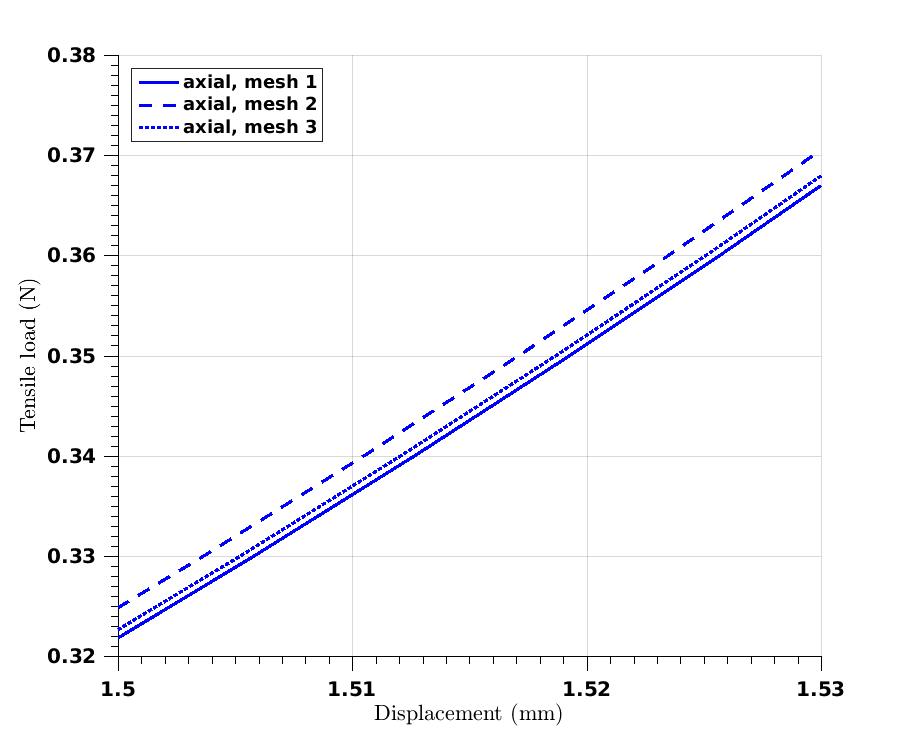} \\
(a) & (b)
\end{tabular}
\caption{Three-dimensional tensile test: (a) computed load-displacement curves of the circumferential (red) and axial specimens (blue) using different meshes; (b) detailed comparison of the computed load-displacement curves near the tensile load $0.35$ N.} 
\label{fig:tensile_stress_strain}
\end{center}
\end{figure}

\begin{table}[htbp]
\begin{center}
\tabcolsep=0.3cm
\renewcommand{\arraystretch}{1.5}
\begin{tabular}{ c c c c c c }
\hline
Mesh & 1 & 2 & 3 & 4 & 5 \\
\hline
$n_{en}$ & 61440 & 120 & 120 & 5760 & 90000 \\
$n_{eq}$ & 1785024 & 5091 & 7584 & 6396 & 75144 \\
\hline 
\end{tabular}
\end{center}
\caption{The number of elements $n_{en}$ and the number of equations $n_{eq}$ in the system \eqref{eq:pre-multi-corrector-stage-5-matrix-problems} for the five different meshes. Meshes 4 and 5 consist of linear tetrahedral elements.}
\label{table:tensile_mesh}
\end{table}

\begin{figure}
\begin{center}
\begin{tabular}{ccccc} 
\includegraphics[angle=0, trim=450 0 450 0, clip=true, scale = 0.32]{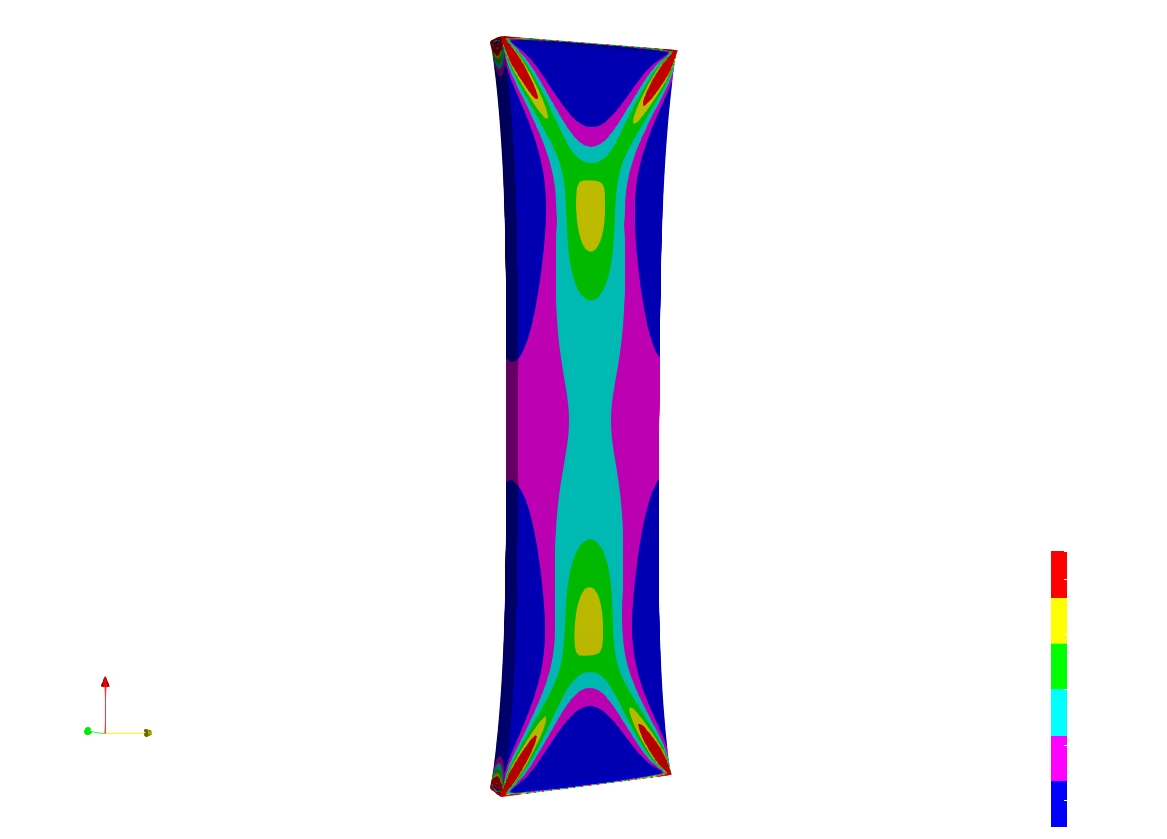} &
\includegraphics[angle=0, trim=450 0 450 0, clip=true, scale = 0.32]{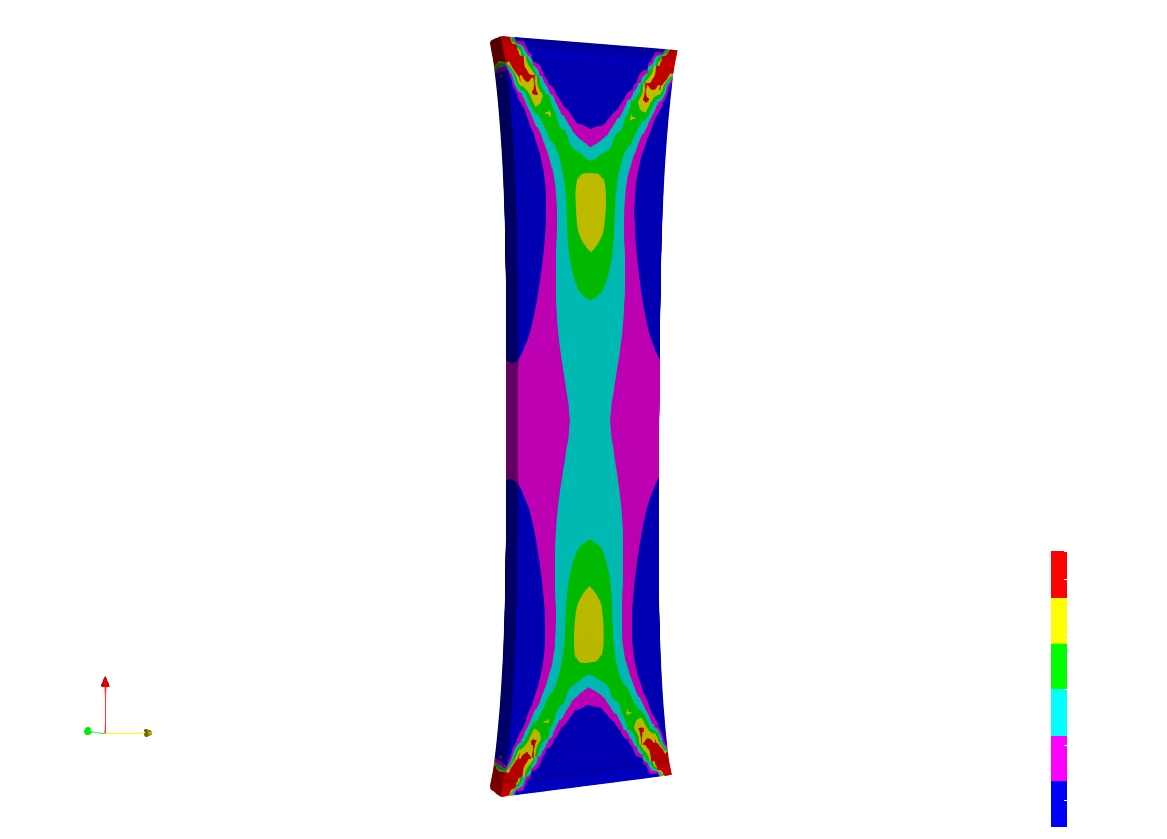} &
\includegraphics[angle=0, trim=450 0 450 0, clip=true, scale = 0.32]{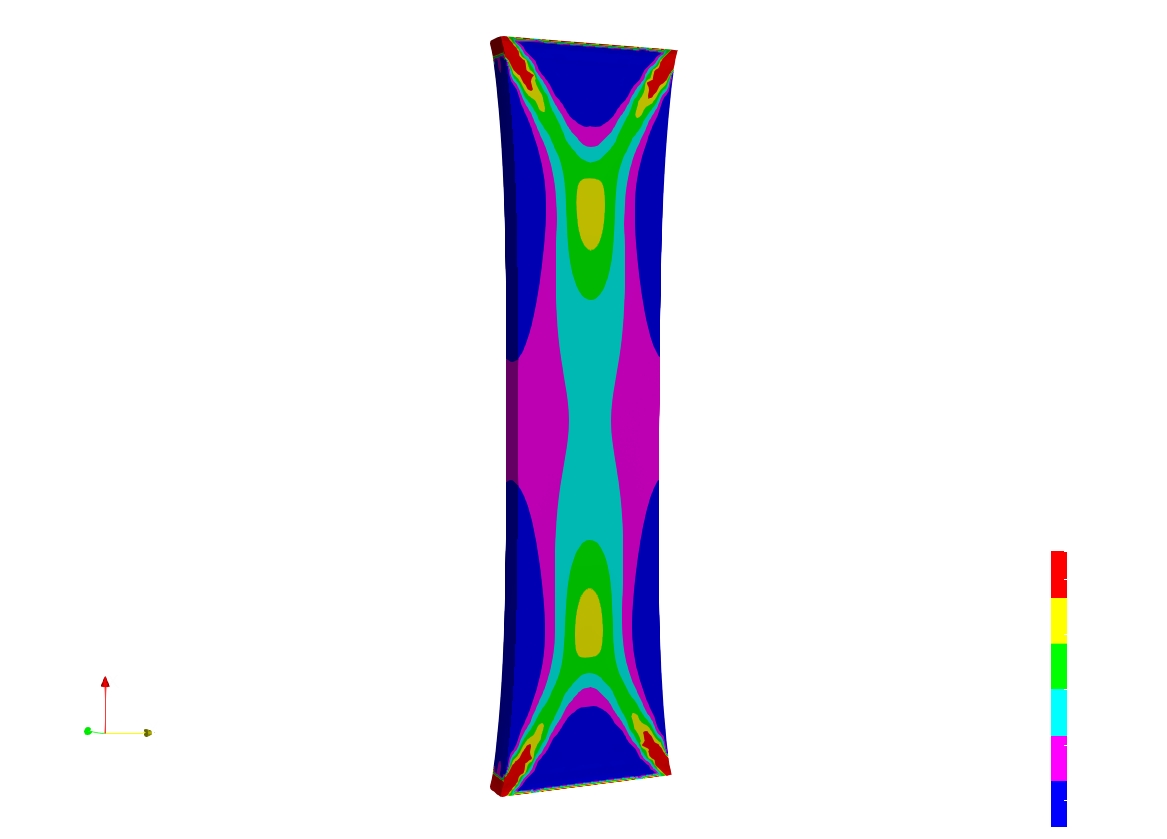} &
\includegraphics[angle=0, trim=830 0 830 0, clip=true, scale = 0.223]{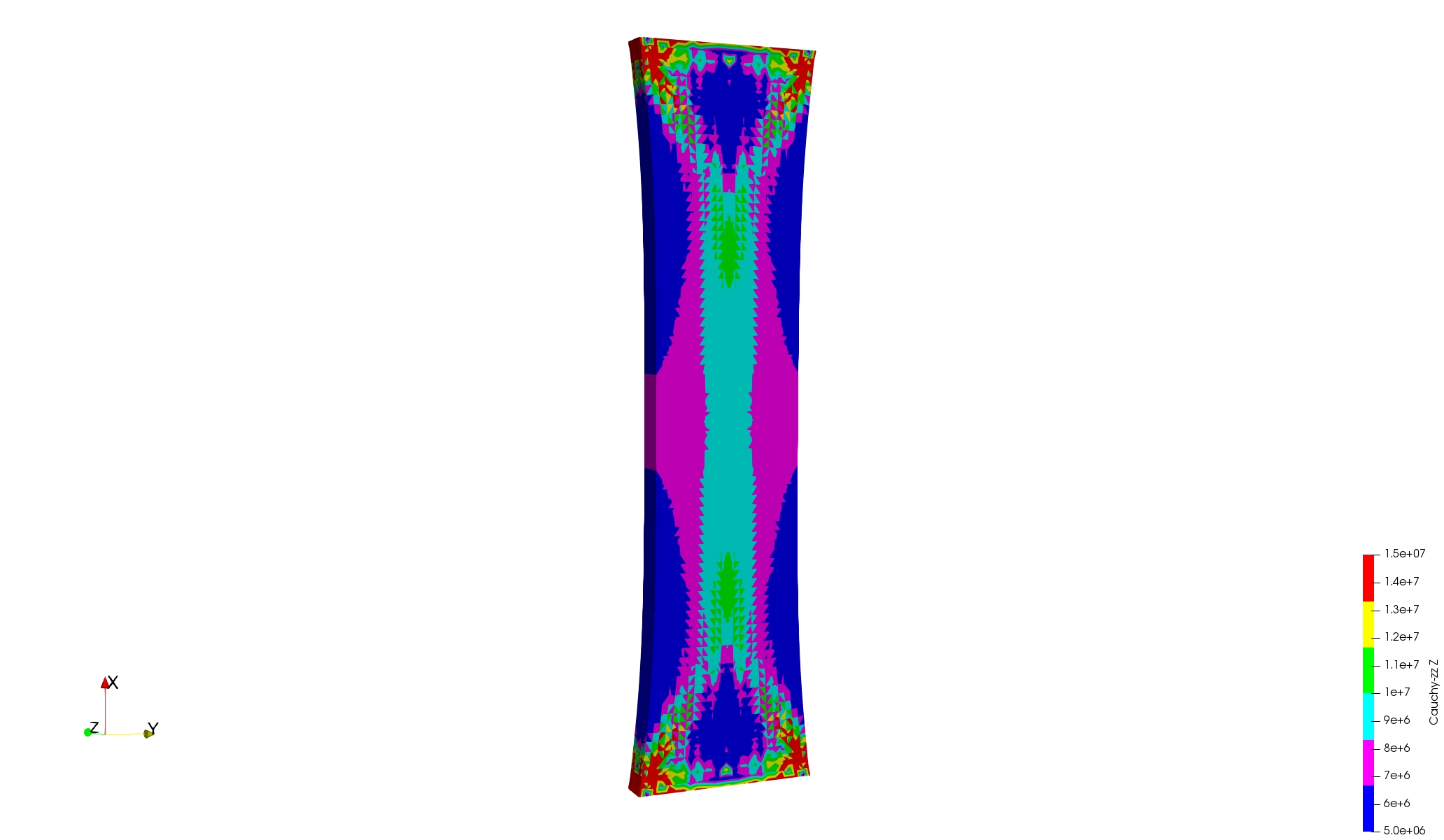} &
\includegraphics[angle=0, trim=830 0 830 0, clip=true, scale = 0.22]{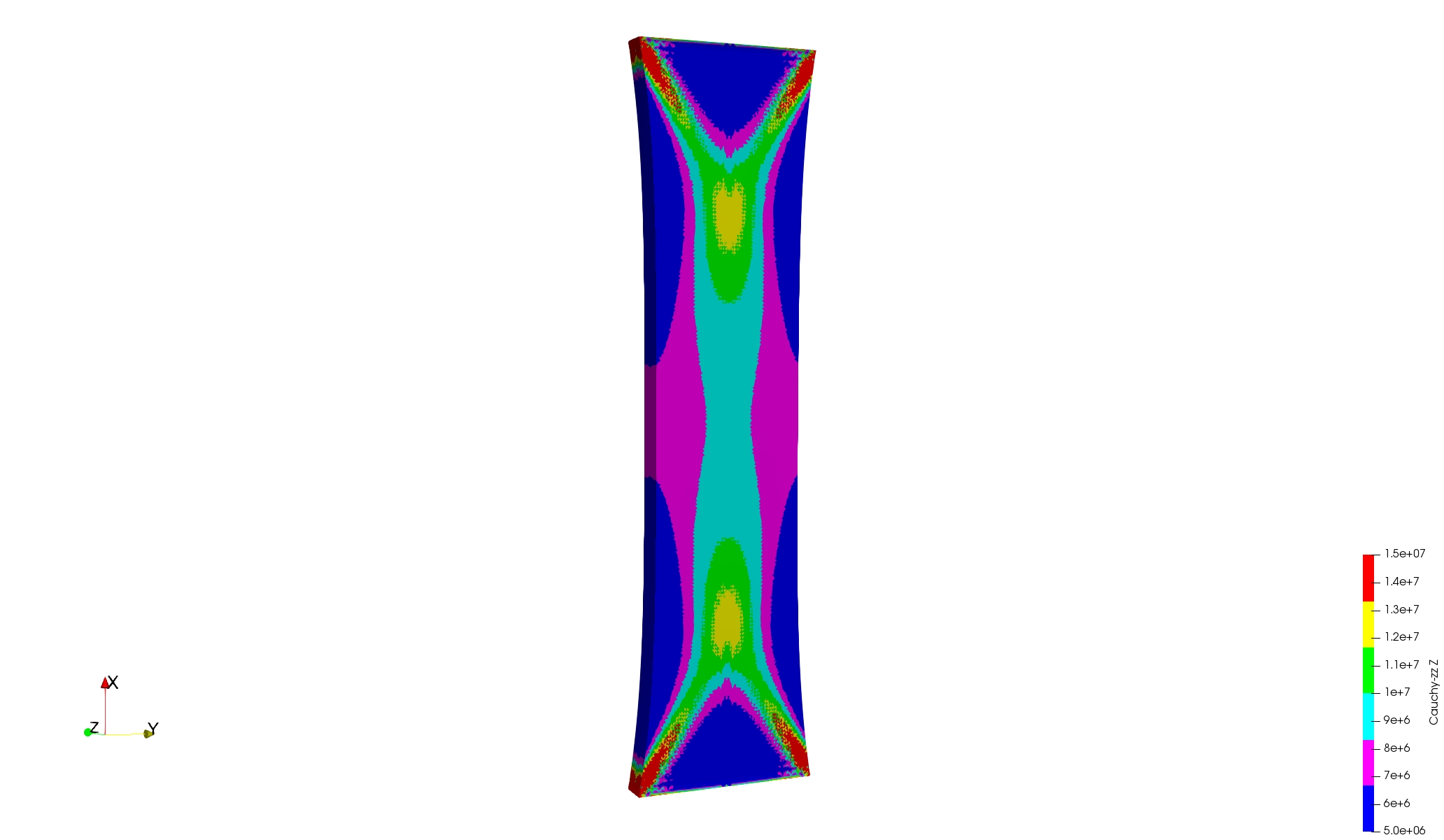} \\
(a) & (b) & (c) & (d) & (e) \\
\multicolumn{5}{c}{
\includegraphics[angle=0, trim=0 0 350 700, clip=true, scale = 0.3]{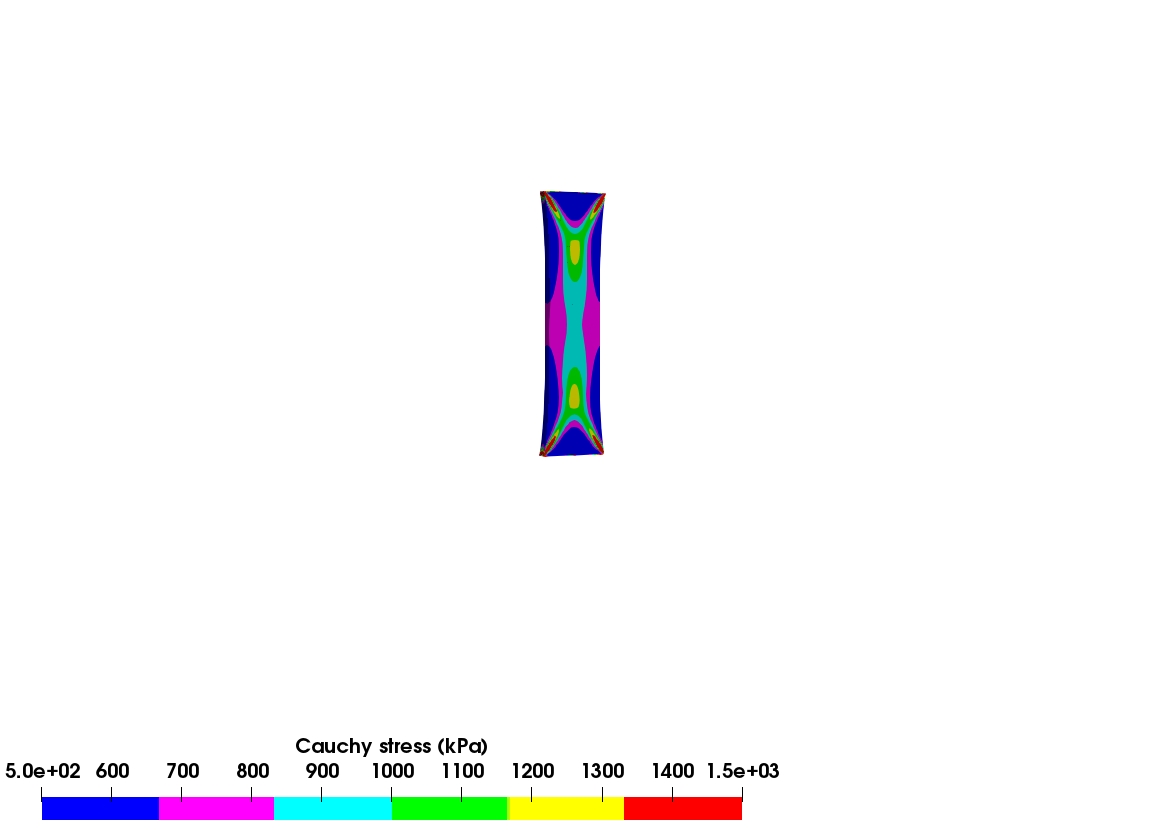}
}
\end{tabular}
\caption{Three-dimensional tensile test: $\bm \sigma_{zz}$ for the axial specimen calculated by (a) mesh 1, (b) mesh 2, (c) mesh 3, (d) mesh 4, and (e) mesh 5 on the deformed configurations at the tensile load $1$ N.} 
\label{fig:tensile_test_stress_axial}
\end{center}
\end{figure}

\begin{figure}
\begin{center}
\begin{tabular}{ccccc} 
\includegraphics[angle=0, trim=450 0 450 0, clip=true, scale = 0.32]{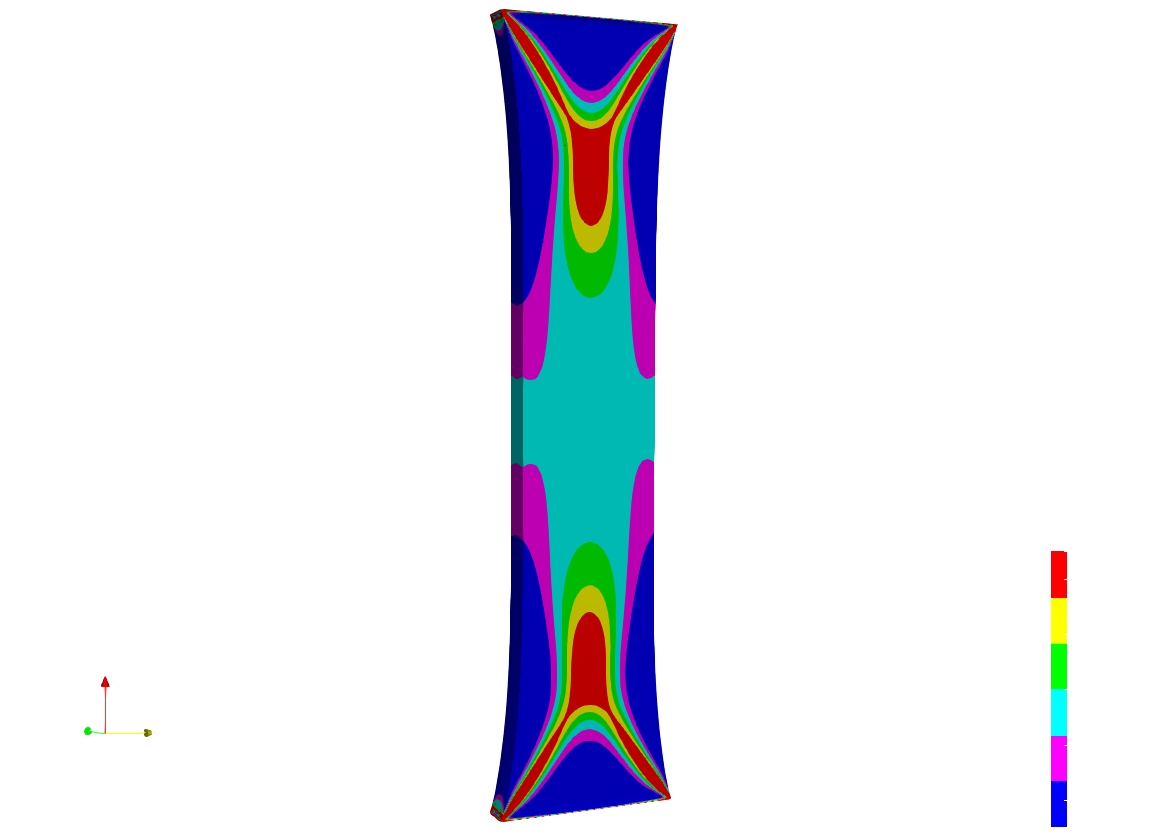} &
\includegraphics[angle=0, trim=450 0 450 0, clip=true, scale = 0.32]{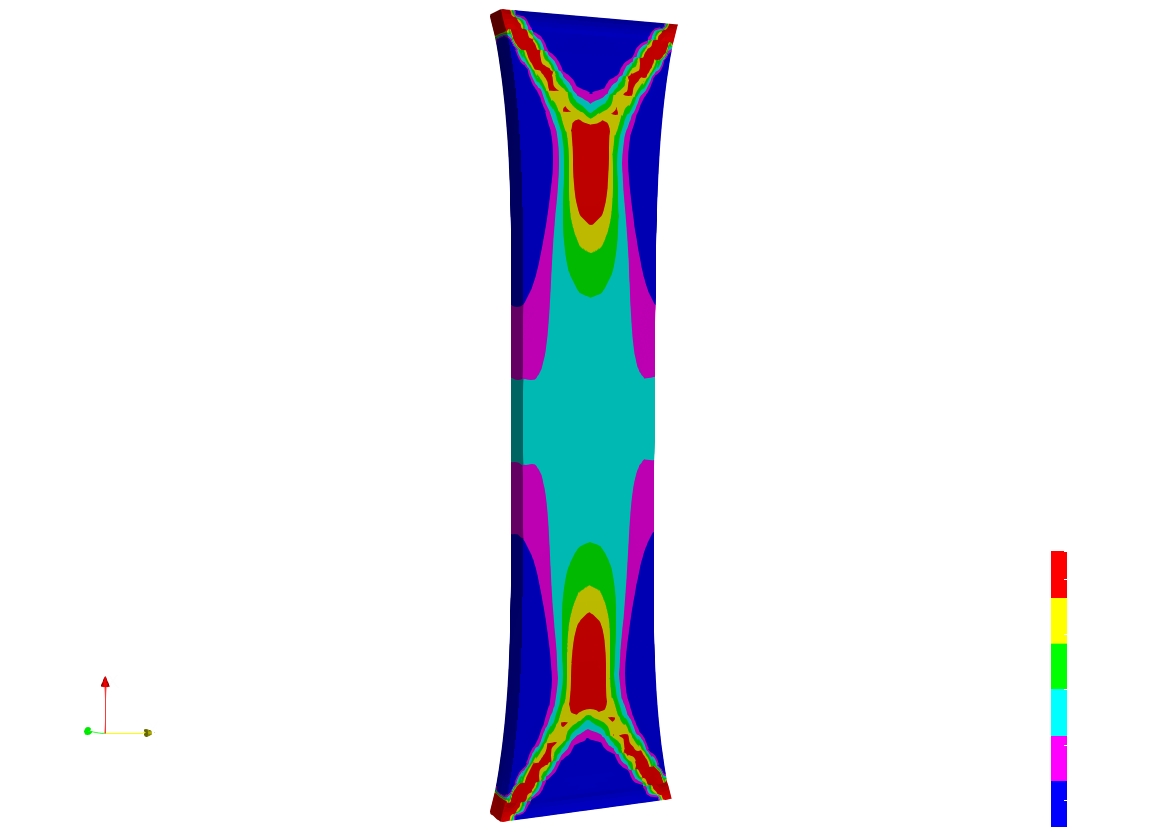} &
\includegraphics[angle=0, trim=450 0 450 0, clip=true, scale = 0.32]{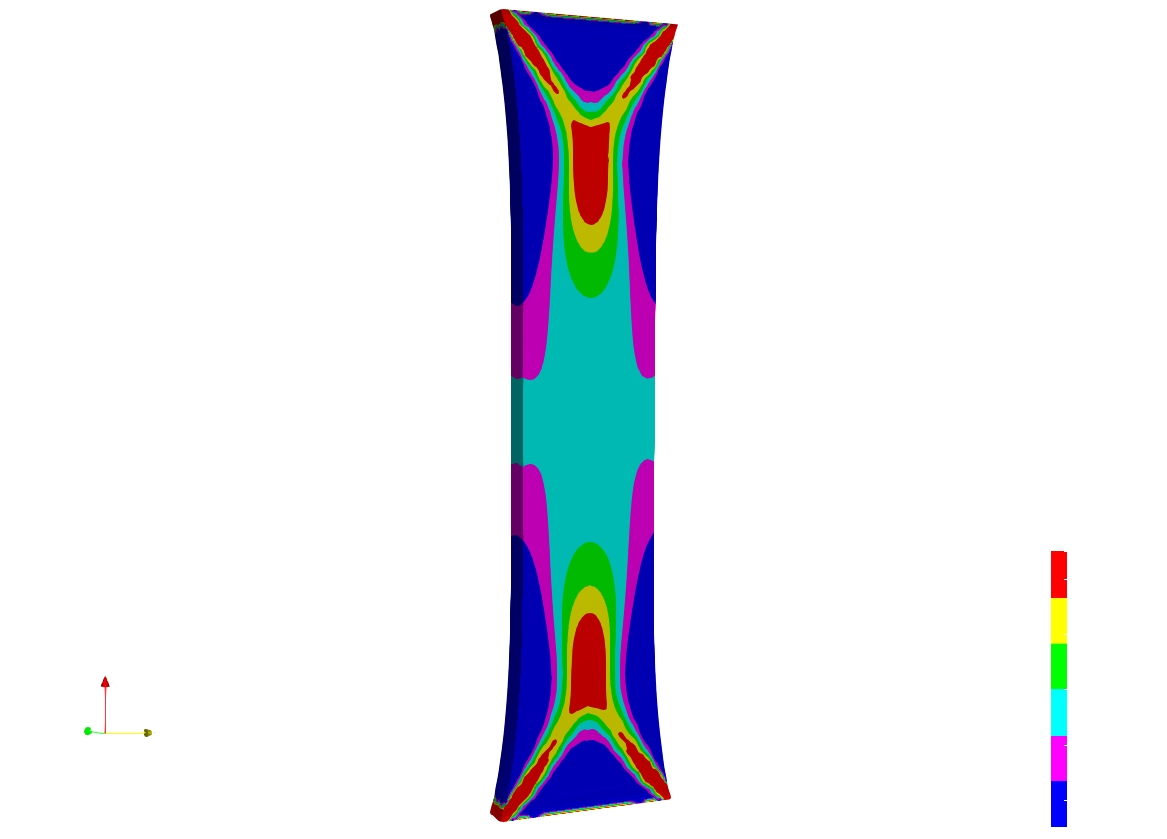} &
\includegraphics[angle=0, trim=830 0 830 0, clip=true, scale = 0.223]{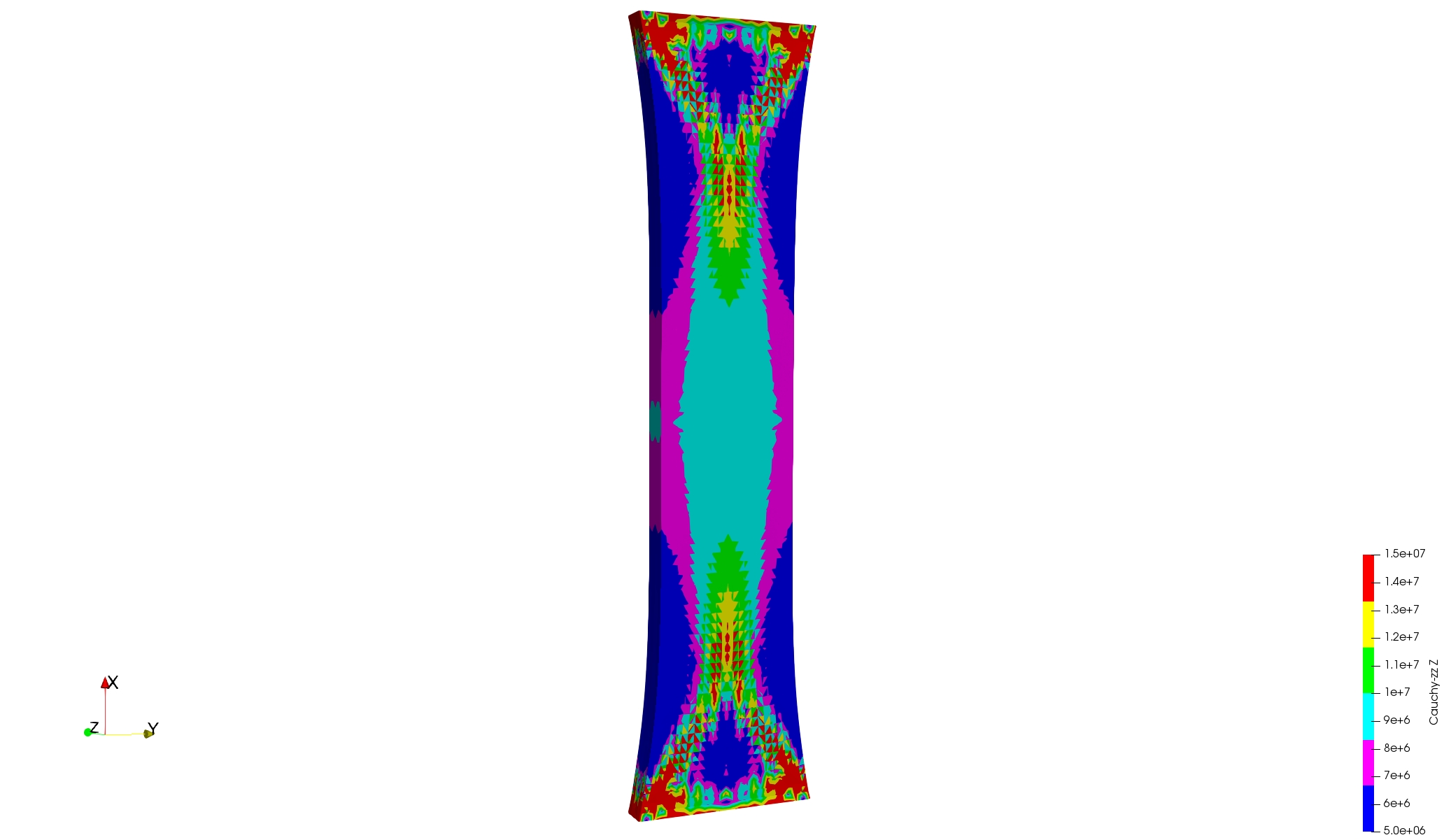} &
\includegraphics[angle=0, trim=880 0 880 0, clip=true, scale = 0.22]{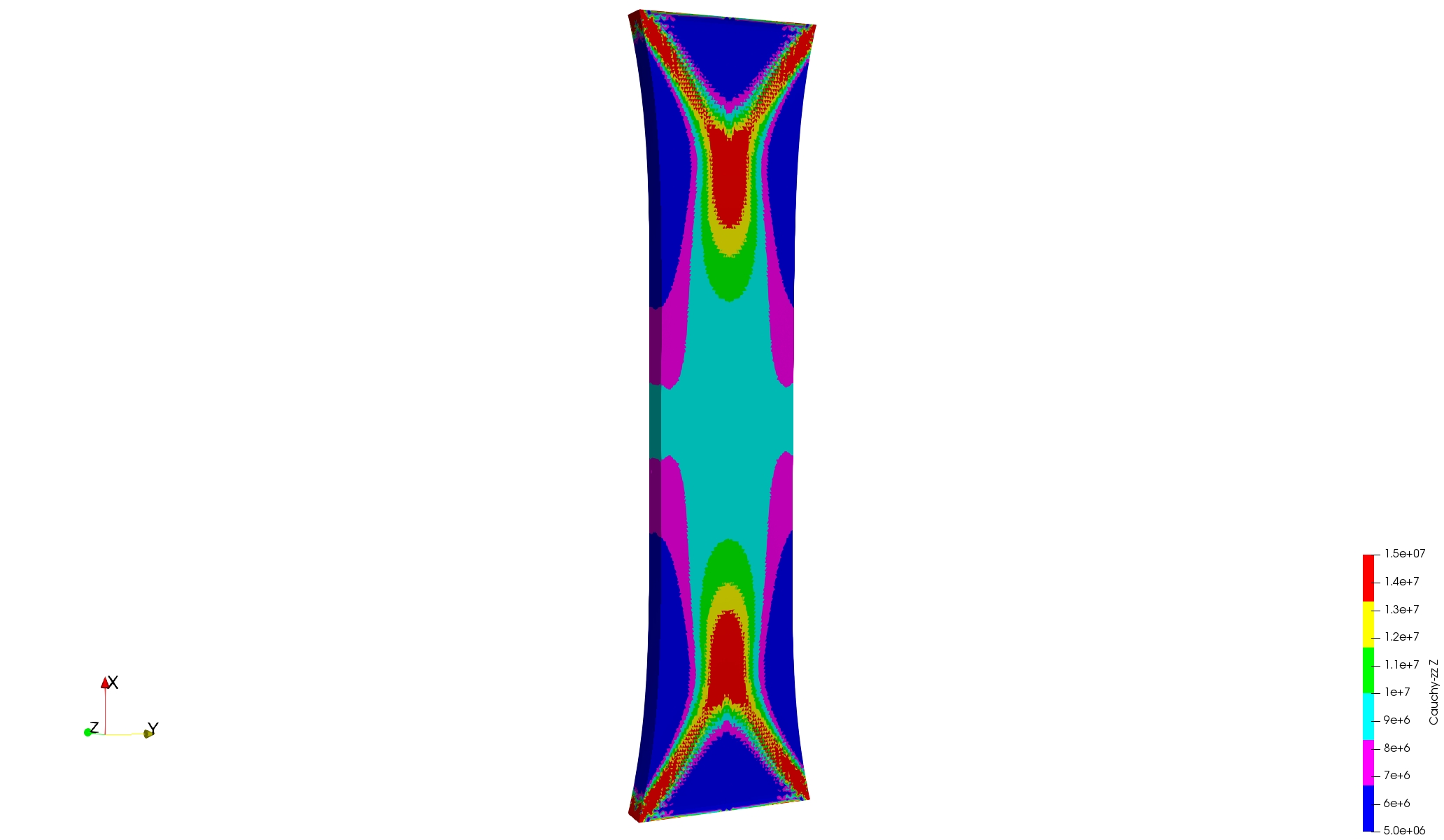} \\
(a) & (b) & (c) & (d) & (e) \\
\multicolumn{5}{c}{
\includegraphics[angle=0, trim=0 0 350 700, clip=true, scale = 0.3]{./figures/tensile_legend.jpg}
}
\end{tabular}
\caption{Three-dimensional tensile test: $\bm \sigma_{zz}$ for the circumferential specimen calculated by (a) mesh 1, (b) mesh 2, (c) mesh 3, (d) mesh 4, and (e) mesh 5 on the deformed configurations at the tensile load $1$ N.} 
\label{fig:tensile_test_stress_circum}
\end{center}
\end{figure}

\subsection{Three-dimensional beam bending}
In this example, we present a three-dimensional beam vibration problem to evaluate the performance of the elastodynamics formulation in a bending dominated scenario \cite{Bonet2015}. The problem configuration as well as the material properties are illustrated in Table \ref{table:3d_beam_bending_benchmark_geometry}. The beam is fully clamped at the base, and the other faces are specified by zero tractions. The body is initially stress free with zero displacement. The vibration is initiated through an initial velocity
\begin{align*}
\bm V(\bm X, 0) = \left( V_0 \frac{Z}{L_0}, 0, 0 \right)^T, \quad V_0 = \frac53 \textup{m}/\textup{s}.
\end{align*}
This initial condition leads to an oscillatory motion of the beam. For the simulations, we choose $\mathsf p = 1$, $\mathsf a = 1$, and $\mathsf b = 0$ for the discrete function spaces. We use $\textup{tol}_{\textup{R}} = 10^{-8}$ and $\textup{tol}_{\textup{A}} = 10^{-8}$ as the stopping criteria. The numerical results show the deformation state of the beam calculated from the two different meshes are indistinguishable, suggesting a coarse mesh with $\Delta x = L_0 / 2$ is capable of accurately describing the beam dynamics (Figure \ref{fig:beam_snapshot}). Since the boundary data is time independent and the body force and surface tractions are zero, the total energy of the beam is conserved according to Proposition \ref{prop:energy_stability}. We observe that the total energy is well-preserved up to $T = 10$ s (Figure \ref{fig:beam_energy} (a)). From the periodic pattern of the kinetic and potential energies, we obtain an average period of the oscillation is 0.9018 s. To better illustrate the energy conservation, we plot the relative errors of the energy in Figure \ref{fig:beam_energy} (b), using three different spatial meshes. Interestingly, the error of the total energy achieves its maximum value when the beam reaches its largest deformation. For the coarsest mesh ($\Delta x = L_0 /2$), the error accumulates slightly over time, and we can see that the relative error reaches about one percent at around $9.5$ s. We also observe that the spatial mesh refinement helps reduce the error of the total energy. For the meshes with $\Delta x = L_0 / 4$ and $\Delta x = L_0 / 6$, we do not observe a pile-up effect of the energy error. Also, the magnitude of the relative error is reduced with mesh refinement. In comparison with the previously published results \cite{Aguirre2014,Lahiri2005}, the new formulation enjoys a better discrete energy conservation property. 

\begin{table}[htbp]
  \centering
  \begin{tabular}{ m{.4\textwidth}   m{.5\textwidth} }
    \hline
    \begin{minipage}{.38\textwidth}
      \includegraphics[angle=0, trim=120 270 250 200, clip=true, scale = 0.7]{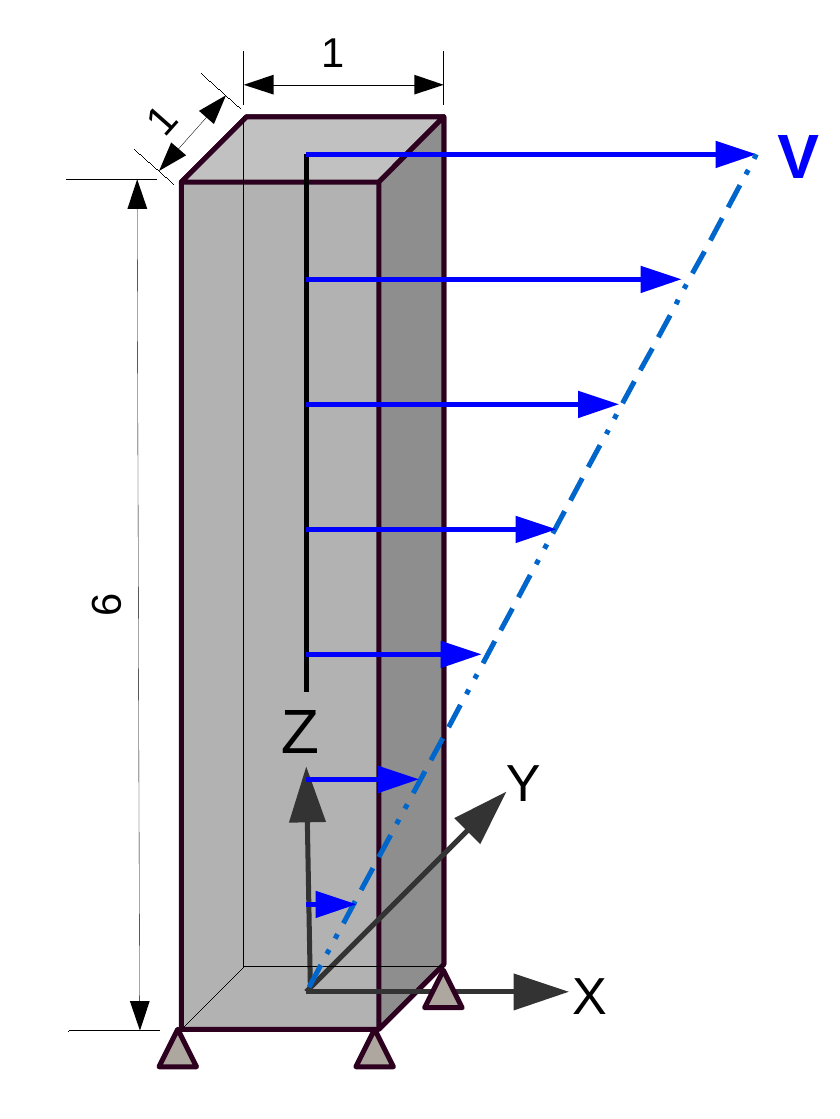}
    \end{minipage}
    &
    \begin{minipage}[t]{.48\textwidth}
      
      \begin{itemize}
      \item[] Material properties:
      \item[] $G_{ich}(\tilde{\bm C}) = \frac{c_1}{2\rho_0} \left( \tilde{I}_1 - 3 \right) + \frac{c_2}{2\rho_0} \left( \tilde{I}_2 - 3 \right)$,
\item[] $\rho_0 = 1.1 \times 10^3$ kg/m$^3$,  
      \item[] $c_1 = c_2 = E/6$,
      \item[] $E = 1.7\times 10^7$ Pa.
      \end{itemize}
      
      \begin{itemize}
      \item[] Reference scales:
      \item[] $L_0 = 1$ m, $M_0 = 1$ kg, $T_0 = 1$ s. 
      \end{itemize}
	     
    \end{minipage}   
    \\ 
    \hline
  \end{tabular}
  \caption{Three-dimensional beam bending: geometry setting, boundary conditions, and material properties.} 
\label{table:3d_beam_bending_benchmark_geometry}
\end{table}

\begin{figure}
\begin{center}
\begin{tabular}{cccc}
$t=0.31$ s & $t=0.50$ s & $t=0.91$ s & $t=1.23$ s \\ 
\includegraphics[angle=0, trim=500 10 550 0, clip=true, scale = 0.16]{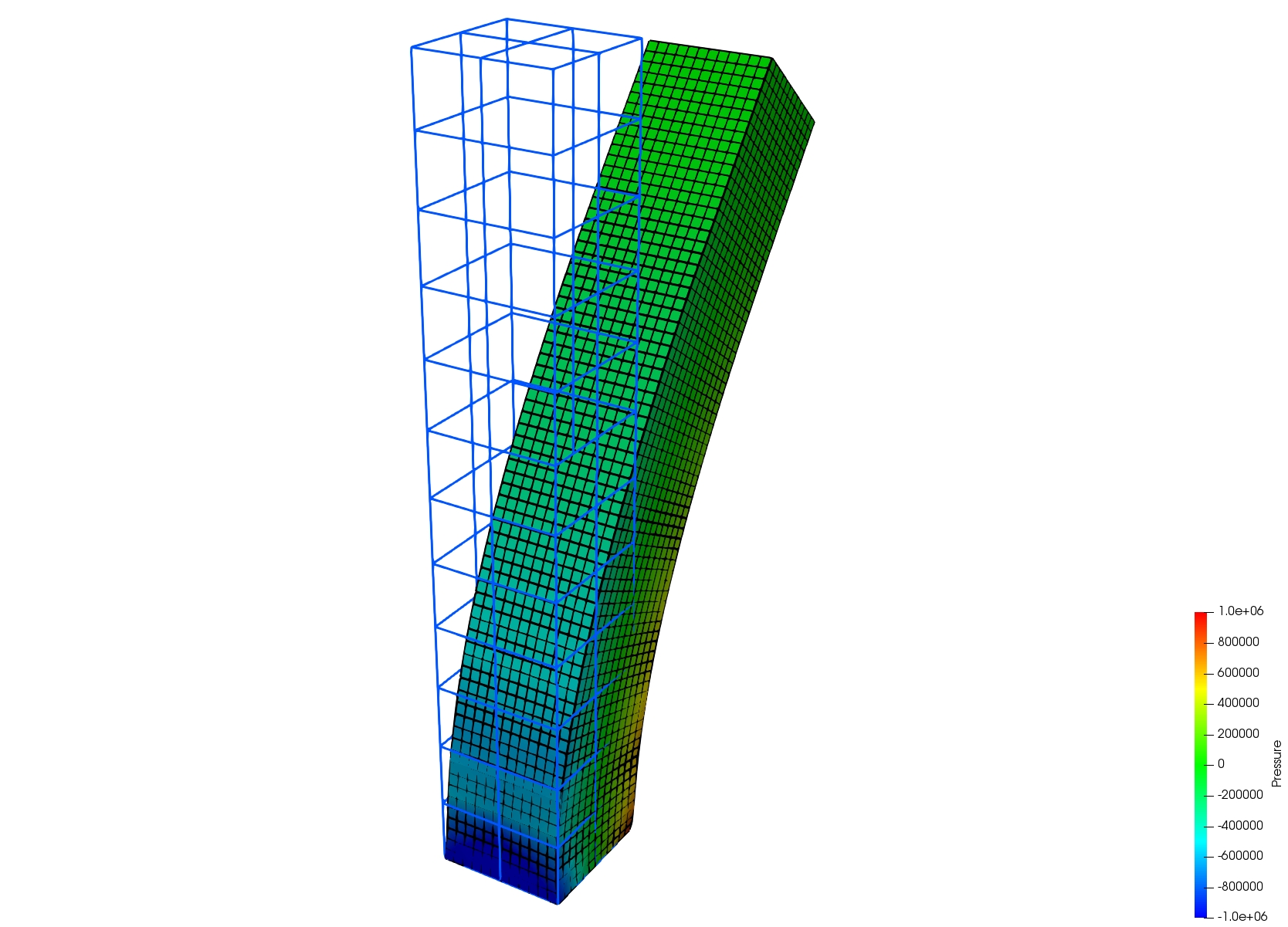} & 
\includegraphics[angle=0, trim=500 10 550 0, clip=true, scale = 0.16]{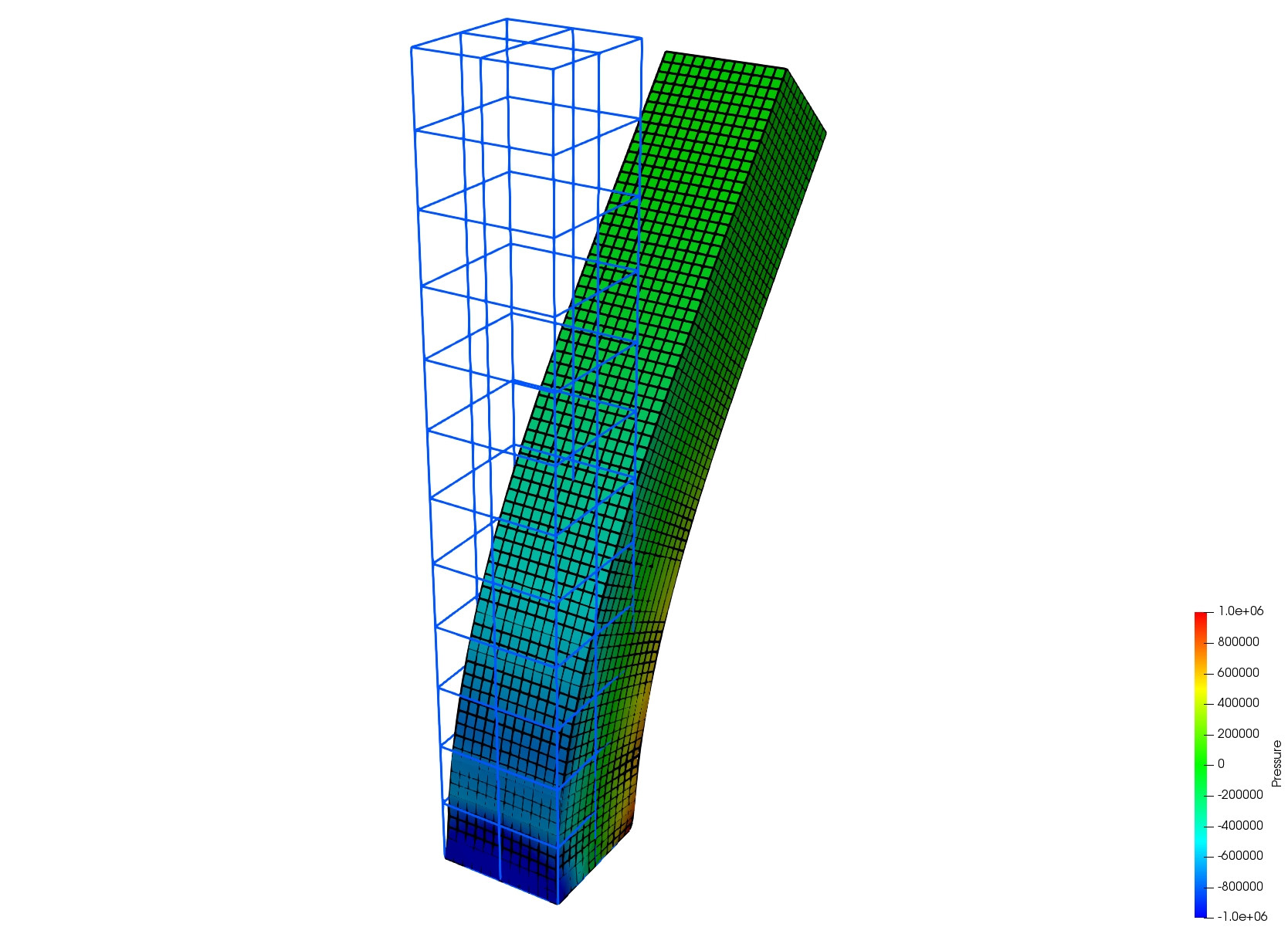} &
\includegraphics[angle=0, trim=500 10 800 0, clip=true, scale = 0.16]{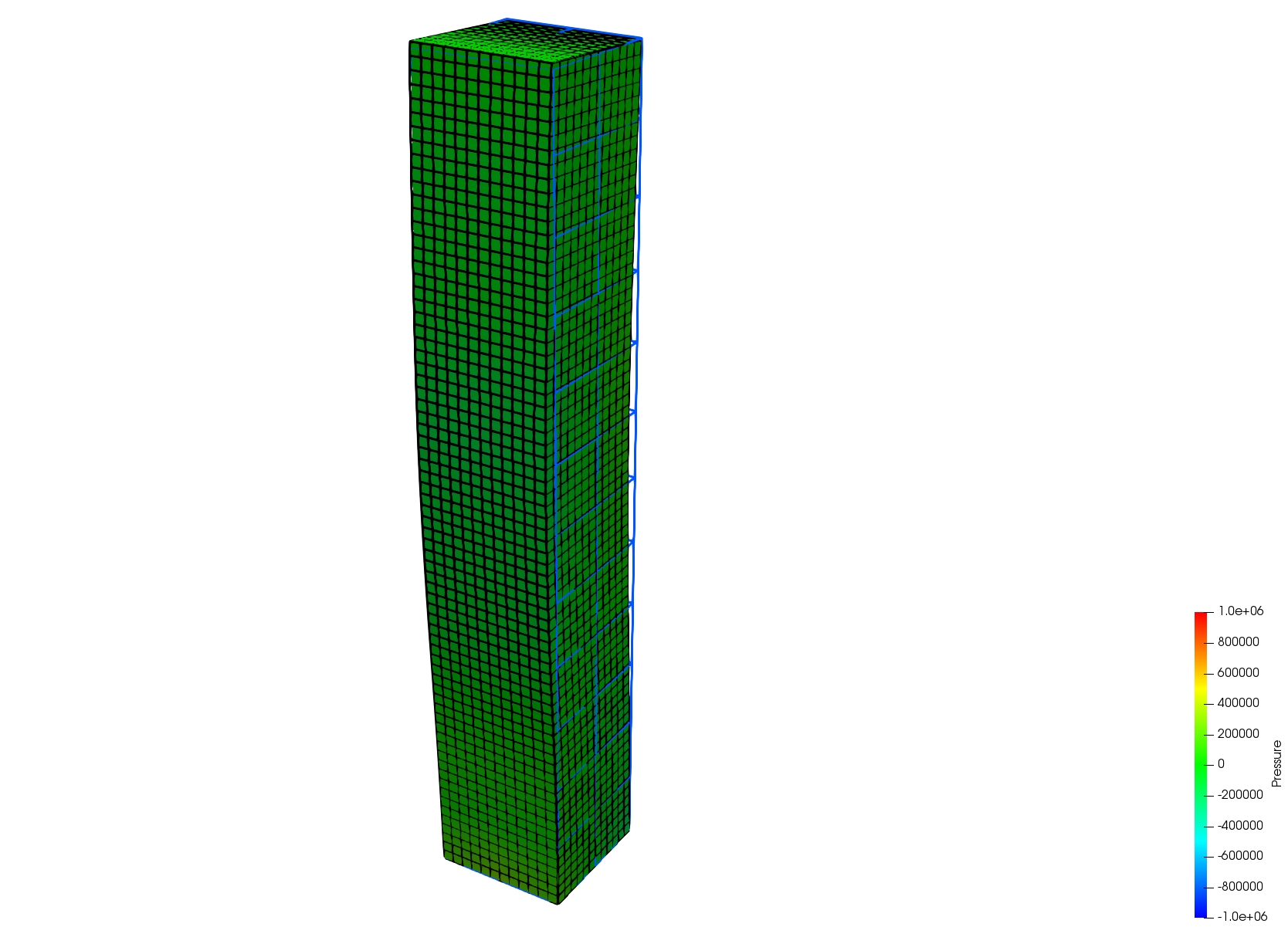} &
\includegraphics[angle=0, trim=120 10 800 0, clip=true, scale = 0.16]{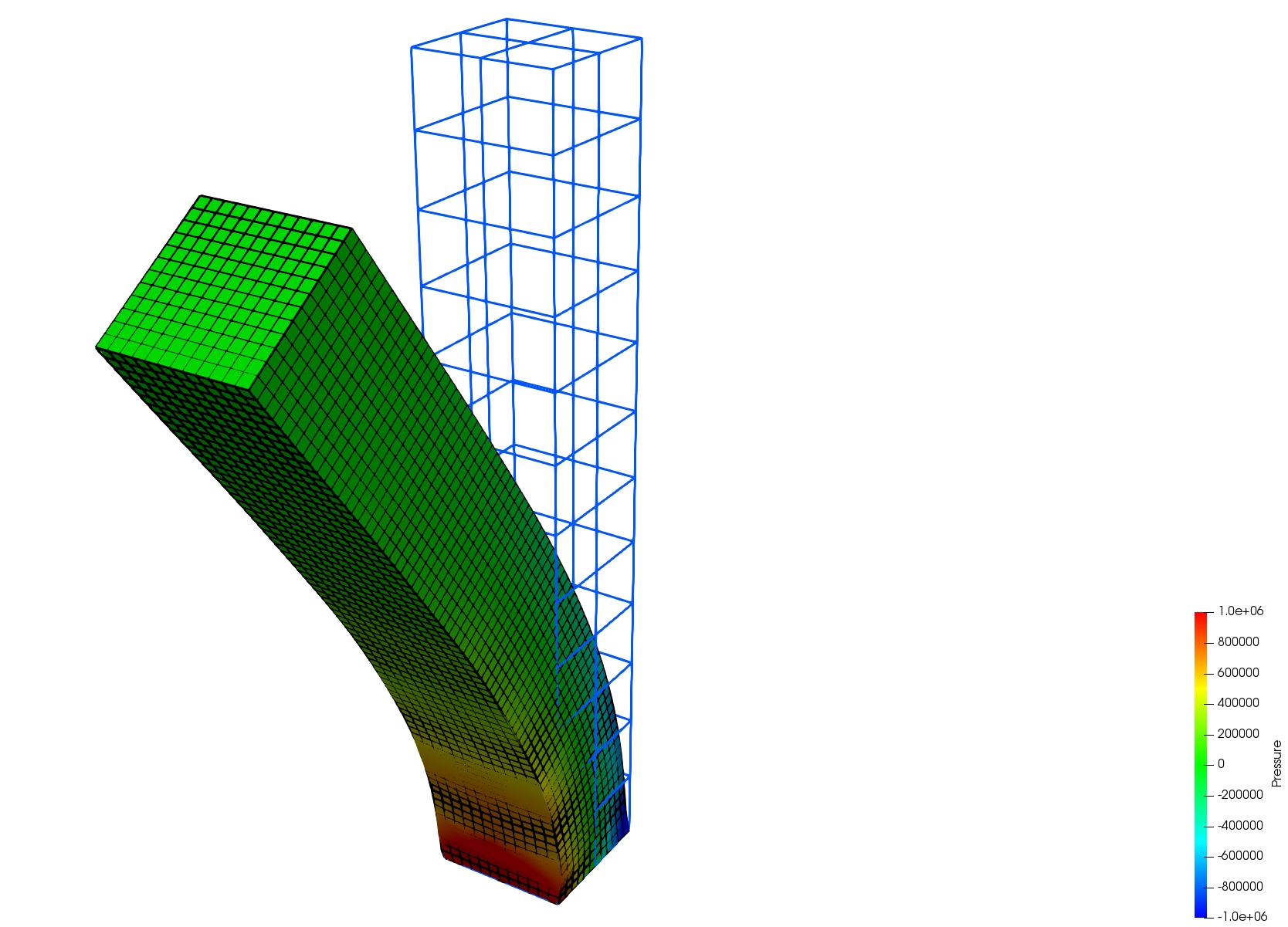} \\
\multicolumn{4}{c}{
\includegraphics[angle=0, trim=0 10 700 1050, clip=true, scale = 0.3]{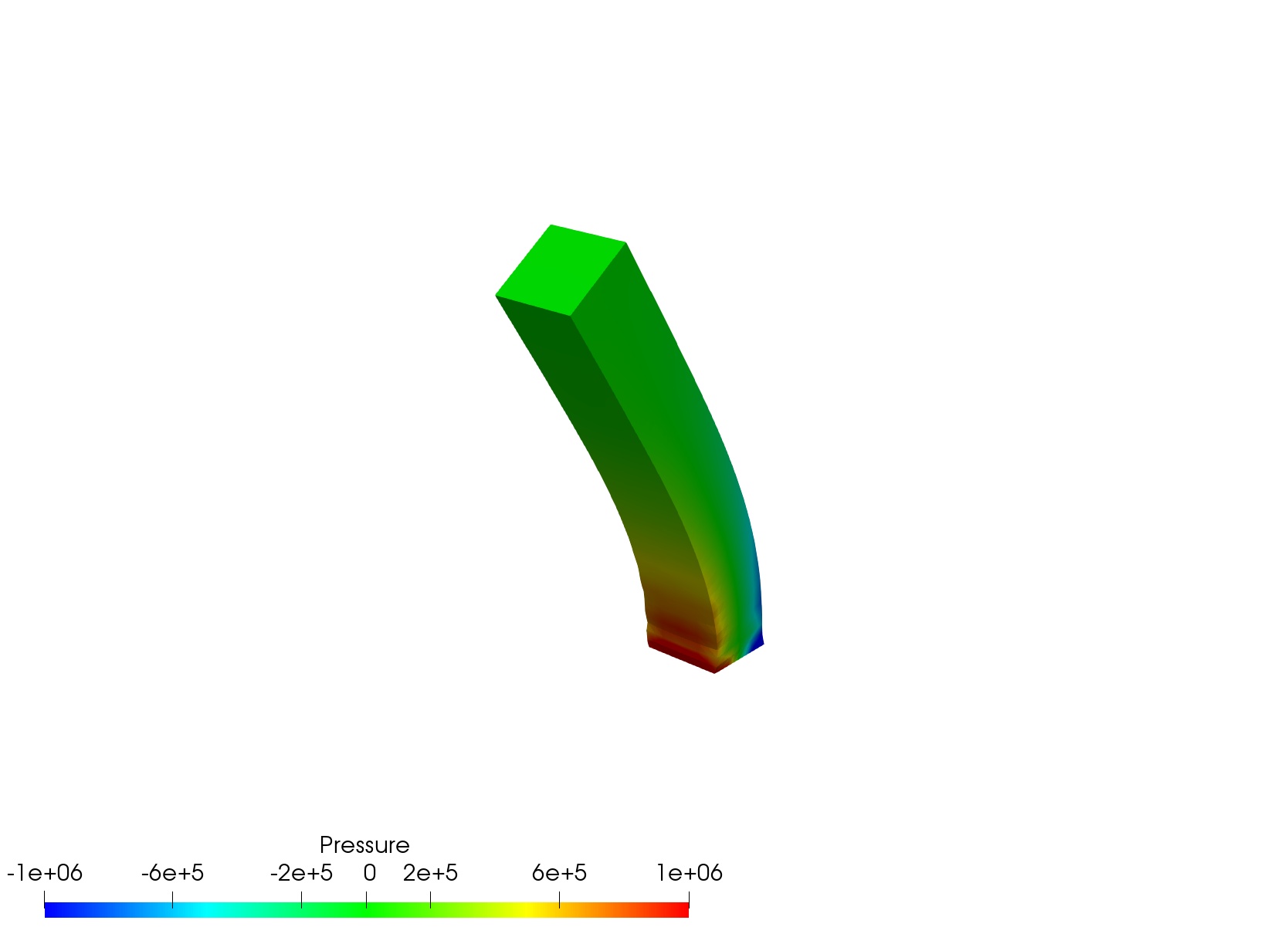}
}
\end{tabular}
\caption{Snapshots of the vibrating beam: The pressure field at different time steps using mesh size $\Delta x = L_0/2$ and time step size $\Delta t = 10^{-3}T_0$. The deformation states at the corresponding time steps using mesh size $\Delta x = L_0/12$ and time step size $\Delta t = 10^{-4}T_0$ are shown as the black grid. The light blue grid shows the mesh with size $\Delta x = L_0/2$ at time $t=0$.} 
\label{fig:beam_snapshot}
\end{center}
\end{figure}

\begin{figure}
\begin{center}
\begin{tabular}{c}
\includegraphics[angle=0, trim=85 85 120 80, clip=true, scale = 0.45]{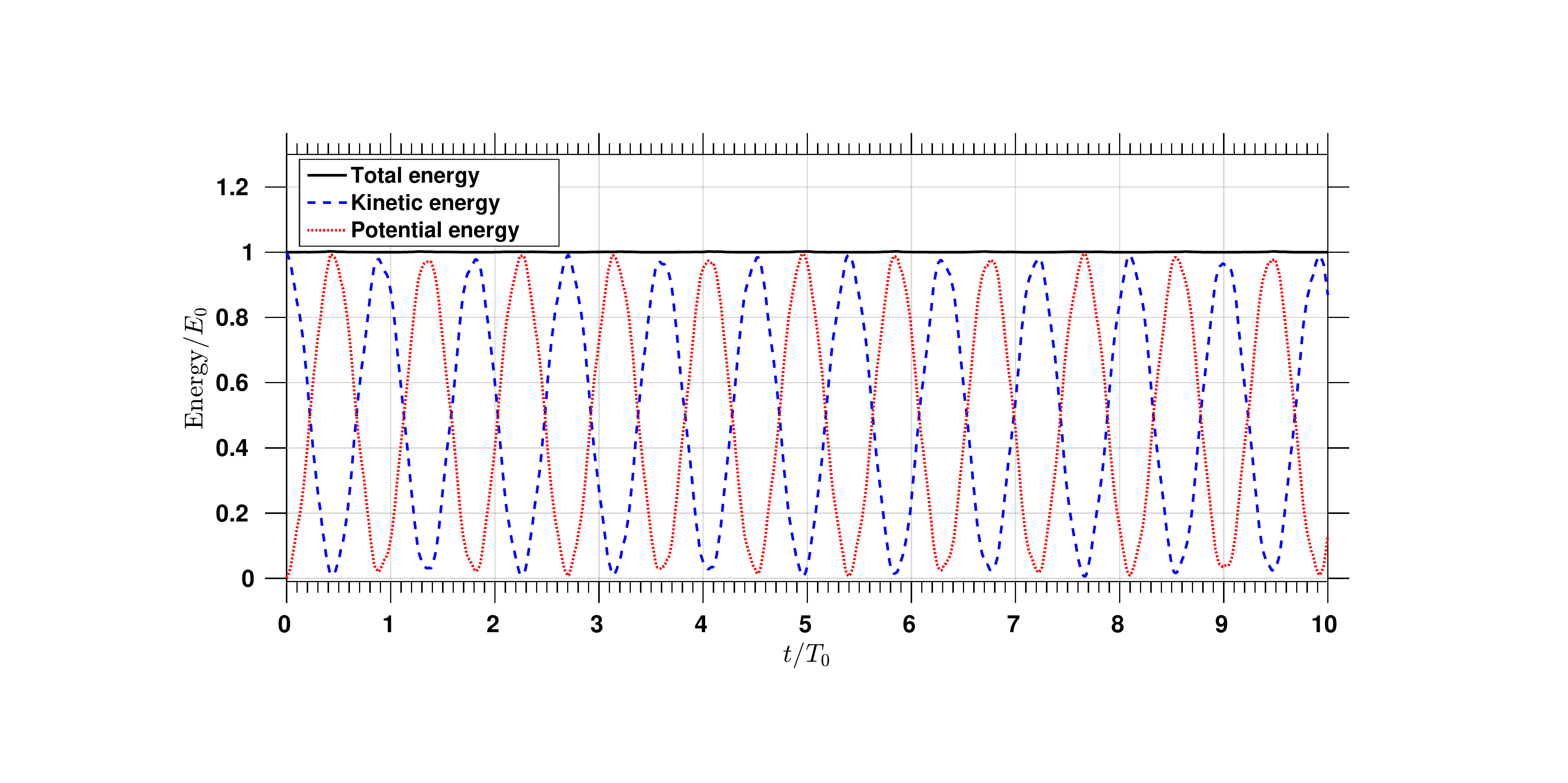} \\
(a) \\
\includegraphics[angle=0, trim=85 85 120 80, clip=true, scale = 0.45]{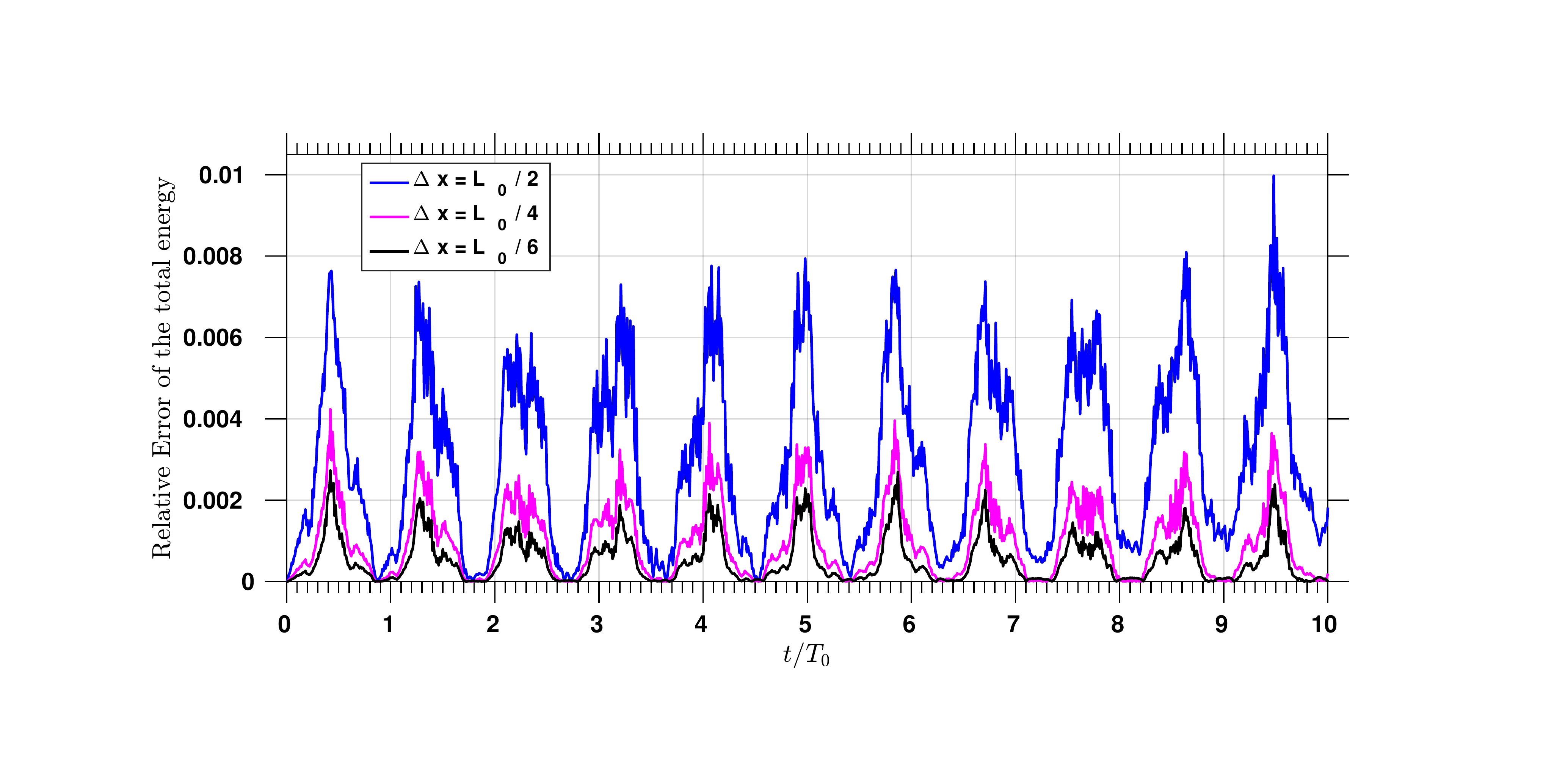}\\
(b)
\end{tabular}
\caption{(a) The total, kinetic, and potential energies over time with $\Delta x = L_0 / 6$; (b) The relative error of the total energy over time. The simulations are performed with $\mathsf p = 1$, $\mathsf a=1$, $\mathsf b = 0$, and $\Delta t = 2\times 10^{-4} T_0$. The reference value of the total energy $E_0$ is chosen to be the total energy at time $t=0$, which is $1.1\times 10^5$ kg m$^2$/s$^2$.} 
\label{fig:beam_energy}
\end{center}
\end{figure}

\begin{figure}
\begin{center}
\begin{tabular}{cc}
\includegraphics[angle=0, trim=25 120 25 120, clip=true, scale = 0.36]{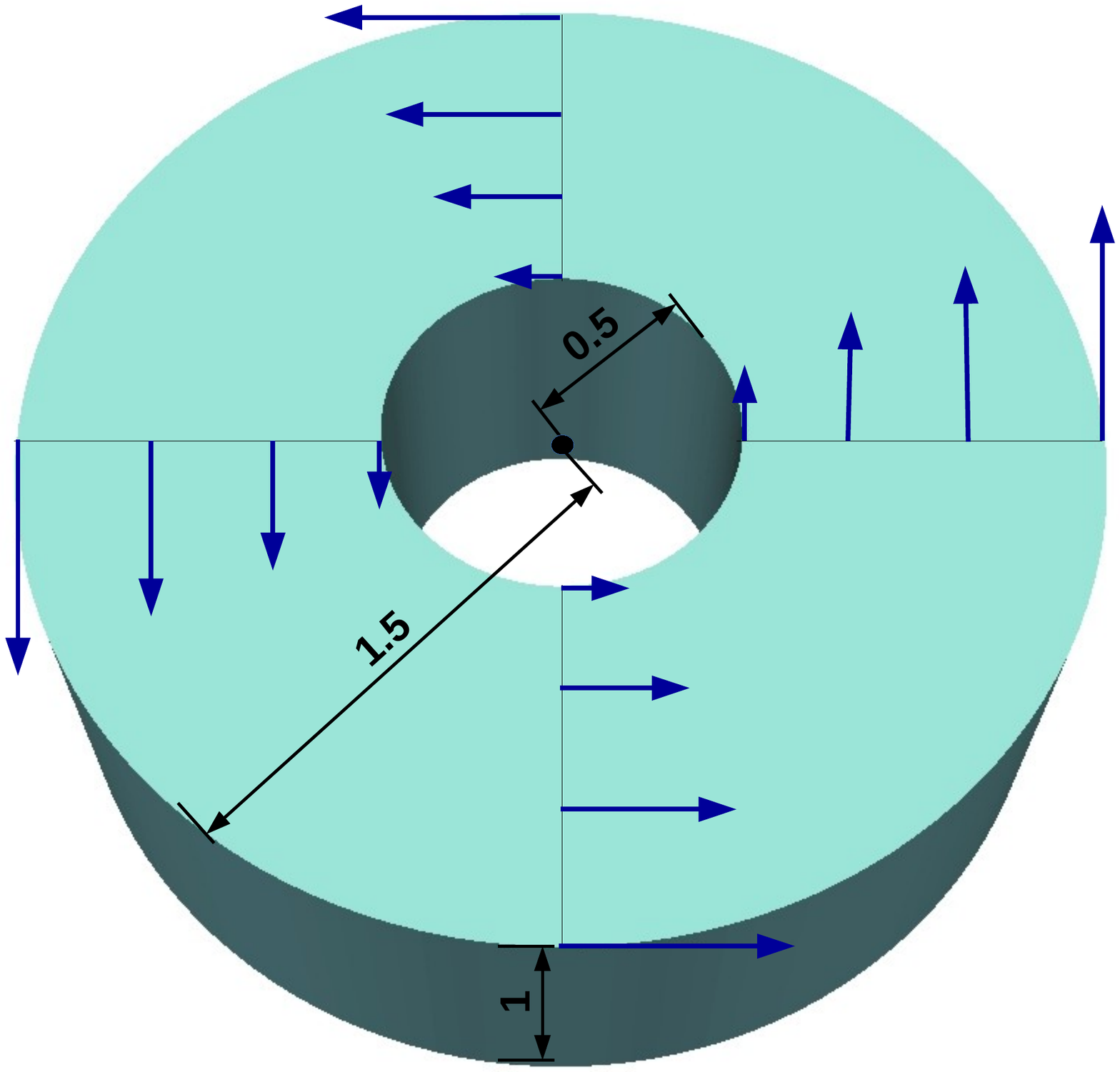} &
\includegraphics[angle=0, trim=400 80 360 280, clip=true, scale = 0.26]{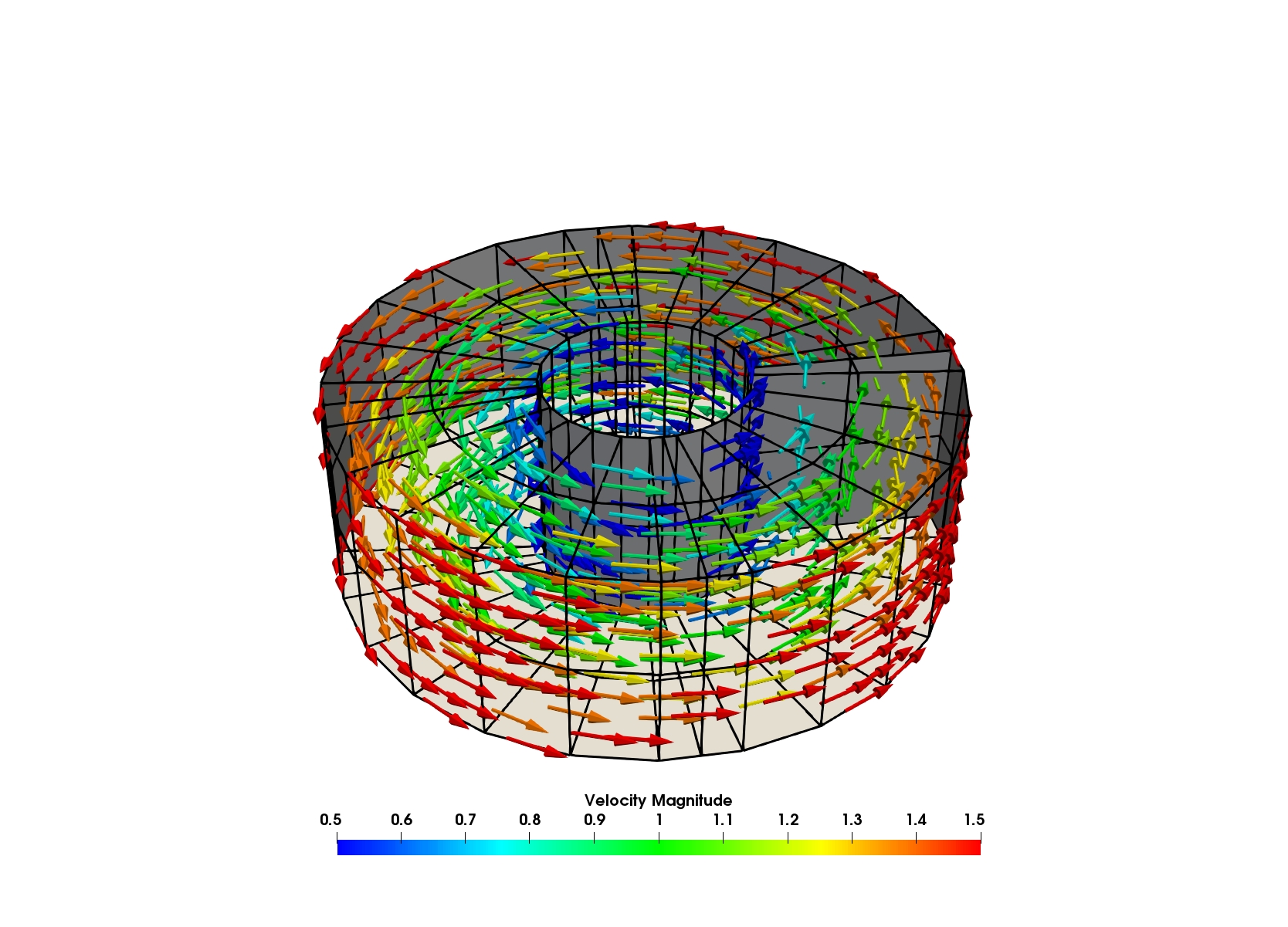}\\
(a) & (b)
\end{tabular}
\caption{The three-dimensional spinning annular disk: (a) the geometrical setting and the initial condition; (b) a snapshot of the velocity field.} 
\label{fig:annular_setting_and_snapshot}
\end{center}
\end{figure}

\subsection{A spinning annular disk}
In this example, we study a spinning annular disk with zero traction boundary condition imposed on all boundary faces. In Figure \ref{fig:annular_setting_and_snapshot} (a), the geometrical setting is illustrated. The inner radius of the disk is $0.5$ m, the outer radius is $1.5$ m, and the thickness is $1$ m. The material of the disk is Neo-Hookean with density $\rho_0 = 10$ kg/m$^3$ and shear modulus $c_1 = 7.5$ Pa. Both the geometrical and material settings follow the benchmark example in \cite{Krueger2016}. The initial displacement is zero, and the spinning motion is initiated by an initial angular velocity of $1$ rad/s in the x-y plane, that is
\begin{align*}
\bm V(\bm X, 0) = \left( - V_0 \frac{Y}{L_0}, V_0 \frac{X}{L_0}, 0 \right)^T, \quad V_0 = 1 \textup{m}/\textup{s}.
\end{align*}
We choose the reference scales as $L_0 = 1$ m, $M_0 = 1$ kg, and $T_0 = 1$ s. The geometry of the domain can be exactly parametrized by quadratic NURBS, and hence we choose $\mathsf p =2$ for the discrete function spaces with $\mathsf a = 1$ and $\mathsf b = 0$. A coarse mesh is generated with 32 elements in the circumferential direction, 4 elements in the radial direction, and 4 elements in the axial direction; a fine mesh is generated with 64 elements in the circumferential direction, 8 elements in the radial direction, and 8 elements in the axial direction. The time step size is $\Delta t = 2\times 10^{-4}$ s, and the problem in integrated up to $T=10.0$ s. For the simulations, we use $\textup{tol}_{\textup{R}} = 10^{-8}$ and $\textup{tol}_{\textup{A}} = 10^{-8}$ as the stopping criteria. In Figure \ref{fig:annular_setting_and_snapshot} (b), a snapshot of the simulated velocity in the annular disk is depicted. Due to the zero traction boundary condition and the zero body force, this problem serves as a benchmark for examining the energy stability as well as the momentum conservation properties. In Figure \ref{fig:spinning_energy_and_linear_momentum} (a), we can see that the kinetic energy and the total energy are nicely conserved. In Figure \ref{fig:spinning_energy_and_linear_momentum} (b), the relative errors of the total energy over time are plotted, which are uniformly smaller than $3\times 10^{-6}$. The exact value of the linear momentum is zero, and we see that the absolute errors are less than $1.5\times 10^{-13}$ in Figure \ref{fig:spinning_energy_and_linear_momentum} (c). The x- and y-components of the angular momentum are zero, with numerical values having absolute errors less than $10^{-13}$ (Figure \ref{fig:spinning_angular_momentum}). The analytic value of the z-component of the angular momentum is $78.5$ kg$\cdot$m$^2$/s, and we depict its relative error from the simulation with the coarse mesh. Note that the error of the z-component of the angular momentum is highly oscillatory and is bounded by $8\times 10^{-9}$. The numerical results corroborate the estimates given in Section \ref{subsec:semi_discrete_formulation}.

\begin{figure}
\begin{center}
\begin{tabular}{c}
\includegraphics[angle=0, trim=85 85 160 80, clip=true, scale = 0.38]{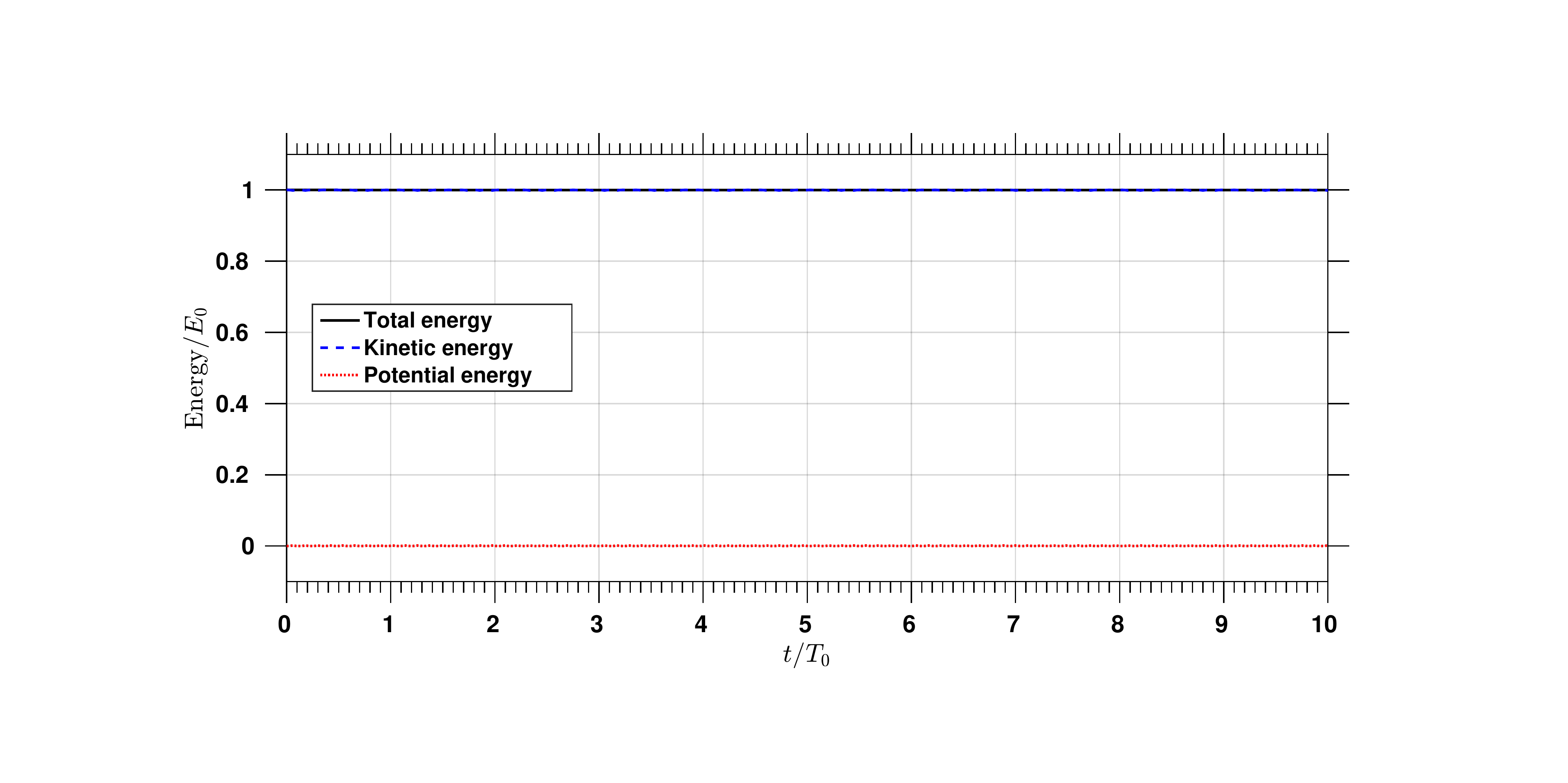} \\
(a)\\
\includegraphics[angle=0, trim=85 85 160 80, clip=true, scale = 0.38]{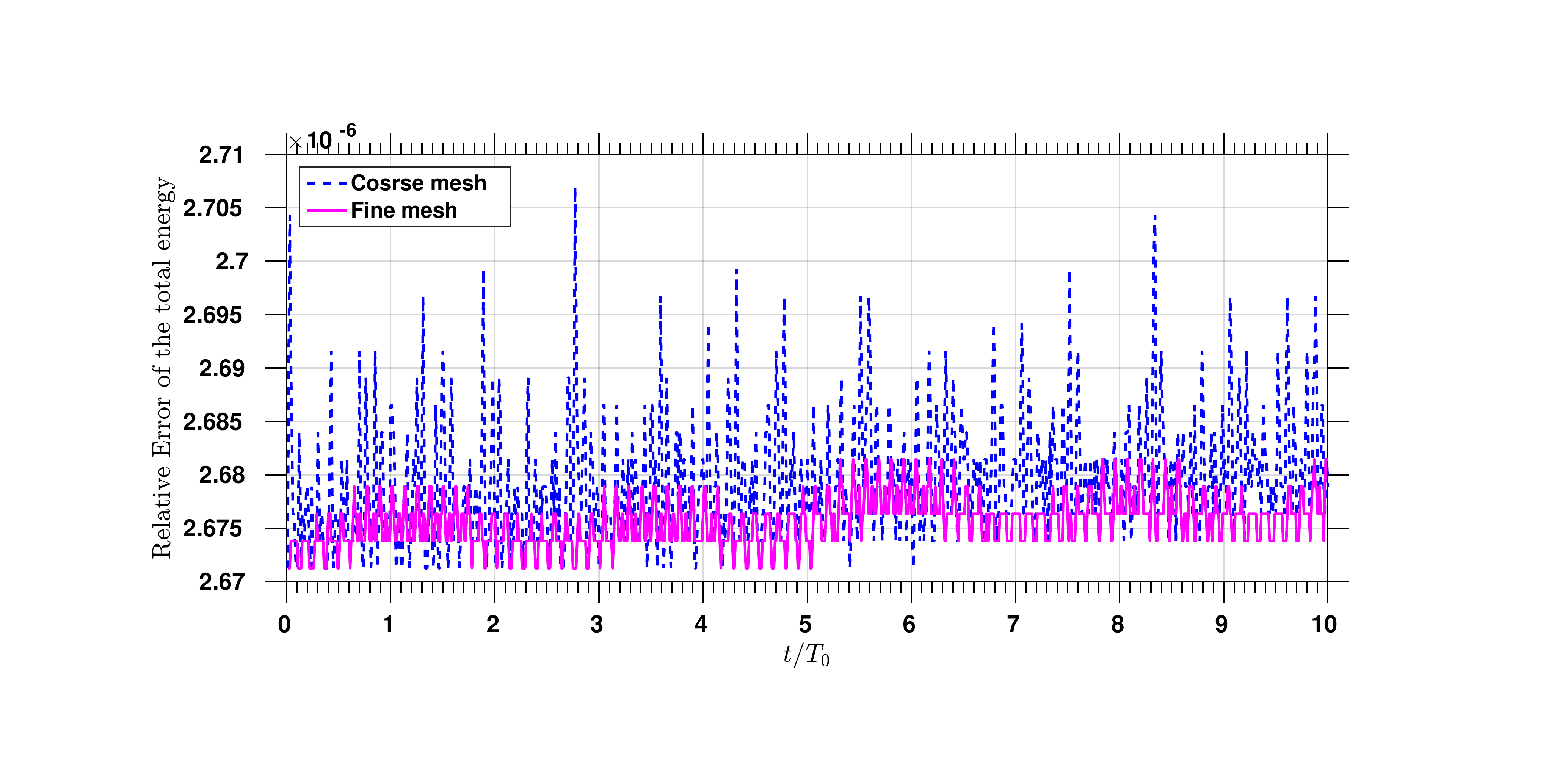} \\
(b)\\
\includegraphics[angle=0, trim=85 85 160 80, clip=true, scale = 0.38]{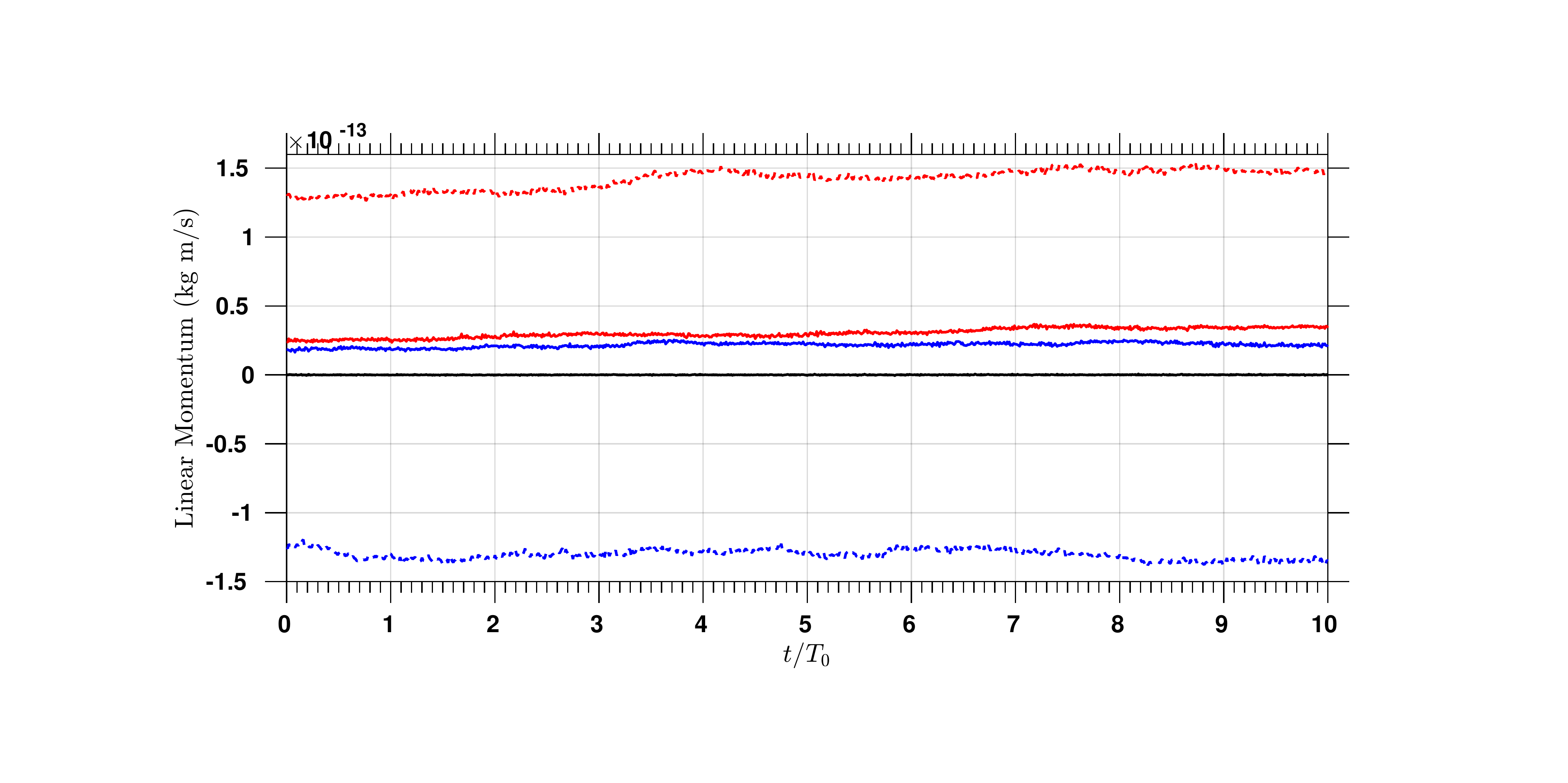} \\
(c)
\end{tabular}
\caption{(a) The total, kinetic, and potential energies over time using the coarse mesh scaled by $E_0 = 29.27$ J, which is the initial total energy; (b) The relative errors of the total energy over time for the two different meshes; (c) The x-, y-, and z-components of the linear momentum are plotted in the blue, red, and black colors respectively, and the results for the fine mesh and coarse mesh are plotted in solid and dashed lines respectively. } 
\label{fig:spinning_energy_and_linear_momentum}
\end{center}
\end{figure}

\begin{figure}
\begin{center}
\begin{tabular}{c}
\includegraphics[angle=0, trim=85 85 160 80, clip=true, scale = 0.38]{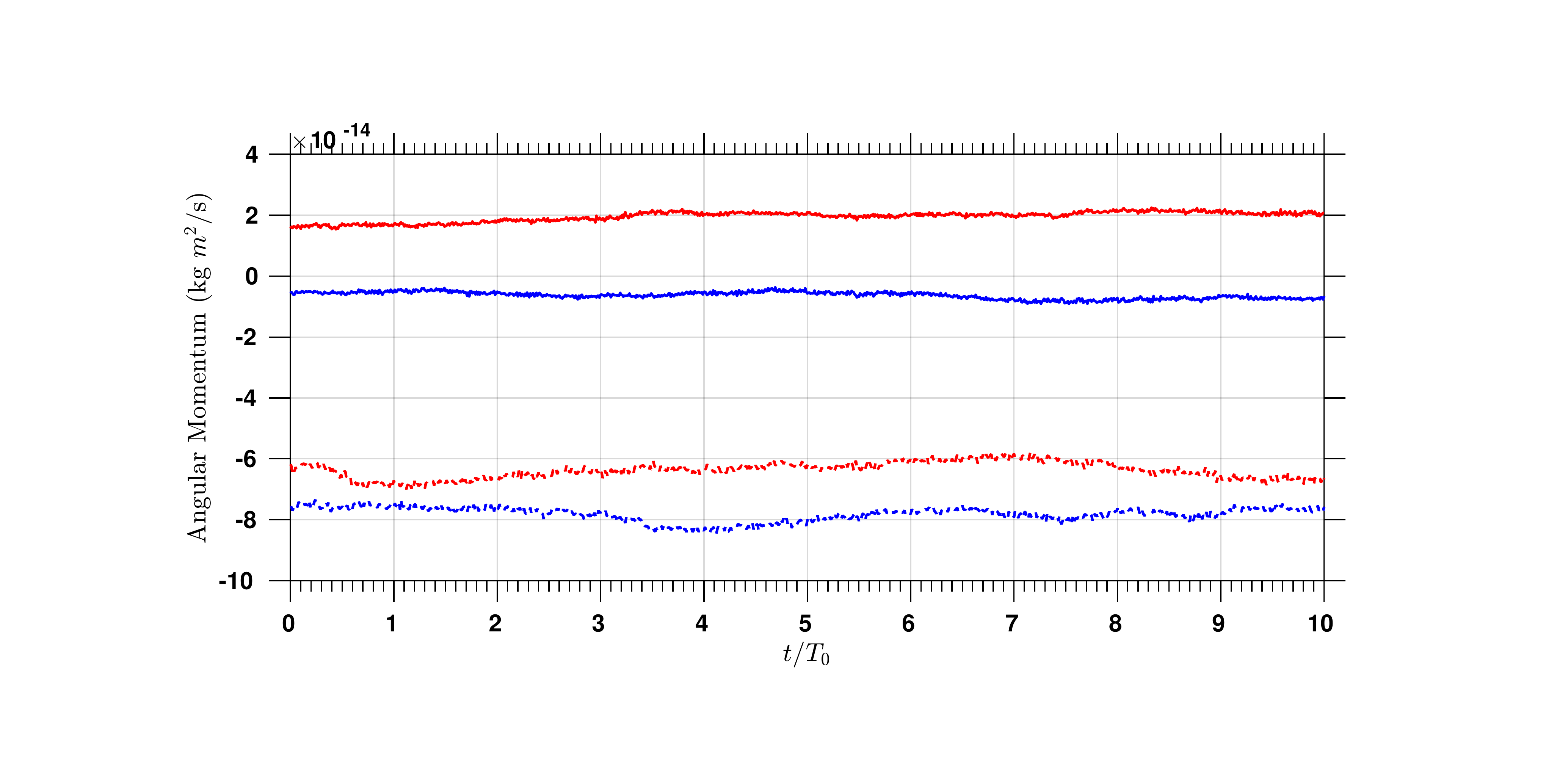} \\
(a)\\
\includegraphics[angle=0, trim=85 85 160 80, clip=true, scale = 0.38]{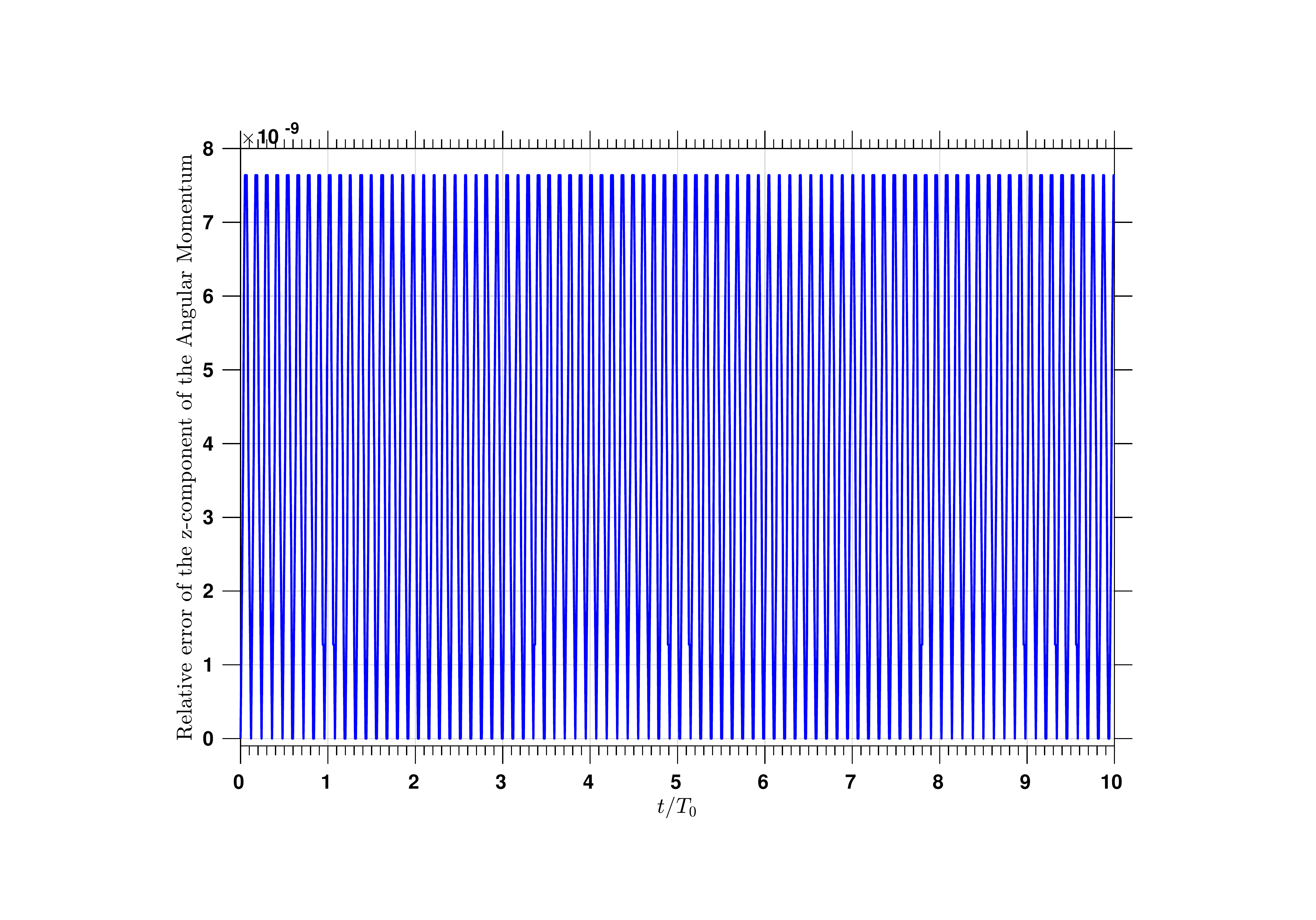} \\
(b)
\end{tabular}
\caption{(a) The x- and y-components of the angular momentum are plotted in the blue and red colors respectively, and the results for the fine mesh and coarse mesh are plotted in solid and dashed lines respectively; (b) The relative error of the z-component of the angular momentum over time for the coarse mesh. } 
\label{fig:spinning_angular_momentum}
\end{center}
\end{figure}

\section{Conclusions and future work}
\label{sec:conclusion}
In this work, we presented a new numerical formulation for incompressible hyper-elastodynamics. We have revealed that the proposed formulation possesses a physically compatible notion of numerical stability, and the inf-sup condition can be utilized to give a bound for the pressure. These properties favorably distinguish the proposed formulation from previously existing ones \cite{Hoffman2011,Idelsohn2008,Liu1986,Sussman1987}. We use smooth generalizations of the Taylor-Hood element based on NURBS for the spatial discretization, aiming to provide a higher-order method that is stable, robust, and implementationally convenient. The inf-sup stability for the elements is elucidated through numerical assessment. A variety of benchmark examples are simulated to investigate the effectiveness of the method in different loading conditions and for different material models. In particular, two dynamic problems are studied to verify the numerical stability and conservation properties.

In addition to the superior accuracy in stress calculations, the adoption of NURBS elements makes the description of material anisotropy convenient because the mesh naturally aligns along the axial, circumferential, and radial directions. These attributes make the proposed formulation a promising candidate for biomedical problems. Based on the proposed formulation, the anisotropic arterial wall model will be further refined with detailed stress-driven mass production and removal for individual constituents that comprise the tissue. This will lead to a three-dimensional patient-specific predictive tool for vascular growth and remodeling. 

\section*{Acknowledgements}
This work is supported by the National Institutes of Health (NIH) under the award numbers 1R01HL121754 and 1R01HL123689, the National Science Foundation (NSF) CAREER award OCI-1150184, and computational resources from the Extreme Science and Engineering Discovery Environment supported by the NSF grant ACI-1053575. The authors acknowledge TACC at the University of Texas at Austin for providing computing resources that have contributed to the research results reported within this paper.

\bibliographystyle{plain}  
\bibliography{mixed_solids}

\end{document}